\documentclass[12pt,reqno]{amsart}  %reqno  for equation numbers on the right
\pdfoutput=1
\usepackage{amsfonts,amssymb,amsthm,eucal,amsmath,amscd} 
\usepackage{pinlabel}
\usepackage{graphicx}
\usepackage[T1]{fontenc}
\usepackage{latexsym,url}
\usepackage{enumerate}
\usepackage[all,cmtip]{xy}
\usepackage{xcolor}
\usepackage{subfig}
\usepackage{listings}
\usepackage{subfig}

\usepackage[inner=40mm, outer=40mm, textheight=225mm]{geometry}

\usepackage{url}
\usepackage[headings]{fullpage}

\usepackage[bookmarks=true,%
    colorlinks=true,%
    linkcolor=blue,%
    citecolor=blue,%
    filecolor=blue,%
    menucolor=blue,%
    urlcolor=blue,%
    breaklinks=true]{hyperref}

\newtheorem{thm}{Theorem}[section]

\newtheorem{lemma}[thm]{Lemma}
\newtheorem{prop}[thm]{Proposition}
\newtheorem{cor}[thm]{Corollary}
\theoremstyle{definition}
\newtheorem{defn}[thm]{Definition}

\newtheorem{ex}[thm]{Example}
\theoremstyle{remark}
\newtheorem{rmk}[thm]{Remark}

\def\H{\mathbb{H}}
\def\R{\mathbb{R}}

\def\E{\mathbb{E}}

\newcommand{\TT}{\mathbb{T}}
\def\tri{\mathcal{T}}

\def\kk{\boldsymbol{k}}

\def\BZ{\mathbb Z}

\def\D{\Delta}

\def\S{\Sigma}

\def\b{\beta}

\def\calE{\mathcal E}

\def\calP{\mathcal P}
\newcommand{\CT}{\mathcal{T}}
\def\be{  \begin{equation} }
\def\ee{  \end{equation} }
\def\ID{I_{\Delta}}

\newcommand{\T}{\mathcal{T}}
\def\Z{\mathbb Z}
\def\JD{J_{\Delta}}
\def\bd{\partial}
\def\C{\mathbb C}

\newcommand{\sect}{\textsection}

\def\BTwo{\mathcal{B}_2}
\def\BThree{\mathcal{B}_3}

\DeclareMathOperator{\Ker}{Ker}
\let\Im\relax
\DeclareMathOperator{\Im}{Im}
\DeclareMathOperator{\Hom}{Hom}

\def\quads{\square}  % quad types
\def\TT{\mathbb T}   % span of tet solutions
\def\EE{\mathbb E}   % span of edge solutions

\usepackage{lineno}
\newcommand*\patchAmsMathEnvironmentForLineno[1]{
  \expandafter\let\csname old#1\expandafter\endcsname\csname #1\endcsname
  \expandafter\let\csname oldend#1\expandafter\endcsname\csname end#1\endcsname
  \renewenvironment{#1}
     {\linenomath\csname old#1\endcsname}
     {\csname oldend#1\endcsname\endlinenomath}} 
\newcommand*\patchBothAmsMathEnvironmentsForLineno[1]{
  \patchAmsMathEnvironmentForLineno{#1}
  \patchAmsMathEnvironmentForLineno{#1*}}
\AtBeginDocument{
\patchBothAmsMathEnvironmentsForLineno{equation}
\patchBothAmsMathEnvironmentsForLineno{align}
\patchBothAmsMathEnvironmentsForLineno{flalign}
\patchBothAmsMathEnvironmentsForLineno{alignat}
\patchBothAmsMathEnvironmentsForLineno{gather}
\patchBothAmsMathEnvironmentsForLineno{multline}
}
\begin{document}

%%%%%%%%%%%%%%%%%%%%%%{page1}

\title[The 3D-index and normal surfaces]{The 3D-index and normal surfaces}
\author{Stavros Garoufalidis}
\address{School of Mathematics\\
Georgia Institute of Technology\\
Atlanta, Georgia 30332-0160, USA}
\email{stavros@math.gatech.edu}

\author{Craig D. Hodgson}
\address{School of Mathematics and Statistics \\
         The University of Melbourne \\
         Parkville, VIC, 3010, AUSTRALIA
         %\tt{\url{http://www.ms.unimelb.edu.au/~cdh }}
         }
\email{craigdh@unimelb.edu.au}
\author{Neil R. Hoffman}
\address{School of Mathematics and Statistics \\
         The University of Melbourne \\
         Parkville, VIC, 3010, AUSTRALIA
         %\tt{\url{http://www.ms.unimelb.edu.au/Personnel/profile.php?PC_id=1424 }}
         }
\email{nhoffman@ms.unimelb.edu.au}

\author{J. Hyam Rubinstein}
\address{School of Mathematics and Statistics \\
         The University of Melbourne \\
         Parkville, VIC, 3010, AUSTRALIA
         %\tt{\url{http://www.ms.unimelb.edu.au/~rubin }}
         }
\email{rubin@ms.unimelb.edu.au}

\begin{abstract}
Dimofte, Gaiotto and Gukov introduced a powerful invariant, the 3D-index,
associated to a suitable ideal triangulation of a 3-manifold with torus
boundary components. The 3D-index is a collection of formal power series
in $q^{1/2}$ with integer coefficients. Our goal is to explain how the 
3D-index is a generating series of normal surfaces associated to the
ideal triangulation. This shows a connection of the 3D-index with
classical normal surface theory, and fulfills a dream of constructing 
topological invariants of 3-manifolds using normal surfaces.
\end{abstract}

\maketitle
\tableofcontents

%%%%%%%%%%%%%%%%%%%%%%%%%%%%%%%%%%%%%%%%%%%%%%%%%%%%%%%%%%%%%%%%%%%%%%%%%%%
%%%%%%%%%%%%%%%%%%%%%%%%%%%%%%%%%%%%%%%%%%%%%%%%%%%%%%%%%%%%%%%%%%%%%%%%%%%

\section{Introduction}  

Recently, the physicists Dimofte, Gaiotto and Gukov \cite{DGG1,DGG2}
introduced a powerful new invariant for a compact orientable 3-manifold $M$ 
with non-empty boundary $\bd M$ consisting of tori, called the 
\emph{3D-index}. This invariant arises from a gauge theory with $N=2$ 
supersymmetry under a low energy limit, and seems to contain a great deal 
of information about the geometry and topology of the manifold. The 
3D-index is a collection of $q$-series, i.e. formal Laurent series in 
$q^{1/2}$, defined as an infinite sum over integer weights attached to 
the edges of a chosen ideal triangulation $\CT$ of $M$. The $q$-series 
$I_\CT(b)$ are parametrised by a choice of peripheral homology class 
$b \in H_1(\bd M;\Z)$.

 Our goal is to explain how the 3D-index can be viewed as a generating
series of normal surfaces on $\CT$; see Definition \ref{def.Iab} and
Corollary \ref{cor.Izero} below. Normal surfaces depend heavily
on the ideal triangulation whereas the 3D-index should not. Our results
show a connection of the 3D-index with classical normal surface theory, 
and fulfill a `folk dream' of constructing topological invariants of 
3-manifolds using normal surfaces. 

In this paper we first recall how the 3D-index of an ideal triangulation 
is defined, following \cite{GHRS}, and discuss some of its key properties. 
Throughout the paper, by a triangulation of a compact 3-manifold we mean an 
{\em ideal triangulations} in the sense of Thurston \cite{Th} 
(see also \cite{NZ,N}). 
Physics predicts that 
the 3D-index should give a {\em topological invariant} of the underlying 
manifold $M$, but this is not known in general. In fact, the sum defining 
the 3D-index need not even converge (as a formal Laurent series) for all 
ideal triangulations $\tri$.  But we can characterise the good 
triangulations using the {\em normal surface theory} developed by Haken 
\cite{Ha1,Ha2}. (Haken's theory applies to ideal triangulations, when 
restricting to closed normal surfaces.) 
In fact, the work of \cite{GHRS} shows the index sum 
for $I_\CT(b)$ converges for all $b \in H_1(\bd M;\Z)$ if and only
if the triangulation is {\em $1$-efficient}, i.e. contains no embedded 
closed normal surfaces $S$ of Euler characteristic $\chi(S) \ge 0$ 
except peripheral tori. 

By work of Matveev and Piergallini, (see \cite{Ma1, Ma2, Pi, Petronio}), 
any two ideal triangulations $\CT, \CT'$ (with at least 2 tetrahedra) 
of a given compact 3-manifold $M$ with non-empty boundary can be connected 
by a sequence of 2-3 and 3-2 Pachner moves. 
Similarly, we can also consider 0-2 and 2-0 moves on a triangulation as 
shown in Figure \ref{2-3 and 0-2 moves} below.

The work of \cite{Gar,GHRS} shows that the index is invariant under 
2-3/3-2 and 0-2/2-0 moves provided all ideal triangulations involved are 
1-efficient. However it is not currently known whether any two 1-efficient 
triangulations of a given manifold can be connected by 2-3/3-2 and 
0-2/2-0 moves preserving 1-efficiency. 

We then give a new formulation of the definition of index as a sum over 
certain \emph{singular normal surfaces}. The minimum degree of each term 
in this sum then has a simple geometric interpretation involving the 
\emph{Euler characteristic} and the number of \emph{double arcs} of the 
corresponding surface. This leads to new, more direct proofs of results on
the convergence of the series for the 3D-index. We also introduce a class 
of embedded \emph{generalised normal surfaces}, and show how to express the 
index as a sum over terms corresponding to these embedded surfaces.

Next we show that the work of Neumann \cite{N} on combinatorics of ideal 
triangulations can be reinterpreted to give a precise description of the 
set $Q(\T;\Z)$ of all {\em integer solutions} to the $Q$-matching equations 
of Tollefson \cite{To}. Each such solution $S$ represents a 
(possibly singular) \emph{spun normal surface} \cite{KR1,Ti} with 
well-defined homology class $[S]_2  \in H_2(M,\bd M;\Z/2\Z)$ and boundary 
$[\bd S] \in H_1(\bd M;\Z)$ such that the mod $2$ reduction of $[\bd S]$ is
the image of $[S]$ under the boundary map 
$\bd_* : H_1(M;\Z/2\Z) \to H_1(\bd M;\Z/2\Z)$. Further, each pair 
$(a,b) \in H_2(M,\bd M;\Z/2\Z) \times  H_1(\bd M;\Z)$ such that 
$\bd_* a = b \bmod 2$ arises in this way, and the $Q$-normal classes $S$ with 
$[S]_2=0$ and $[\bd S]=0$ are the integer linear combinations of the 
``edge solutions'' and the ``tetrahedral solutions'' constructed by 
Kang-Rubinstein \cite{KR1}.

We then show that the definition of 3D-index can be extended to give a 
$q$-series $I^a_\CT(b)$ for each $(a,b)$ as above; this is a sum over 
$Q$-normal classes $S$ with $[S]=a$ and $[\bd S]=b$ modulo tetrahedral 
solutions. The previous definition of 3D-index only applies to the cases 
where $a=0$ and $b \in H_1(\bd M; 2 \Z)$; we have $I^0_\CT(b) = I_\CT(2b)$ 
in the notation of \cite{GHRS}. Here $I^0_\T(b)$ is also defined for 
$b \in  {\mathcal K} = \Ker\left(H_1(\bd M; \Z) \to H_1(M;\Z/2\Z)\right)$.

We then give some computations of the 3D-index, including an example 
showing that that the series for the 3D-index $I^a_\CT(b)$ can sometimes 
converge for $b \ne 0$ even when the triangulation $\T$ is not 1-efficient 
(so the series for $I_\CT^0(0)$ does not converge).

We conclude with some discussion of experimental results on the 
connectedness of the set of 1-efficient ideal triangulations of a given 
manifold under 2-3/3 and 0-2/2-0 Pachner moves. (Recall that such moves
do not change the 3D-index.) In particular, we find examples of 1-efficient 
triangulations of the solid torus with six tetrahedra that cannot be 
connected by sequences of 2-3/3-2 moves through 1-efficient triangulations. 
These triangulations can, however, be connected via 1-efficient 
triangulations if 0-2 and 2-0 moves are also allowed; hence they have the 
same 3D-index.

Finally, the appendix gives some results on generalised angle structures 
and Euler characteristic for Q-normal surfaces; these extend well-known 
results of Luo-Tillmann \cite{LT} for closed normal surfaces. 

\subsection*{Acknowledgements} 
The first author is partially supported by National Science Foundation 
Grant DMS-14-06419. The last three authors are partially supported by 
Australian Research Council Discovery Grant DP130103694. The authors 
wish to thank Tudor Dimofte and Henry Segerman for useful discussions 
related to this paper.

%%%%%%%%%%%%%%%%%%%%%%%%%%%%%%%%%%%%%%%%%%%%%%%%%%%%%%%%%%%%%%%%%%%%%%%%%%%
%%%%%%%%%%%%%%%%%%%%%%%%%%%%%%%%%%%%%%%%%%%%%%%%%%%%%%%%%%%%%%%%%%%%%%%%%%%

\section{The tetrahedral index}

The index for a triangulation is built up from the \emph{tetrahedron index}  
$\ID \colon \mathbb{Z}^2\rightarrow \mathbb{Z}[[q^\frac{1}{2}]]$, defined 
for $m,e \in \Z$ by
\be
\ID(m,e)(q)=\sum^{\infty}_{n=\max\{0,-e\}}(-1)^n 
\frac{q^{\frac{1}{2}n(n+1)-(n+\frac{1}{2}e)m}}{(q)_n(q)_{n+e}}
\label{ID_def}
\ee
where $(q)_n=\Pi_{i=1}^{n}(1-q^i)$ is the $q$-Pochhammer symbol 
(by convention $(q)_0=1$). The index coincides with the coefficients 
of $z^e$ in the generating function
\be
I(m,q,z)=\sum_{e \in \Z}  \ID(m,e)(q) z^e 
= \frac{(q^{-\frac{m}{2}+1} 
\, z^{-1};q)_\infty}{(q^{-\frac{m}{2}} \,z;q)_\infty}
\label{ID_gen_func}
\ee
where
\be
(z; q)_\infty = \prod_{n=0}^\infty (1- z q^n) .
\label{zq-Poch}
\ee
(see \cite{Gar}). It follows that for any fixed $q \in \C$ with $0 < |q|<1$,
 the product for $I(m,q,z)$ converges for all $z\in \C$ with $z\ne 0$, 
and defines a holomorphic function of $z$ in the annulus $0<|z|<1$
with Laurent series given  by $\sum_{e \in \Z}  \ID(m,e)(q) z^e$.

\medskip
A more symmetric version of the tetrahedral index was defined in \cite{GHRS}  
 for $a,b,c \in \Z$ by
\be
\JD(a,b,c)= (-q^{\frac{1}{2}})^{-b} \ID(b-c,a-b)
= (-q^{\frac{1}{2}})^{-c} \ID(c-a,b-c) = (-q^{\frac{1}{2}})^{-a} \ID(a-b,c-a).
\ee
Then $\JD(a,b,c)$ is \emph{invariant under all permutations of its 
arguments} $a,b,c$ and satisfies
\be
\JD(a,b,c) = (-q^{\frac{1}{2}})^{s} \JD(a+s,b+s,c+s)  
\text{ for all } s \in \Z.
\label{J_plus_tets}
\ee

In particular, if $a, b\ge 0$ then (\ref{ID_def}) gives 
\be
\JD(a,b,0) =\ID(-a,b) = q^{\frac{1}{2}ab} + \text{ higher order terms},
\label{min_j_deg}
\ee
and (\ref{J_plus_tets}) shows that, in general, the lowest degree term 
in $\JD(a,b,c)$ has $q$-degree 
\be
\deg \JD(a,b,c) = 
\frac{1}{2}\left( (a-m)(b-m) +(b-m)(c-m)+(c-m)(a-m) - m \right) 
\ee
and coefficient $(-1)^m$,  where $m=\min\{a,b,c\}$.

The tetrahedral index also satisfies many other interesting algebraic 
identities including the \emph{quadratic identity} 
\be
\label{eq.IDquadratic}
\sum_{e \in \BZ} \ID(m,e)\ID(m,e+c)q^e =\delta_{c,0},
\ee
the \emph{pentagon identity}
\begin{align}
\sum_{e_0 \in \BZ}
& q^{e_0} \ID(m_1,x_1+e_0)\ID(m_2,x_2+e_0)\ID(m_1+m_2,x_3+e_0)
\label{pent_eq} \\
&=q^{-x_3} \ID(m_1-x_2+x_3,x_1-x_3)\ID(m_2-x_1+x_3,x_2-x_3) ,\nonumber
\end{align}
and some important {\em recurrence relations}. 
(See \cite{Gar} for the details.)

%%%%%%%%%%%%%%%%%%%%%%%%%%%%%%%%%%%%%%%%%%%%%%%%%%%%%%%%%%%%%%%%%%%%%%%%%%%
%%%%%%%%%%%%%%%%%%%%%%%%%%%%%%%%%%%%%%%%%%%%%%%%%%%%%%%%%%%%%%%%%%%%%%%%%%%

\section{Notation and conventions for triangulated 3-manifolds}

In this paper, we will observe the following conventions and notation. 

Let $M$ be a compact orientable 3-manifold with boundary consisting of 
$r$ tori. Let $\tri$ be an ideal triangulation of $\text{int} M$ with 
$n$ ideal tetrahedra $\sigma_j$, (i.e. tetrahedra with their vertices 
removed). Then there are also $n$ edge classes, since $\chi(M)=0$.  

\subsection{Gluing equations} 
\label{gluing_eqn_sect}

We now briefly describe how to encode the combinatorics of a triangulation 
contained in Thurston's gluing equations (compare  \cite[Chapter 4]{Th}), 
while also including notation for normal disks in a tetrahedron. 
First, in each tetrahedron there are three pairs of opposite edges, which 
are disjoint from the three quadrilateral disks from normal surface theory. 
We label these quadrilaterals by the pair of edges they face, as in 
Regina \cite{regina}; Figure \ref{fig:edgesAndQuad} shows the quads
$q_A=q_{A:01:23}$, $q_A'=q_{A:02:13}$, $q_A''=q_{A:03:12}$ in a tetrahedron 
labelled $A$. As our discussion unfolds the quadrilaterals will be the 
focus, however we also provide a labelling for the  triangular disks for 
completeness: $t_{A:j}$ refers the the triangle in tetrahedron $A$ that 
cuts off vertex $j$.

\begin{figure}[htpb]
%\resizebox{0.25\textwidth}{!}{\input{figures/qa0123.pdf_tex}}\hspace{0.5cm}
%\resizebox{0.25\textwidth}{!}{\input{figures/qa0213.pdf_tex}}\hspace{0.5cm}
%\resizebox{0.25\textwidth}{!}{\input{figures/qa0312.pdf_tex}}\hspace{0.5cm}
\includegraphics[width=0.25\textwidth]{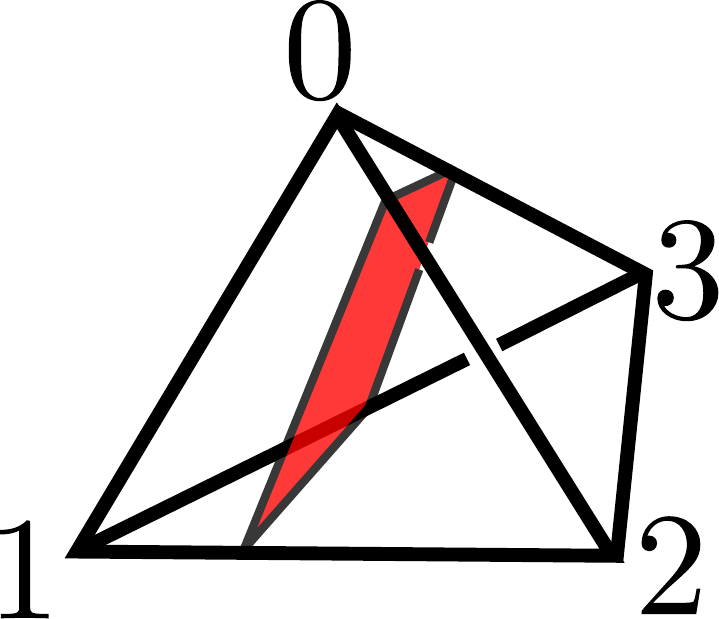}\hspace{0.5cm}
\includegraphics[width=0.25\textwidth]{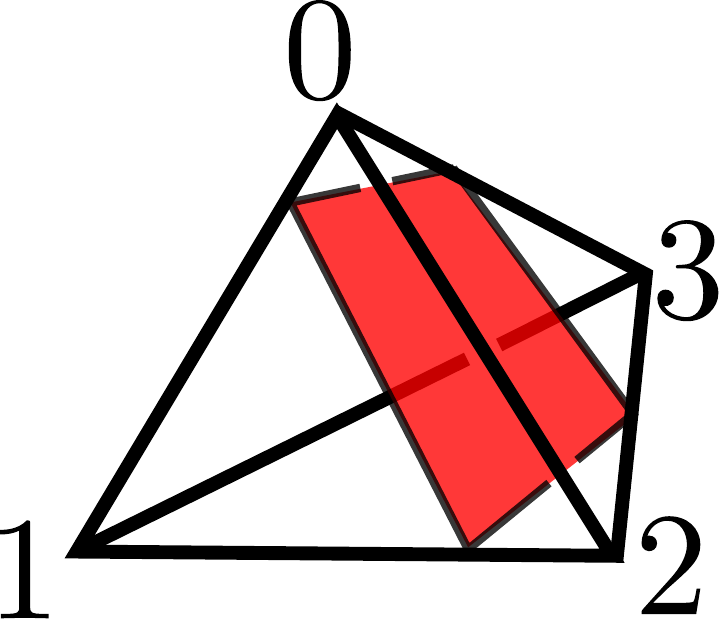}\hspace{0.5cm}
\includegraphics[width=0.25\textwidth]{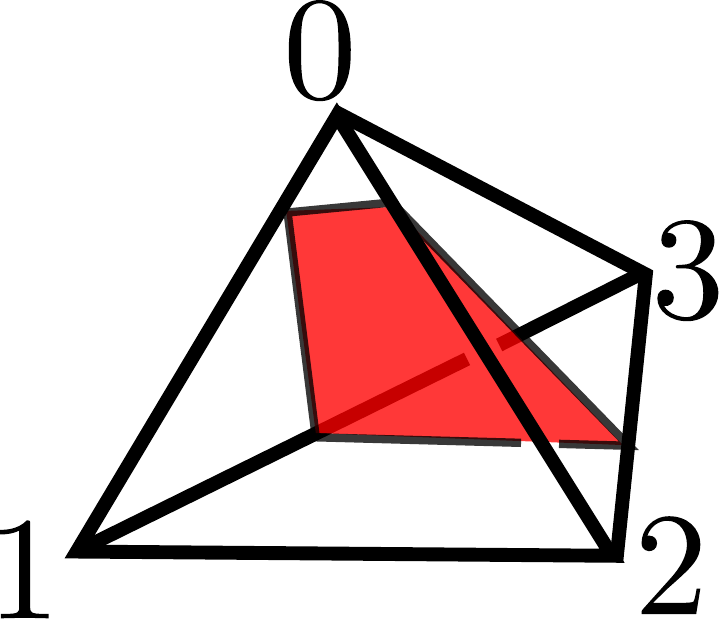}
\label{fig:qa0312}
\caption{The labelling of the 3 quadrilateral disk types of tetrahedron $A$. 
Left: $q_A=q_{A:01:23}$, centre: $q_A'=q_{A:02:13}$, right: $q_A''=q_{A:03:12}$.}
\label{fig:edgesAndQuad}
\end{figure}

Central to the construction of Thurston's gluing equations is the 
observation that edges in each tetrahedron form equivalence classes in 
the triangulation. In Thurston's notes, a set of logarithmic equations 
is associated to the $n$ edge classes of the triangulation. First, we 
associate a complex tetrahedral parameter $z_j$ to the $01,23$ edges in 
tetrahedron $j$, $z_j'=\frac{1}{1-z_j}$ to the $02,13$ edges and 
$z_j''=\frac{z_j-1}{z_j}$ to the $03,12$ edges. 

Then we can build the following {\em edge equations}:
\be
\label{log_edge_eqns}
\sum_j \left(a_{i,j} \log(z_j )+ b_{i,j} \log(z_j' )
+ c_{i,j} \log(z_j'')\right)=0+2\pi \sqrt{-1} 
\mbox{ for } i\in\{1,\ldots, n\},
\ee
where $a_{i,j},b_{i,j},c_{i,j}$ record the number of times the edges 
opposite quads $q_j, q_j', q_j''$ of tetrahedron $j$ appear in the $i$th 
edge class. (Our convention is that 
$\arg(z_j), \arg(z_j'), \arg(z_j'')  \in (-\pi,\pi]$.)

Note that the tetrahedral parameters also satisfy the relation 
\be
\label{log_tet_eqns}
\log(z_j )+\log(z_j') + \log(z_j'')= \pi \sqrt{-1}.
\ee

We can also build $2r$ {\em cusp equations} using the same tetrahedral 
parameters: 
\be
\label{log_cusp_eqns}
\sum_j \left(a_{i,j} \log(z_j )
+ b_{i,j} \log(z_j' )+ c_{i,j} \log(z_j'')\right)
=0+\pi \sqrt{-1} \mbox{ for } i\in\{n+1,\ldots, n+2r\},
\ee
however here $a_{i,j},b_{i,j},c_{i,j}$ record the number of times a 
peripheral curve winds anti-clockwise (or minus the number of clockwise 
times) around the edges opposite the quads  $q_j, q_j', q_j''$ in 
tetrahedron $j$. 

The coefficients in the edge and cusp equations define a 
$(n+2r) \times 3n$ \emph{gluing matrix} $[ a_{i,j} ~ b_{i,j} ~ c_{i,j} ]$, 
which is also given by the {\tt gluing\_equations()} function in 
SnapPy \cite{SnapPy}.

Throughout the paper, we will write $E_i$ for the vector of coefficients 
in the $i$th edge equation, and we will also label the cusp equation 
coefficients coming from the $k$th cusp by $M_k$ and $L_k$ and the 
corresponding curves $\mu_k$ and $\lambda_k$ where it proves convenient.

\subsection{Generalised angle structures} 
\label{angle_struct_sect}

Considering the imaginary parts of above gluing equations allows us to 
define an angle structure on the triangulation. If we define $\quads$ 
to be the set of quad types in a triangulation $\T$, then a 
\emph{generalised angle structure} on $\T$ is a function 
$\alpha : \quads \to \R$ which assigns an ``angle'' 
$\alpha(q) \in \R$ to the pair of edges opposite a quad $q$
and satisfies the set of equations:  
\be 
\label{ang_struct1}
\alpha(q_j) + \alpha(q_j')+ \alpha(q_j'') = 
\pi \mbox{ for } j\in\{1,\ldots, n\} \ee
and
\be 
\label{ang_struct2}
\sum_j \left(a_{i,j} \alpha(q_j)+ b_{i,j} \alpha(q_j')
+ c_{i,j} \alpha(q_j'')\right)=2\pi \mbox{ for } i\in\{1,\ldots, n\} 
\ee
where $a_{i,j},b_{i,j},c_{i,j}$ are defined above. A 
\emph{strict angle structure} satisfies the additional condition that 
$\alpha(q)>0$ for all quads $q$ and a \emph{semi angle structure} 
satisfies $\alpha(q)\ge 0$.

A further refinement of generalised angle structures which considers the 
holonomies of peripheral curves will be defined in Appendix 
\ref{gen_angle_euler}.\\

A central theme of this paper is that the entries of the gluing matrix 
can be interpreted in a number of different and interesting ways. 
By exploiting these relationships, we are able to better understand the 
3D-index of a triangulation, which will be defined in the next section.

%%%%%%%%%%%%%%%%%%%%%%%%%%%%%%%%%%%%%%%%%%%%%%%%%%%%%%%%%%%%%%%%%%%%%%%%%%%
%%%%%%%%%%%%%%%%%%%%%%%%%%%%%%%%%%%%%%%%%%%%%%%%%%%%%%%%%%%%%%%%%%%%%%%%%%%

\section{The 3D-index of an ideal triangulation}

Using the notation from the previous section, we now give a definition of 
the 3D-index of an ideal triangulation $\tri$ as formulated 
in \cite[section 4.7]{GHRS}. 

Given $\kk =(k_1, \ldots, k_n)\in \Z^n$, we can assign an integer weight 
$k_i$ to the $i$th edge class. Effectively, this also gives a weight on 
each edge of each tetrahedron $\sigma_j$ in $\CT$.
Let $a_j(\kk), b_j(\kk), c_j(\kk)$ be the sums of weights assigned to 
the pair of edges in  tetrahedron $j$ opposite quads $q_j, q_j',q_j''$ 
respectively. Note that these coefficients are precisely the entries in 
the linear combination $\sum_{i=1}^n k_i E_i$ of the rows of the gluing matrix 
corresponding to the edge equations. 

Now we define a tetrahedral index associated to the tetrahedron $\sigma_j$ by
$$
J(\sigma_j;\kk)=\JD(a_j(\kk), b_j(\kk) , c_j(\kk)).
$$
and the \emph{3D-index}  $I_{\CT}({\bf 0})$ of the triangulation $\CT$ is 
defined as 
\be
I_{\CT}({\bf 0})(q)  = \sum_{\kk \in \Z^{n-r}} 
 q^{\sum_i k_i}   \prod_{j=1} ^n
J(\sigma_j;\kk),
\ee
where the summation is over a (suitable) sublattice  
$\Z^{n-r} \subset \Z^{n}$, obtained by setting $r$ of the edge weights 
$k_i$ equal to zero. For $r=1$ we can set any $k_i=0$; but  in general the 
$r$ edge weights set to zero must be chosen carefully (see 
\cite[sections 4.5, 4.6]{GHRS}).

More generally, the 3D-index gives a function 
$I_{\CT} : H_1(\bd M; \Z) \cong \Z^{2r}\to \Z((q^\frac{1}{2}))$ defined 
as follows. Let $\gamma$ be an oriented multicurve on $\bd M$ with no 
contractible components, such that $\gamma$  is in {\em normal position} 
with respect to the induced triangulation $\tri_\bd$ of $\bd M$. 
Then let $\gamma_h$ be the component of $\gamma$ in cusp $h$ and express 
$\gamma_h = p_h \mu_h + q_h \lambda_h$, where $\mu_h$ corresponds to the 
row $M_h$ in the gluing matrix and $\lambda_h$ corresponds to $L_h$.
Now put
$$
J(\sigma_j;\kk,\gamma)=\JD(a_j(\kk,\gamma), b_j(\kk,\gamma) , c_j(\kk,\gamma))
$$
where the coefficients  
$a_j(\kk,\gamma), b_j(\kk,\gamma), c_j(\kk,\gamma)$ are precisely the 
entries in the linear combination 
$\sum_{i =1}^nk_i E_i  +\sum_{h=1}^n (p_h M_h + q_h L_h)$ corresponding 
to tetrahedron $\sigma_j$.

Then the general \emph{3D-index} of the triangulation $\CT$ is defined as 
\be
I_{\CT}([\gamma])(q)  = \sum_{\kk \in \Z^{n-r}} q^{\sum_i k_i}   \prod_j 
J(\sigma_j;\kk,\gamma),
\label{JD_index_eq}
\ee
where $[\gamma] \in H_1(\bd M; \Z)$ is the homology class of $\gamma$ and
the summation is over $\Z^{n-r} \subset \Z^{n}$ as above. (It is shown 
in \cite{GHRS} that the sum only depends on the homology class of $\gamma$.)

\begin{ex}
\label{ex:fig8}
For the figure eight knot complement $M$, with the ideal triangulation 
given by Thurston \cite[\sect4]{Th}, we have the following 
induced triangulation on $\bd M$ (as viewed from the cusp).

\begin{center}
\begin{figure}[htpb]
\includegraphics[width=0.8\textwidth]{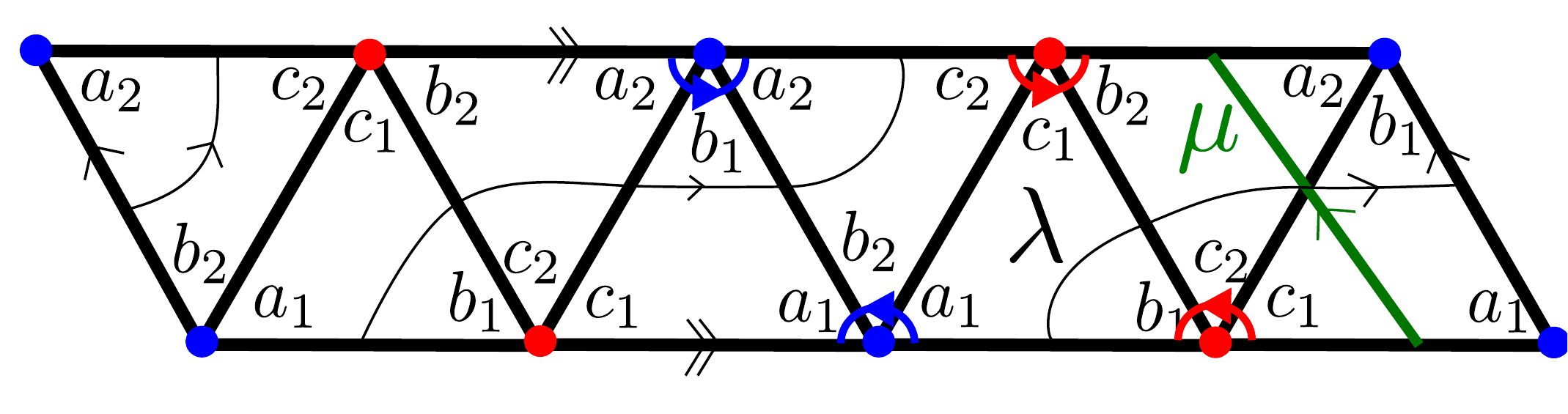}
\caption{\label{fig:figure8cusp} A labelled triangulation of the cusp 
of the figure 8 knot complement.}
\end{figure}
\end{center}

From this we can read off the following gluing data.

\begin{table}[htpb]
\begin{tabular}{|c||c|c|c|c|c|c|}
\hline
edge/peripheral curve
& ~~$a_1$~~ & ~~$b_1$~~ & ~~$c_1$~~ & ~~$a_2$~~ & ~~$b_2$~~ & ~~$c_2$ \\
\hline
\hline
edge 1 (blue) & 2 & 1 & 0 & 2 &1 & 0 \\
\hline
edge 2 (red) & 0& 1 & 2& 0 & 1 & 2 \\ 
\hline
meridian $\mu$ & 0& 0 & 1& -1 & 0 & 0 \\ 
\hline
longitude $\lambda$ & 0& 0 & 0& 2  & 0 & -2\\ 
\hline
\end{tabular}
\end{table}

Choosing generators $\mu,\lambda$ for $H_1(\bd M, \Z) = \Z^2$ corresponding 
to the standard meridian and longitude, and integer weights $k_1=k,k_2=0$ 
on the edges gives

\begin{align*} 
I_{\CT}(x \mu + y\lambda )
&=\sum_{k \in \Z}  q^k \JD(2k,k,x)\JD(2k-x+2y,k,-2y) \\
& = \sum_{k \in \Z}  \ID(k-x,k) \ID(k+2y,k-x+2y).
\end{align*}

For example, up to terms of order $q^{10}$, we have 

\begin{align*}
I_{\CT}({\bf 0} ) &=  
1-2 q-3 q^2+2 q^3+8 q^4+18 q^5+18 q^6+14 q^7-12 q^8-52 q^9-106 q^{10} + \ldots
\\
I_{\CT}(\mu) &= 
2 q-2 q^2+2 q^3+8 q^4+16 q^5+16 q^6+10 q^7-14 q^8-52 q^9-102 q^{10}+ \ldots
\\
I_{\CT}(2\mu) &= 
-q - q^2 + 3 q^3 + 6 q^4 + 12 q^5 + 9 q^6 + 3 q^7 - 19 q^8 - 50 q^9 -  88 q^{10}
+ \ldots \\
I_{\CT}(\lambda ) &=  
q^3+2 q^4+5 q^5+2 q^6-3 q^7-16 q^8-32 q^9-52 q^{10}+ \ldots \\
I_{\CT}(4\mu+\lambda ) &= 
q-q^4-2 q^5-5 q^6-8 q^7-10 q^8-11 q^9-6 q^{10}+ \ldots
\end{align*}
In this example, it is easy to check that 
$I_{\CT}(\pm x \mu  \pm y \lambda)=I_{\CT}(x \mu + y\lambda)$,
and we can also define the index for $x \in \Z, y \in \frac{1}{2} \Z$, 
as done in \cite{DGG1}. For example,

\begin{align*}
I_{\CT}(\frac{1}{2} \lambda) &= -2 q^{3/2}+4 q^{7/2}+10 q^{9/2}
+14 q^{11/2}+10 q^{13/2}-2 q^{15/2}-32 q^{17/2}-68   q^{19/2} + \ldots \\
I_{\CT}(\mu+\frac{1}{2} \lambda) &= -q-q^2+2 q^3+7 q^4+11 q^5+11 q^6
+3 q^7-17 q^8-49 q^9-88 q^{10} + \ldots  \\
I_{\CT}(2\mu+\frac{1}{2} \lambda) &=
-q^{1/2}+q^{5/2}+4 q^{7/2}+7 q^{9/2}+7 q^{11/2}+3 q^{13/2}-12 q^{15/2}
-31 q^{17/2}-62 q^{19/2}  + \ldots 
\end{align*}
\end{ex}

%%%%%%%%%%%%%%%%%%%%%%%%%%%%%%%%%%%%%%%%%%%%%%%%%%%%%%%%%%%%%%%%%%%%%%%%%%%
%%%%%%%%%%%%%%%%%%%%%%%%%%%%%%%%%%%%%%%%%%%%%%%%%%%%%%%%%%%%%%%%%%%%%%%%%%%

\section{Dependence on the choice of triangulation}

Physics predicts that the 3D-index as defined above should give a 
{\em topological invariant} of the underlying manifold $M$, but this is not 
known in general.  

The first difficulty is that the summation in (\ref{JD_index_eq}) need not 
even converge (as a formal power series) for all ideal triangulations $\tri$. 
But the good triangulations can be characterised using normal surface theory.

Given a triangulation $\tri$ of a 3-manifold $M$, an embedded surface 
$S \subset M$ is a {\em normal surface} if it intersects each tetrahedron 
in a finite collection of disjoint normal quadrilaterals (`quads') and 
triangles as shown in Figure \ref{fig:normal_disk_types}. In each 
tetrahedron there are 4 types of normal triangles and 3 types of normal quads.

\begin{center}
\begin{figure}[htpb]
\includegraphics[width=0.45\textwidth]{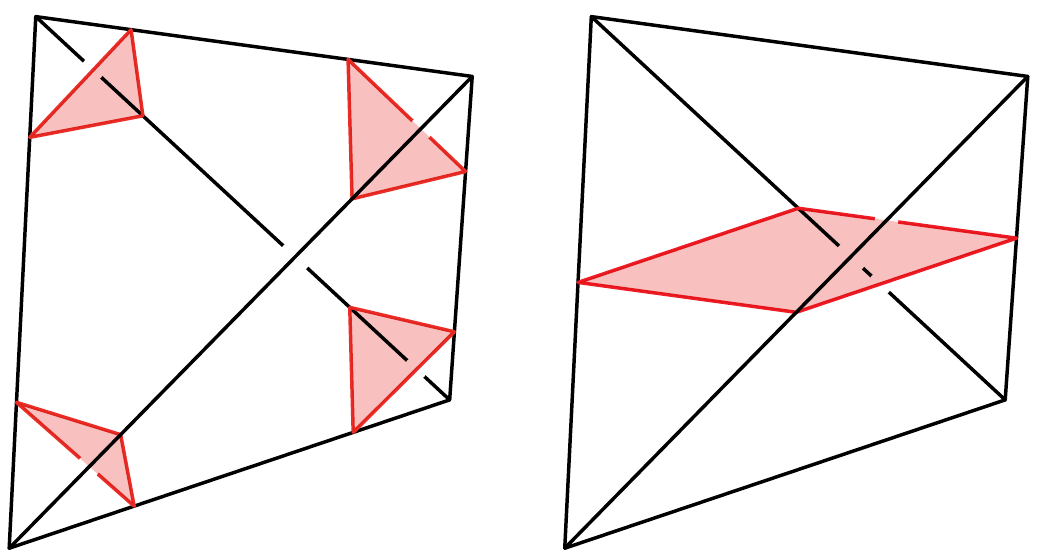}
\caption{\label{fig:normal_disk_types} The four types of triangular disks 
(left) and one of the three types of quadrilateral disks (right).}
\end{figure}
\end{center}

Embedded normal surfaces were introduced by Kneser \cite{Kn} and developed 
by Haken \cite{Ha1,Ha2} to create a normal surface theory which has played 
a key role in the development of algorithmic 3-manifold topology.  Haken 
observed that (not necessarily embedded) closed normal surfaces are 
parametrised by vectors of quad numbers and triangle numbers in $\Z^{7n}$ 
satisfying linear {\em matching equations}. Tollefson \cite{To} 
showed that the quad numbers are enough to determine a closed normal surface 
up to multiples of the boundary tori. These give vectors in $\Z^{3n}$ 
satisfying linear {\em$Q$-matching equations}.

\begin{defn}\cite{JR, KR2}
An ideal triangulation $\tri$ is \emph{1-efficient} if 
\begin{enumerate}
\item[(0)] it contains no embedded normal 2-spheres or projective planes, and
\item[(1)] the only embedded normal tori or Klein bottles are vertex-linking.  
\end{enumerate}
\end{defn}

\begin{thm}\cite[Theorem 1.2]{GHRS}
\label{index_structure_iff_1-efficient}
The sum defining %the 3D-index 
$I_{\CT}([\gamma])$ converges for all $[\gamma] \in H_1(\bd M; \Z)$ 
if and only if the sum $I_{\CT}({\bf 0})$ converges 
if and only if the ideal  triangulation $\CT$ is 1-efficient.
\end{thm}

We will outline a new, more direct proof of this result in Section 
\ref{sect:index_via_normal_surfaces} below, by rewriting the index as a 
sum of contributions from normal surfaces.

\medskip
It is shown in \cite{GHRS}  that 1-efficient triangulations exist for many 
important classes of cusped 3-manifolds including all hyperbolic manifolds 
and small Seifert fibre spaces.  But they cannot exist for 3-manifolds
containing (embedded) essential spheres, projective planes, tori or Klein 
bottles which are not boundary parallel. 

\medskip
It is known by the work of Matveev and Piergallini (see \cite{Ma1, Ma2, Pi}) 
that any two triangulations $\CT, \CT'$ (with at least 2 tetrahedra) of a 
given closed 3-manifold $M$ can be connected by a sequence of 2-3 and 3-2 
Pachner moves. We can also consider 0-2 and 2-0 moves on triangulations 
as shown in Figure \ref{2-3 and 0-2 moves}.

\begin{figure}[htpb]
\centering
\subfloat[The 2-3 and 3-2 moves.]{
\labellist
\small\hair 2pt
\pinlabel 2-3 at 139 84
\pinlabel 3-2 at 139 44
\endlabellist
\includegraphics[width=0.44\textwidth]{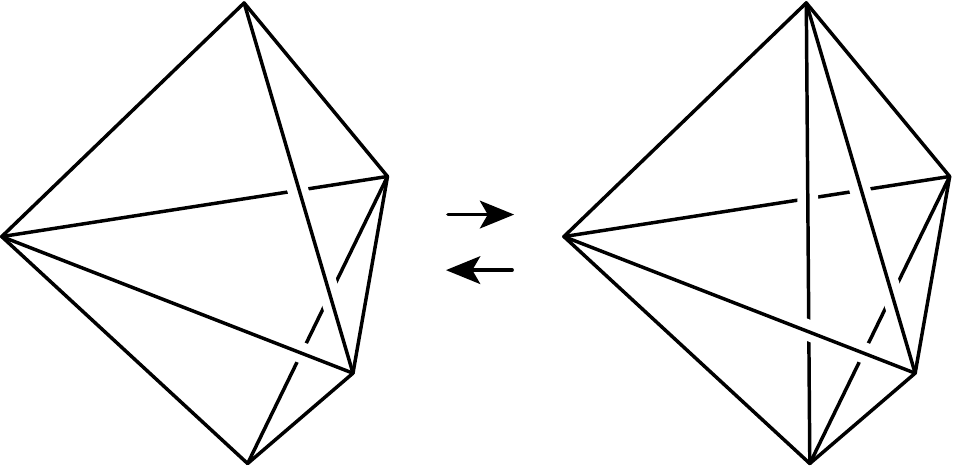}
\label{2-3}}
\qquad
\subfloat[The 0-2 and 2-0 moves.]{
\labellist
\small\hair 2pt
\pinlabel 0-2 at 109 84
\pinlabel 2-0 at 109 44
\endlabellist
\includegraphics[width=0.34\textwidth]{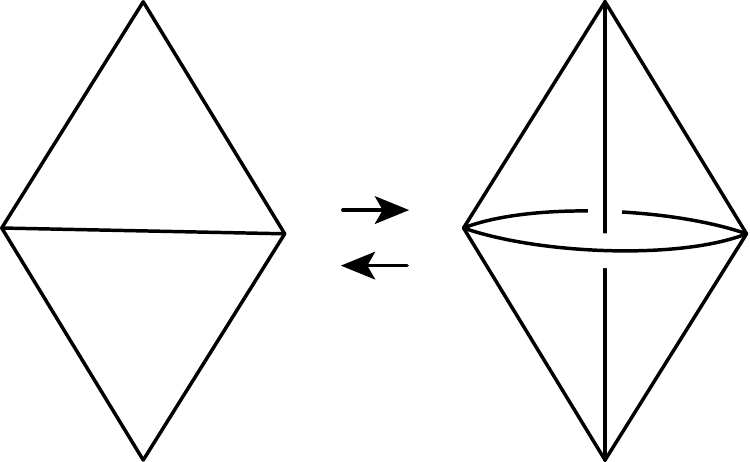}
\label{0-2}}
\caption{Moves on (topological) triangulations.}
\label{2-3 and 0-2 moves}
\end{figure}

The pentagon identity for the tetrahedral index gives the following key 
property of the 3D-index:

\begin{thm} [{\cite[\sect6]{Gar},\cite[Theorem 1.1]{GHRS}}]
\label{thm.00}
If $\CT$ and $\CT'$ are related by a 2-3 move and both are 1-efficient, 
then $I_{\CT}=I_{\CT'}$.
\end{thm}

The quadratic identity for the tetrahedral index gives:

\begin{thm} \cite[Theorem 5.1]{GHRS}
\label{thm.20move.index}
If $\CT$ and $\CT'$ are ideal triangulations related by a 
0-2 move and both are 1-efficient, then $I_{\CT}=I_{\CT'}$.
\end{thm}

\medskip
It is currently unknown whether any two  1-efficient triangulations of a 
given cusped manifold are connected by 2-3/3-2 and 0-2/2-0 moves preserving 
1-efficiency.  See Section \ref{PachnerGraph} for a discussion of some 
experimental results on this question.   

\medskip
However, using the previous results, we can obtain a  
\emph{topological invariant} $I_M$ of any cusped hyperbolic 3-manifold $M$ 
as follows (see [GHRS] for the details).

\begin{itemize}
\item 
Consider the {\em canonical} or {\em Epstein-Penner} decomposition of $M$ 
into convex ideal hyperbolic polyhedra (using horoball cusp neighbourhoods 
of equal volume if $M$ has more than one cusp).
\item 
If the Epstein-Penner decomposition is a triangulation $\tri$,
then we define $I_M=I_\tri$.
\item 
In general, we can define $I_M=I_\tri$ where $\tri$ is any triangulation
in a class $\mathcal{EP}$ consisting of all \emph{regular triangulations} 
of the Epstein-Penner cells, with all possible layered triangulations of 
the bridge regions between any incompatible triangulations of the faces.
\item 
A \emph{regular triangulation} of a set of points $A= \{ a_i \}$ in $\R^n$, 
is obtained by lifting the points vertically to points $(a_i,t_i)$ in 
$\R^{n+1}$, taking the bottom faces of their convex hull in $\R^{n+1}$, and 
then projecting back to $\R^n$. (Use the projective model to extend 
this definition to $\H^n$.)
\item 
All triangulations in class $\mathcal{EP}$ admit semi-angle structures,
hence are 1-efficient.
\item 
Results of Gelfand, Kapranov and Zelevinsky on {regular triangulations} 
imply that $\mathcal{EP}$ is connected by sequences of 2-3, 3-2, 0-2 and 
2-0 moves, staying within 1-efficient triangulations. But the index is 
invariant under such moves, so $I(\tri)$ is the same for all $\tri$ in 
$\mathcal{EP}$.
\item 
Hence we obtain a well-defined invariant for any cusped hyperbolic 3-manifold.
\end{itemize}

%%%%%%%%%%%%%%%%%%%%%%%%%%%%%%%%%%%%%%%%%%%%%%%%%%%%%%%%%%%%%%%%%%%%%%%%%%%
%%%%%%%%%%%%%%%%%%%%%%%%%%%%%%%%%%%%%%%%%%%%%%%%%%%%%%%%%%%%%%%%%%%%%%%%%%%

\section{Some Q-normal surface theory}

In this section we discuss surfaces in general position with respect to
an ideal triangulation. This discussion is similar to the one of Haken
who studied and encoded normal surfaces in general position with respect
to a triangulation of a closed 3-manifold. A major difference is that amongst surfaces 
in general position with respect to an ideal triangulation is the class of embedded
spun normal surfaces (a concept introduced by Thurston, \cite{Th,Wa}), 
which intersect each tetrahedron 
in a finite number of quadrilaterals, and a possibly infinite number of triangles.
Also, just as in the closed case, a spun normal surface can be reconstructed from its quadrilateral data and the quadrilaterals satisfy the $Q$-matching equations.

Let $M$ be an orientable 3-manifold with boundary consisting of $r$ tori, 
and let $\T$ be an {\em oriented} ideal triangulation with $n$ tetrahedra.
Let $\calE$ denote the set of edges of $\T$, and recall that $\square$ 
denotes the set of quad types in the tetrahedra of $\T$. 
Of course, $|\calE|=n$ and $|\square|=3n$.

Recall that the {\em $Q$-normal surface solution space} $Q(\T;\R)$ for 
$\T$ is a subset of $\R^\square \cong \R^{3n}$ consisting of real quad 
coordinates assigned to the quad types in $\T$
satisfying the $Q$-matching equations of Tollefson. 
After choosing a cyclic ordering of quad types in each tetrahedron, 
compatible with the orientation on $\T$, 
we can write the quad coordinates of $S \in Q(\T;\R)$ as a vector
$$
S=(a_1, b_1, c_1, a_2, b_2, c_2, \ldots , a_n, b_n, c_n) \in(\R^3)^n= \R^{3n}
$$
where $(a_j,b_j,c_j)$ are the quad coordinates of $S$ in tetrahedron $j$.

Let $Q(\T;\Z) = Q(\T;\R) \cap \Z^{3n}$ be the sublattice of $Q(\T;\R)$ of 
integer solutions, and  let $Q(\T;\Z_+) = Q(\T;\Z) \cap \R_+^{3n} $ denote 
the set of integer solutions with non-negative quad coordinates. Then 
each element $x$ of $Q(\T;\Z_+)$ determines a (possibly singular) 
{\em spun normal surface} in $\T$ obtained by taking $x_A$ disjoint copies 
of the quad $q_A$ together with additional normal triangles. 
The resulting surface  is a (possibly singular) {\em closed normal surface} 
if only finitely many triangles are added, and is {\em embedded} if there 
is at most one non-zero quad coordinate in each tetrahedron.
(See Kang \cite{Kang},  Tillmann \cite{Ti}.)

\subsection{Construction of spun normal surfaces} 
\label{spun_normal_construction}

Next we give a brief exposition of the construction of a spun normal surface
from solutions to the Q-matching equations following \cite{DuGa} and 
\cite{Ti}. (This approach goes back to lectures of W. Thurston and was used 
by J. Weeks in SnapPea.)

For a spun normal surface, there is a fixed pattern of normal arcs in each 
ideal triangle, consisting  of three infinite families of parallel arcs 
(one at each corner). Further there is a well-defined ``middle'' interval 
in each edge of the triangle separating two of these families. The union 
of these three intervals with 3 normal arcs form a hexagon, as described 
in \cite{DuGa}.

\begin{center}
\begin{figure}[htpb]
\includegraphics[width=0.2\textwidth]{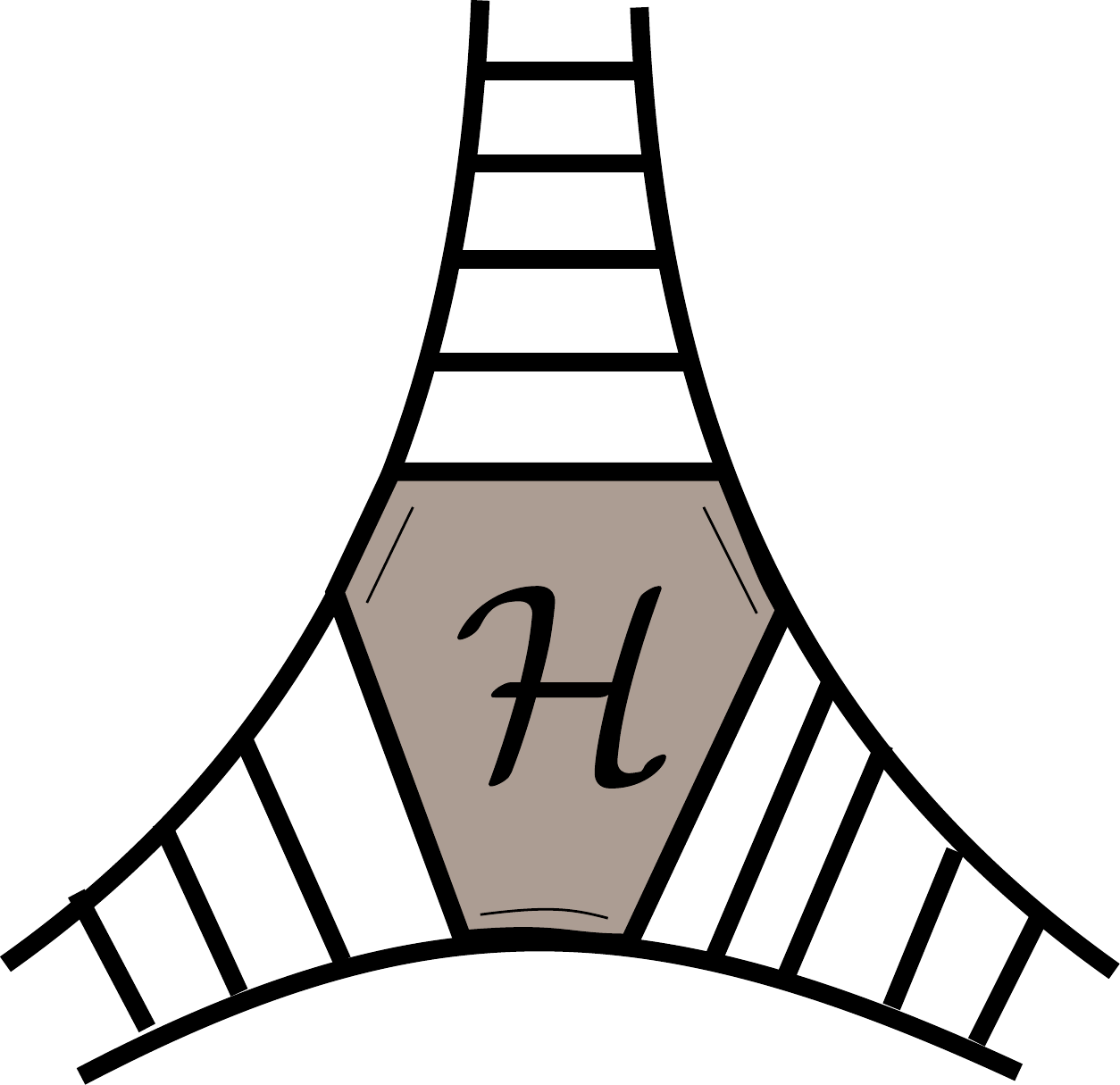}
%\resizebox{5 cm}{!}{\input{hexagon.pdf_tex}}
\caption{\label{fig:hexagon} The hexagon $\mathcal{H}$ has three arcs 
which are normal arcs and three arcs (labelled by double lines) that 
are not normal arcs and represent the "middle" intervals.}
\end{figure}
\end{center}

Next consider the arc pattern on the boundary of a tetrahedron: 
Along each edge,  there is an integer ``shear'' or ``shift'' parameter 
which specifies how the adjacent arc patterns fit together, and we choose 
a sign convention as shown in Figure \ref{fig:shear} below (as viewed 
from outside the tetrahedron).

\begin{center}
\begin{figure}[htpb]
\includegraphics[width=0.35\textwidth]{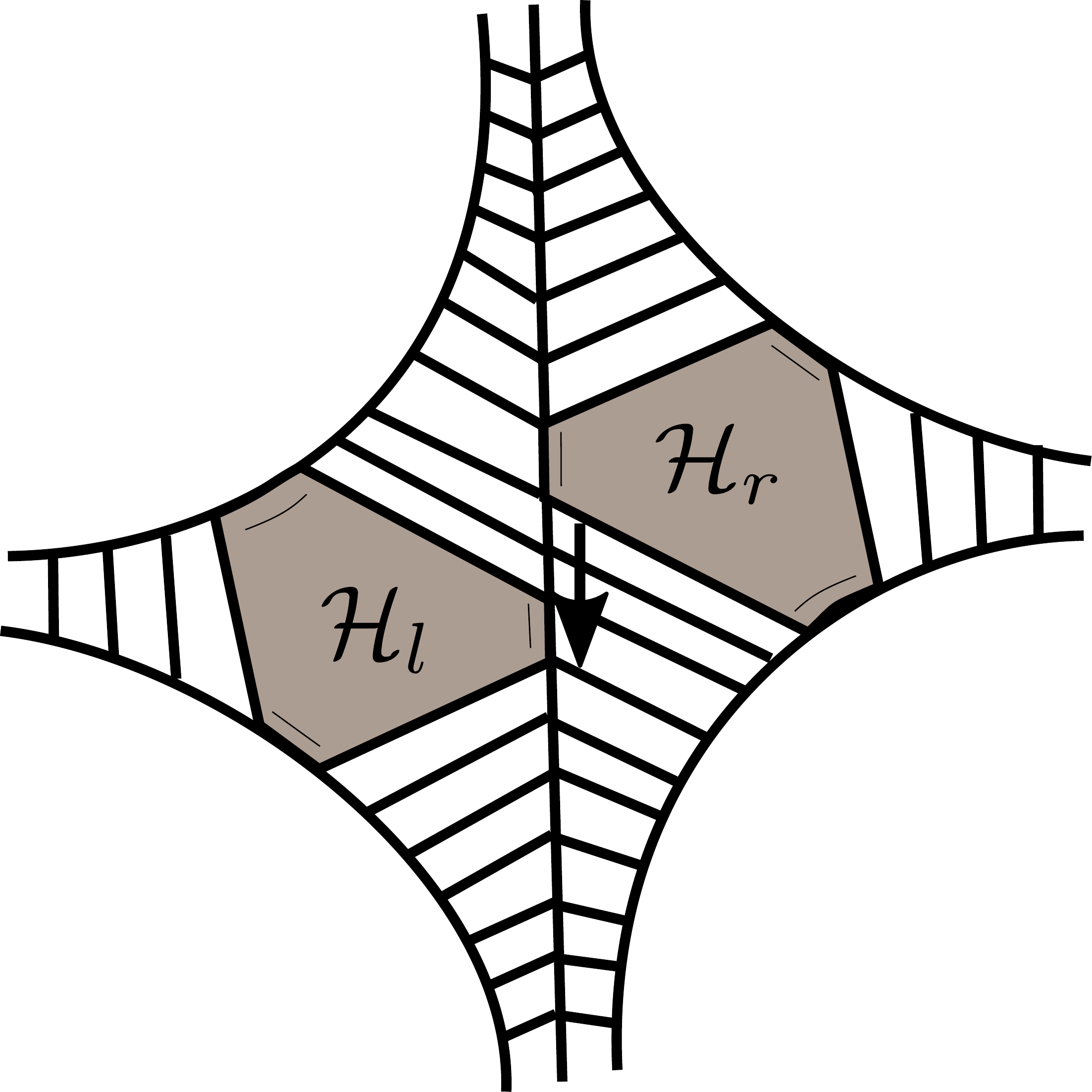}
%\resizebox{8 cm}{!}{\input{figures/shear.pdf_tex}}
\caption{\label{fig:shear} A shift of -3 from $\mathcal{H}_r$ to 
$\mathcal{H}_l$. (In general, shearing parameters measure the offset 
from right to left as viewed from outside the tetrahedron.)}
\end{figure}
\end{center}

We require that  the sum of shear parameters along the three edges meeting 
at a vertex is zero. This means that there is an infinite family of 
parallel normal triangles at each tetrahedron vertex, and also implies 
that the shear parameters on opposite edges are equal. With our sign 
convention, quad coordinates $(a,b,c)$ for quad types 
$q_{A:01:23},q_{A:02:13},q_{A:03:12}$ in a tetrahedron correspond to 
shear parameters $(c-b,a-c,b-a)$ along the edges of the tetrahedron facing 
$q_{A:01:23},q_{A:02:13},q_{A:03:12}$. (Figure \ref{fig:shearing_tet} below shows 
the case of two quads of type $q_{A:01:23}$, separating vertices 0, 1 from 
vertices 2, 3.)

\begin{figure}[htpb]
\begin{center}
\includegraphics[width=0.3\textwidth]{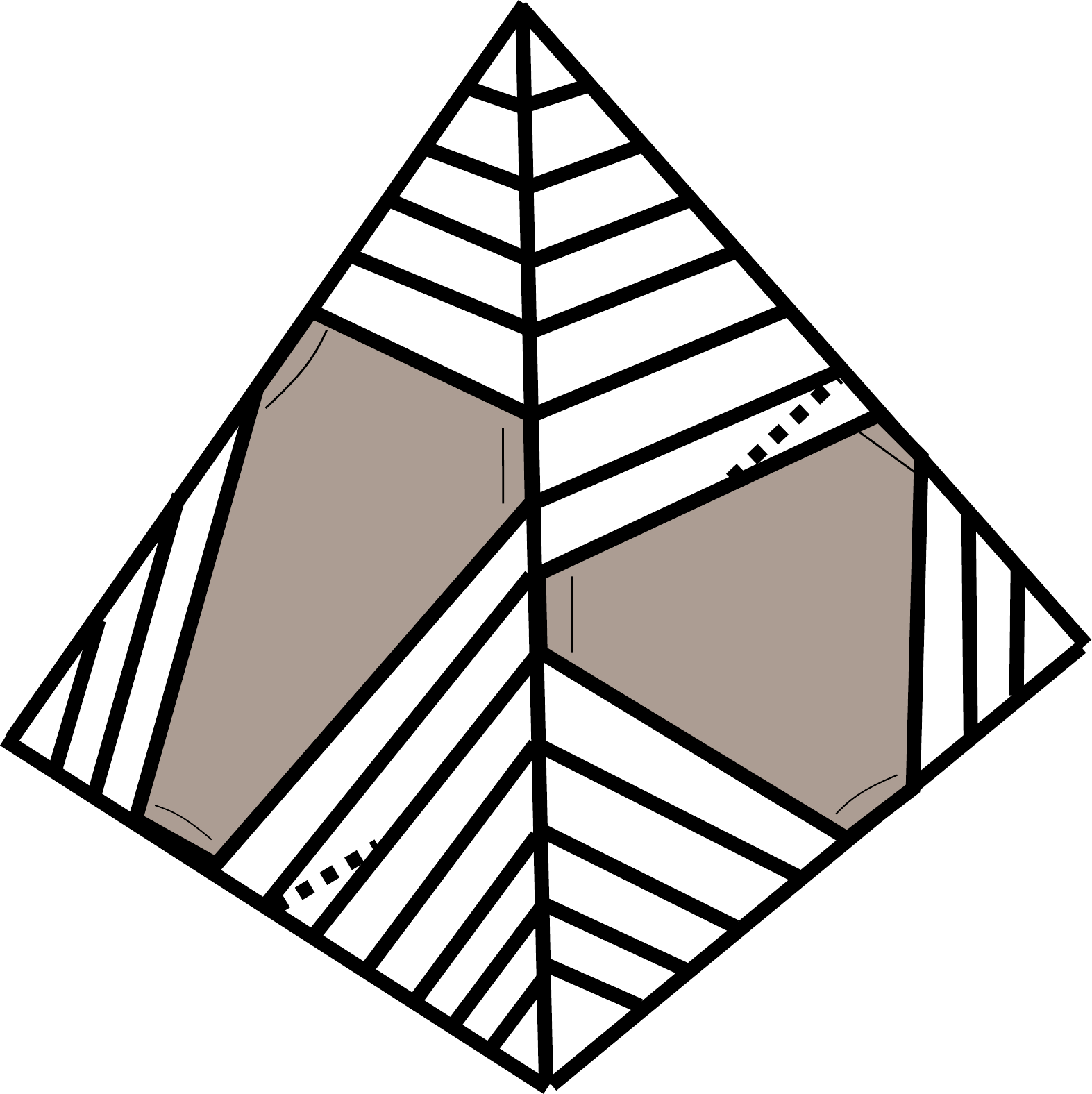} \hspace{1cm}
\includegraphics[width=0.3\textwidth]{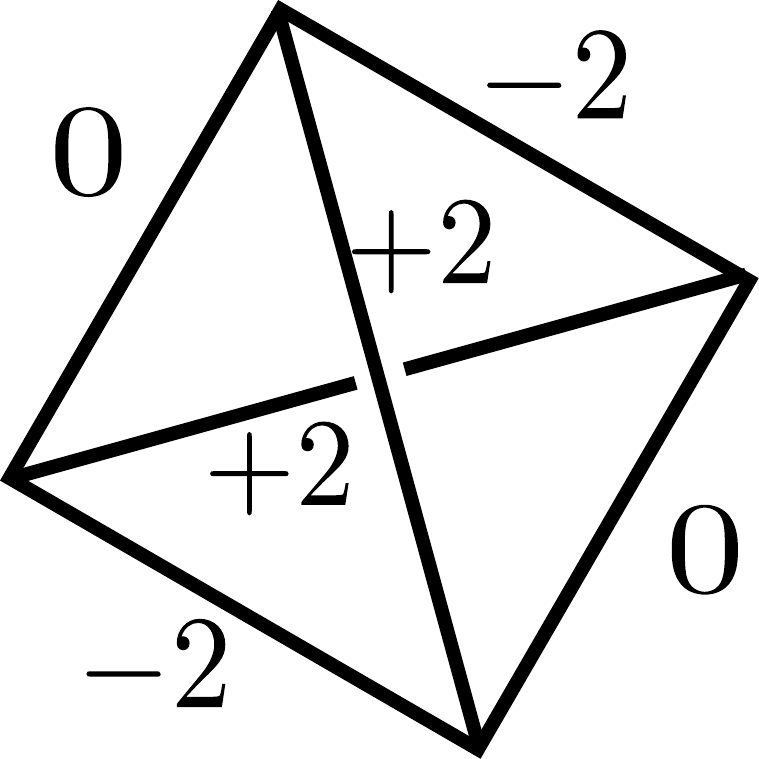}
%\resizebox{6 cm}{!}{\input{tet_shearing_change.pdf_tex}}
\end{center}
\caption{\label{fig:shearing_tet} With two quads at the bottom edge of 
the left hexagon and top edge of the right hexagon, we have a shear of +2 
along the edge between. The right figure gives all of the shearing 
coordinates on the tetrahedron.} 
\end{figure}

Now there is a unique way to glue together the arc patterns in two 
tetrahedra meeting along a common face, as the middle intervals in the faces 
must match up. The arc patterns from all the tetrahedra in the triangulation 
then fit together consistently (without any shearing) around an edge class 
if and only the sum of the shear parameters over all edges in the edge 
class vanishes. These conditions are precisely the Q-matching equations.

\subsection{Geometric generators for Q-normal classes}

The work of Kang-Rubinstein \cite[Theorem 2.1]{KR1} (also see 
Luo-Tillmann \cite[\sect2]{LT}) shows that the space 
$\tilde N(\T;\R) \subset \R^{7n}$  of all {\em closed normal classes} 
satisfying Haken's matching equations (including, for example, the 
boundary tori) has a basis given by {\em edge solutions} $\tilde E_i$ and 
{\em tetrahedron solutions} $\tilde T_j$ where $1\le i ,j \le n$. 
These project to {\em edge solutions} $E_i$ and {\em tetrahedron solutions} 
$T_j$ under the map which forgets triangular coordinates, giving the image $N(\T;\R) \subset \R^{3n}$. Define $Q_0(\T;\R)$ as the solutions of the $Q$-matching equations corresponding to closed normal surfaces. It is a result of \cite{To} that the $Q$-coordinates determine a normal surface up to multiples of the peripheral normal tori and so $Q_0(\T;\R)=N(\T;\R)$.

When considering the map from all 
closed normal classes to $Q$-normal solution space, the closed normal 
surface solutions coming from the boundary tori map to the zero vector. 
Thus, such a map introduces one linear relation for each cusp, which in fact 
gives the only relations (see for example the proof of \cite[Theorem 3.1]{KR1}), 
so $\dim Q_0(\T;\R) = 2n-r$,
Explicitly, the quad coordinates for $E_i$ are 
precisely the coefficients in the {\em edge equation} for the $i$th edge 
in $\T$, and the quad coordinates for $T_j$ are $(1,1,1)$ in the $jth$ 
tetrahedron and $(0,0,0)$ in all other tetrahedra. 

The space $Q(\T;\R)$ of all {\em Q-normal classes} is spanned by 
{\em edge solutions}, {\em tetrahedron solutions} and 
{\em peripheral curve solutions} (see \cite[Theorem 3.1]{KR1}).     
The peripheral curve solutions are linearly independent from the
edge and tetrahedron solutions, and there are two linearly independent 
peripheral curve solutions $M_k, L_k$ per cusp corresponding to a choice of 
basis (``meridian'' $\mu_k$, ``longitude'' $\lambda_k$) 
of $H_1(\bd_k M;\R)$ for each component $\bd_k M$ of $\bd M$. 
In fact, a suitable choice of $n-r$ edge solutions together with the $n$ 
tetrahedral solutions and the $2r$ peripheral solutions form a basis for 
the real solution space $Q(\T;\R)$, which therefore has dimension $2n+r$ 
(which also follows from \cite[Theorem 3.1]{KR1}). 
 
In fact the statements above follow from the symplectic relations of 
Neumann-Zagier (see \cite{NZ}, \cite{N}, \cite{Choi}). Let $A$ be the 
$n \times 3n$ gluing equation matrix for $\tri$, with rows given by 
$E_1, \ldots, E_n$. Then the $n \times 3n$ matrix of $Q$-matching equations 
$B$ is given by $B=AC$  where $C$ is the $3n \times 3n$ block diagonal 
matrix made up of $n$ copies of
$$
\begin{bmatrix} 0 & 1 & -1 \\ -1 & 0 & 1 \\ 1 & -1 & 0\end{bmatrix}
$$
(see Tillmann \cite{Ti}). We can also regard $C$ as the linear map 
$C: \R^{3n} \to \R^{3n}$
given by
$$
C(a_1, b_1, c_1, a_2, b_2, c_2, \ldots , a_n, b_n, c_n) 
= (-b_1+c_1, -c_1+a_1,-a_1+b_1, \ldots, -b_n+c_n, -c_n+a_n,-a_n+b_n).
$$
Then the $Q$-matching equations have coefficients $C(E_i)$ for 
$i=1, \ldots , n$. (Compare Section \ref{spun_normal_construction} above.)

Now the skew-symmetric pairing of Neumann-Zagier 
%can be written in the form 
is given by
$$
\omega : \R^{3n} \times \R^{3n} \to \R, \qquad %\text{  defined by }  
\omega(x,x') = Cx \cdot x'
$$
where $\cdot$ denotes the dot product. Explicitly, given vectors 
$$
x=(a_1, b_1, c_1, a_2, b_2, c_2, \ldots , a_n, b_n, c_n), 
x'=(a'_1, b'_1, c'_1, a'_2, b'_2, c'_2, \ldots , a'_n, b'_n, c'_n) 
\in   (\R^3)^n
$$
we have
$$
\omega(x,x') = \sum_{j=1}^n  \left(  
\begin{vmatrix} a_j & a_j' \\ b_j & b_j' \end{vmatrix} +  
\begin{vmatrix} b_j& b_j' \\ c_j & c_j' \end{vmatrix} +   
\begin{vmatrix} c_j& c_j' \\ a_j & a_j' \end{vmatrix}\right) .
$$
(Compare \cite{N,Choi}.)

It follows that $S \in \R^{3n}$ satisfies the $Q$-matching equations if 
and only if 
\be 
\label{Qmatch}
\omega(E_i,S) = C(E_i) \cdot S = 0 \text{ for all } i = 1, \ldots ,n.
\ee

It is immediate that each tetrahedral solution $T_j$ satisfies the 
$Q$-matching equations; in fact,  $C(T_j)=0$ so 
$\omega(x,T_j)= -\omega(T_j,x) = -C(T_j) \cdot x = 0$ for all $x \in \R^{3n}$.
Further, the {\em Neumann-Zagier symplectic relations}  say that all 
symplectic products of $E_i, M_k, L_k$ for $i=1,\ldots, n$, 
$k=1,\ldots,r$ are {\em zero} except that
$$
\omega(L_k,M_k) = 2 = -\omega(M_k,L_k)  \text{ for each }k=1, \ldots , r .
$$
Hence, we can conclude that the edge solutions $E_i$ and peripheral 
solutions $M_k,L_k$ satisfy the $Q$-matching equations.  
In  fact, with some more analysis, it can been seen that the tetrahedral 
solutions together with  a suitable choice of $n-r$ edge solutions  and 
the peripheral solutions form a basis for the real solution space $Q(\T;\R)$,
which has dimension $2n+r$.

As mentioned above, there is a basis for the $Q$-coordinates of spun 
normal surfaces given by appropriately chosen edge solutions 
$\{E_i \}_{i=1}^{n-r}$, all tetrahedral solutions $\{T_j \}_{j=1}^n$, and 
peripheral solutions $\{p_k M_k + q_k L_k\}_{k=1}^r$. Hence for each choice of 
$x_i,y_j,p_k,q_k \in \R$ we obtain a spun normal class
$$
S= \sum_i x_i E_i + \sum_j y_i T_j + \sum_k (p_k M_k +q_k L_k).
$$
A normal surface has two important invariants; its Euler 
characteristic and its boundary slope. These invariants give rise to 
two linear maps on the normal surface  solution space $Q(\T;\R)$.
%Now there are two important linear maps defined on the normal surface 
%solution space $Q(\T;\R)$.

\begin{defn}
The {\em formal Euler characteristic} is the linear map 
$$
\chi : Q(\T;\R) \to \R, \qquad \chi(S)= \sum_i -2x_i- \sum_j y_j 
$$
giving the usual Euler characteristic for embedded closed and spun normal 
surfaces.
\end{defn}

(See \cite{LT} for a detailed discussion of the closed case, and 
Appendix \ref{gen_angle_euler} for the spun normal case using generalised 
angle structures.)

\begin{defn} 
The {\em boundary} map 
$$
\bd : Q(\T;\R) \to H_1(\bd M;\R), \qquad \bd(S)= 
2 \sum_k (p_k \mu_k +q_k \lambda_k)
$$
gives the boundary slope and the direction of spinning for embedded spun 
normal surfaces. Here $\mu_k$ and $\lambda_k$ are the meridian and longitude of cusp $k$ as defined previously. 
\end{defn}

(See Tillmann \cite{Ti} for the details.\footnote{Our sign convention is 
opposite to that in \cite{Ti} and Regina v.4.96 (\cite{regina}),
but is consistent with the boundary map defined in \cite{N} and discussed 
in \sect \ref{ZNST} below.}) Note that at each cusp, a tail of a spun 
normal surface is an infinite annulus spiralling into the cusp. There are 
two possible directions of spinning for each such a tail.

We can compute the boundary  in terms of the Neumann-Zagier symplectic 
form: since
$$
\omega(S,M_k) = \omega(p_k M_k +q_k L_k,M_k)=-2q_k
$$
and
$$
\omega(S,L_k) = \omega(p_k M_k +q_k L_k,L_k)=2p_k
$$
we have
$$
\bd(S)= \sum_k \left(  -\omega(S,L_k) \mu_k + \omega(S,M_k) \lambda_k \right).
$$
Note that 
\begin{center}
\emph{$\bd(S)=0$ if and only if $S$ defines a closed normal class.}
\end{center}

Hence $N(\T;\R)$ is the solution space of the $Q$-matching equations 
(\ref{Qmatch}) together with the
additional equations 

\be 
\label{Qmatch_closed}
\omega(M_k,S) = \omega(L_k,S)  = 0 \text{ for all } k = 1, \ldots, r.
\ee

We now define a %new
quadratic function which will play an important role 
in understanding the degree of terms in the 3D-index sums.

\begin{defn}
We define a {\em double arc} function 
\be
\delta : Q(\T;\R)  \to \R, \qquad 
\delta(S) = \sum_j ( a_j b_j+b_j c_j + c_j a_j),
\ee
for 
$$
S=(a_1, b_1, c_1, a_2, b_2, c_2, \ldots , a_n, b_n, c_n). 
$$
\end{defn}

Note that for $S \in Q(\tri;\Z_+)$, $a_j b_j +b_j c_j + c_j a_j$ counts the 
number of arcs of intersection between the quads of $S$ in tetrahedron $j$.
In particular,
\begin{center}
\emph{a normal surface solution $S \in Q(\T;\Z_+)$ is embedded if and only 
if $\delta(S)=0$.}
\end{center}

There is also an associated symmetric bilinear function 
$$
\delta : Q(\T;\R) \times Q(\T;\R) \to \R
$$
such that 
\be
\delta(S+S')=\delta(S)+\delta(S') + 2 \delta(S,S') 
\text{ for all } S, S' \in Q(\T;\R) .
\label{bilinear_delta}
\ee
Explicitly
$$
\delta(S,S') = \frac{1}{2}  
\sum_j ( a_j b_j' +b_j c_j' + c_j a_j' )+( a'_j b_j +b'_j c_j + c'_j a_j )
$$
if $S$ and $S'$ have quad coordinates in  $\R^{3n}$ given by
$$
[S]=(a_1, b_1, c_1, a_2, b_2, c_2, \ldots , a_n, b_n, c_n)
$$
and
$$
[S']=(a_1', b_1', c_1', a_2', b_2', c_2', \ldots , a_n', b_n', c_n').
$$

%%%%%%%%%%%%%%%%%%%%%%%%%%%%%%%%%%%%%%%%%%%%%%%%%%%%%%%%%%%%%%%%%%%%%%%%%%%
%%%%%%%%%%%%%%%%%%%%%%%%%%%%%%%%%%%%%%%%%%%%%%%%%%%%%%%%%%%%%%%%%%%%%%%%%%%

\section{Combinatorics of ideal triangulations and integer 
normal surface theory}
\label{ZNST}

The work of Kang-Rubinstein and Luo-Tillmann, described above, shows that 
the edge solutions and tetrahedral solutions 
span the real vector space $N(\T;\R) \subset \R^{3n}$ of closed normal 
classes, and that these together with the 
peripheral curve solutions span the real vector space 
$Q(\T;\R) \subset \R^{3n}$ of all $Q$-normal classes.
However, over the integers the situation is more subtle --- the integer 
linear combinations of edge and tetrahedral solutions generally give only 
a finite index submodule of $N(\T;\Z)$ and the integer linear combinations 
of edge, tetrahedral and peripheral curve solutions  give a finite index 
submodule of $Q(\T;\Z)$. In this section we give a precise 
description of these integer classes. It turns out that this is a 
consequence of the results of Neumann  \cite{N} on combinatorics of ideal 
triangulations.

We regard a $Q$-normal class  as a linear combination of quads in 
$\T$ satisfying the $Q$-matching equations of Tollefson \cite{To}. 
Explicitly, let $\Z E \cong \Z^n$ and $\Z\quads \cong \Z^{3n}$ 
denote the free $\Z$-modules with bases given by the (unoriented) edge 
classes and quad types in $\T$ respectively. Given a quad $q_\sigma$ in an 
oriented tetrahedron $\sigma$ of $\T$ we associate a sign $\pm 1$ to each 
of the edges of $\sigma$ meeting $q_\sigma$ as shown in 
figure \ref{twisted_octagon}. (These signs give the shear parameters 
along the edges as explained in Section \ref{spun_normal_construction}.)
Adding these contributions gives a linear map 
\be
\label{Qmatching_map}
F: \Z\quads \to \Z E
\ee
whose kernel is the submodule $Q(\T;\Z)$ of $Q$-normal classes.

\begin{figure}[htpb]
\begin{center}
\includegraphics[width=0.26\textwidth]{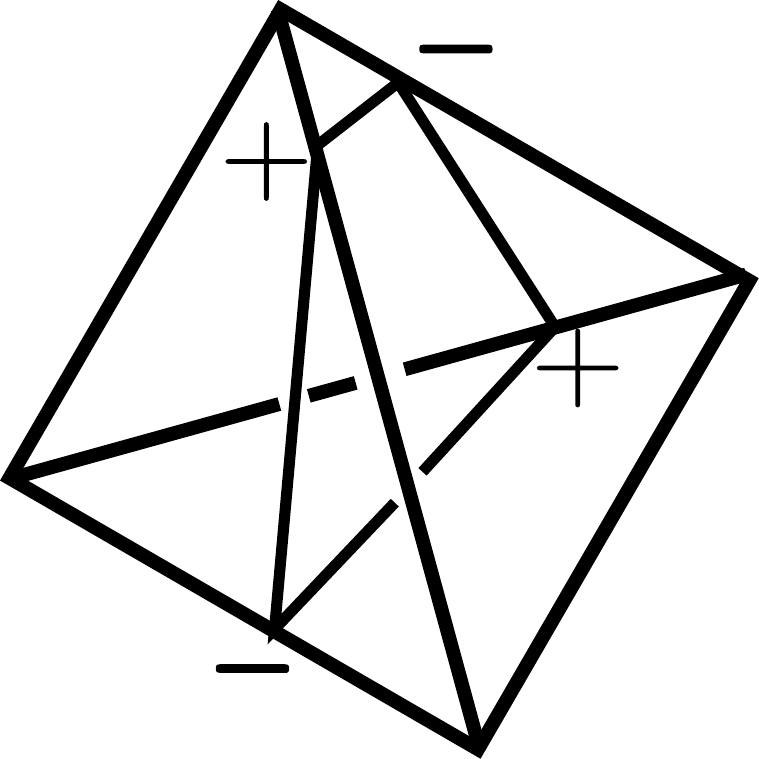} 
\hspace{1cm}
\includegraphics[width=0.26\textwidth]{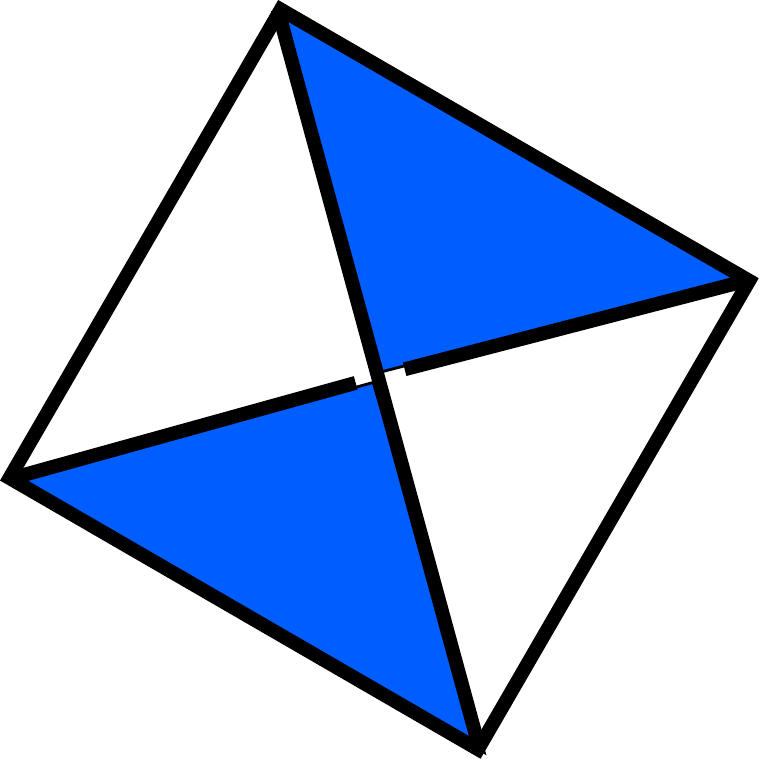} 
\hspace{1cm}
\includegraphics[width=0.26\textwidth]{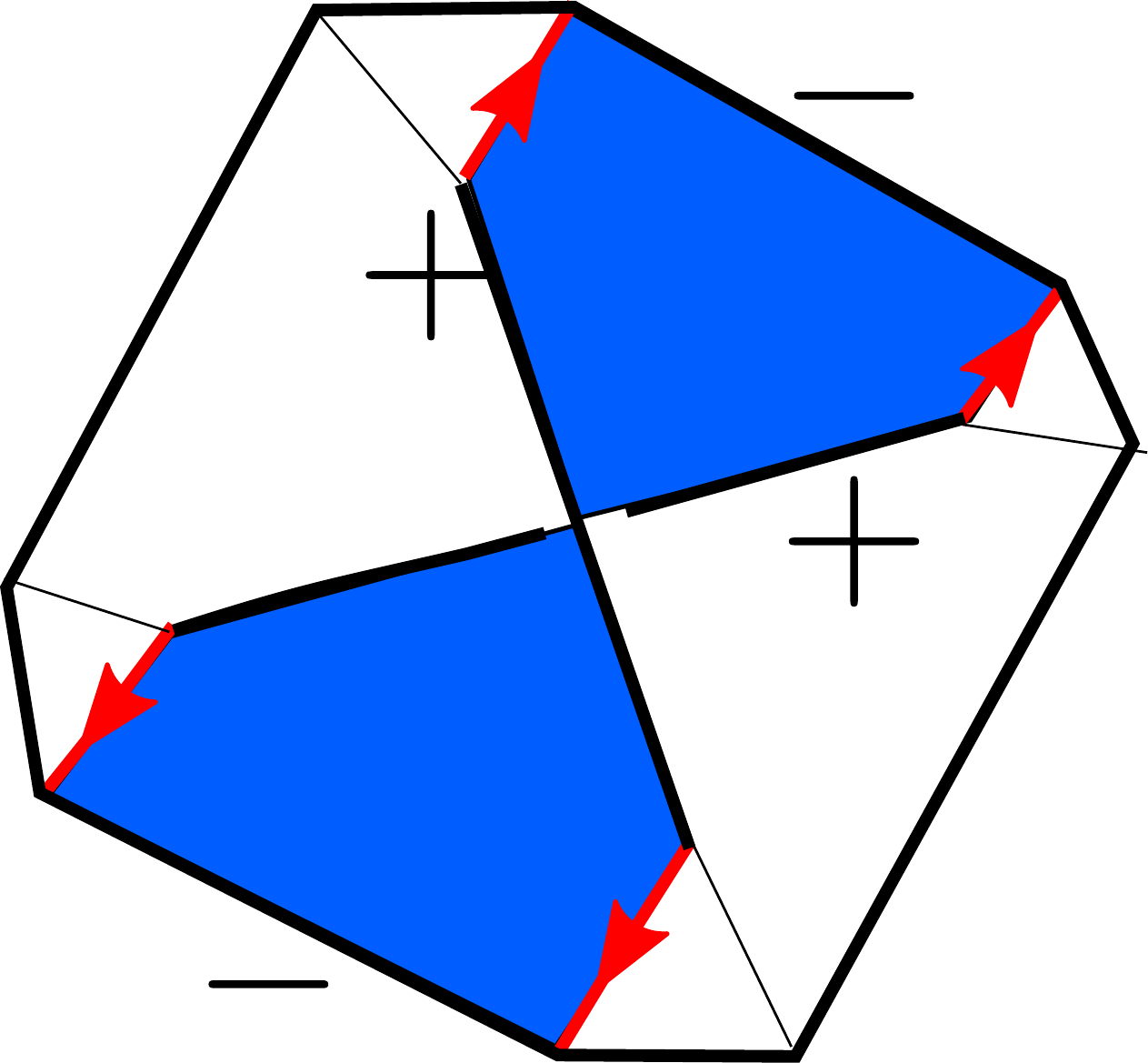}
\end{center}
\caption{Each quad can be replaced by a twisted square or twisted octagon.}
\label{twisted_octagon}
\end{figure}

Kang-Rubinstein \cite{KR1} observed that for each edge class  in $\T$, the 
sum of all quad types facing the edges in this class gives an 
``edge solution''  in $Q(\T;\Z)$; this gives a linear map 

\be 
\label{edge_solution_map}
G: \Z E \to \Z\quads 
\ee 
with image contained in $Q(\T;\Z)$.
Further, the sum of the three quad types in a tetrahedron in $\T$ gives a 
``tetrahedral solution'' in $Q(\T;\Z)$. We let 
$\EE,\TT \subset Q(\T;\Z) \subset \Z\quads$ denote the $\Z$-submodules 
spanned by the edge solutions and tetrahedral solutions respectively, and 
let $Q_0(\T;\Z)$ denote the submodule $\EE+\TT  \subset Q(\T;\Z)$.

Next we associate two important homology (or cohomology) classes with any 
$Q$-normal class. First observe that any quadrilateral in a tetrahedron can 
be replaced by a `twisted square' meeting the same edges, or by a 
`twisted octagon' in the corresponding truncated tetrahedron with its 
four external edges oriented as shown in Figure \ref{twisted_octagon} 
(compare \cite[Figures 9, 10]{N}). Note that the arrow on each external 
edge goes from an internal edge labelled $+1$ to an internal edge labelled $-1$.

Each spun normal class $S \in Q(\T;\Z)$ is a linear combination of quads, so 
we can replace this by a linear combination $S'$ of twisted octagons. It 
follows from the $Q$-matching equations that  $S'$ represents a mod 2 
homology class $[S]_2 \in H_2(M,\bd M; \Z/2\Z)$. Here and throughout the 
paper, we will use $[~]_2$ to indicate a $\Z/2\Z$ homology (or co-homology) 
class. Further, the boundary edges of $S'$, oriented as above, represent 
a homology class $[\bd S] \in H_1(M;\Z)$. In particular, the kernel of the 
boundary map $S \mapsto [\bd S]$ is precisely the  submodule 
$N(\T;\Z) \subset Q(\T;\Z)$ of {\em  closed} normal classes. 

When $S$ gives an embedded spun normal surface,  $[\bd S]$ agrees with 
boundary map defined by Tillmann in \cite{Ti} (up to sign), and describes 
the boundary components of the spun normal surface and the direction of 
spiralling of the spun normal surface around $\bd M$.  (With the 
orientation convention given in Figure \ref{twisted_octagon}, the ends of 
a spun normal surface spiral up into the cusp to the {\em right} of the 
oriented boundary curve, as viewed from the cusp.)

Since $S'$ gives a 2-chain mod 2 whose boundary represents the reduction of 
$[\bd S]$ mod 2, it follows that 
$$
([S]_2,[\bd S]) \in \{ (a,b) \in  H_2(M,\bd M;\Z/2\Z) \times H_1(\bd M;\Z) : 
\bd_* a =   b \bmod 2  \}
$$
where $\bd_*: H_2(M,\bd M;\Z/2\Z) \to H_1(\bd M;\Z/2\Z) $ is the connecting 
homomorphism in the long exact sequence of the pair $(M, \bd M)$.   
Hence, $[\bd S] \in {\mathcal K} =\Ker(H_1(\bd M;\Z) \to H_1(M;\Z/2\Z))$, 
by the long exact sequence of the pair $(M,\bd M)$ with $\Z/2\Z$ coefficients.

\begin{thm} 
\label{ZQNT_struct_thm}
The homomorphism 
$$
H:  Q(\T;\Z) \to H_2(M,\bd M;\Z/2\Z) \times H_1(\bd M;\Z),~ S 
\mapsto ([S]_2,[\bd S])
$$
has image 
$$
\Im H =\{ (a,b) \in H_2(M,\bd M;\Z/2\Z) \times  H_1(\bd M;\Z) :  
\bd_* a =  b \bmod 2 \}
$$
and kernel
$$
\Ker H = Q_0(\T;\Z) = \EE + \TT.
$$
In particular, the homomorphism
$$
H_0:  N(\T;\Z) \to  H_2(M,\bd M;\Z/2\Z),~ S \mapsto [S]_2
$$
has image
$$
\Im H_0 = \Ker( H_2(M,\bd M;\Z/2\Z) \to H_2(\bd M;\Z/2\Z)) 
=  \Im ( H_2(M;\Z/2\Z) \to H_2(M,\bd M;\Z/2\Z) )
$$
and kernel
$$
\Ker H_0 =Q_0(\T;\Z)= \EE + \TT.
$$
\end{thm}

\begin{proof} 
This essentially follows from Neumann \cite[Theorem 5.1]{N} together with 
Poincar\'e duality. To match Neumann's notation from sections 4--6 
of \cite{N} we let $K$ denote the oriented pseudo-manifold given by the 
triangulation $\T$ with its ideal vertices included.
Thus $K$ is homeomorphic to the end compactification $\hat M$ of the 
$\text{int} M$, obtained by collapsing each component of  $\bd M$
to a separate point. Truncating the corners of the tetrahedra in $K$ gives 
a cell complex $K_0$ homeomorphic to $M$ with boundary
$\bd K_0$ homeomorphic to $\bd M$. 

Let  $C_0 =\Z V \cong \Z^r , C_1=\Z E \cong \Z^n ,  
\overline{J}=\Z \quads \cong\Z^{3n}$  be the free $\Z$-modules 
with bases given by the cusps, (unoriented) edge classes, and quad types 
in $\T$ respectively, and let $\EE, \TT\subset \Z \quads$ denote the 
$\Z$-submodules spanned by the edge solutions and  tetrahedral solutions
respectively.  Then Neumann defines a chain complex
$$
0 \to C_0  \xrightarrow{\alpha} C_1  \xrightarrow{\beta}  J  
\xrightarrow{\beta^*}  C_1  \xrightarrow{\alpha^*}  C_0 \to 0
$$
where $J= \Z \quads/\TT \cong \Z^{2n}$, $\alpha$ gives the sum of all 
edges incident to a vertex, $\alpha^*$ is a map that associates to an 
edge the sum of its endpoints, and $\beta$ and $\beta^*$ are defined so 
that the following diagram commutes:

\begin{center}
\begin{tabular}{c}
\xymatrix{
& \overline{J}= \Z\quads \ar[dd]_\pi \ar[dr]^F &\\
C_ 1= \Z\EE  \ar[ur]^G \ar[dr]_\beta & & C_ 1= \Z\EE\\
&J=\Z\quads/\TT \ar[ur]_{\beta^*} & 
}
\end{tabular}
\end{center}

Here, $\pi: \Z\quads  \rightarrow  \Z\quads/\TT$ is the quotient map.

To relate this  to integer  $Q$-normal surface theory we observe that 
$\Ker \beta^* = Q(\T;\Z)/\TT$ and $\Im \beta = (\EE + \TT)/\TT$.

Now Neumann's Theorem 5.1 gives an isomorphism from 
$\Ker \beta^*/\Im \beta = Q(\T;\Z)/(\EE + \TT)$ to
$$
\{ (f,g) \in  H^1(M;\Z/2\Z) \times H^1(\bd M;\Z)  :   i^* f =  g \bmod 2 \}
$$
where 
$i^*: H^1(M;\Z/2\Z) \to H^1(\bd M;\Z/2\Z) $ is induced by the inclusion 
map $i: \bd M \to M$. In other words, this is the set of pairs
$(f,g)$ with $f \in \Hom(H_1(M),\Z/2\Z)$, $g \in \Hom(H_1(\bd M),\Z)$, 
such that $f(\gamma) =g(\gamma) \bmod 2$  for all $ \gamma \in H_1(\bd M)$,
where we write  $H_1(\bd M)= H_1(\bd M; \Z)$ and $H_1(M)= H_1(M; \Z)$. 

Now we have a commutative diagram
$$
\begin{CD}
H^1(M,\bd M;\Z/2\Z) @>j^*>> H^1(M;\Z/2\Z) @>i^*>> H^1(\bd M;\Z/2\Z) \\
@VV\cong V                      @VV\cong V     @VV\cong V          \\
H_2(M;\Z/2\Z) @>j_*>> H_2(M,\bd M; \Z/2\Z)  @>\bd_*>>   H_1(\bd M; \Z/2\Z)
\end{CD}
$$
where the vertical maps are the isomorphisms given by Poincar\'e duality.
This gives an isomorphism 
$$
Q(\T;\Z)/(\EE + \TT) \cong \{ (a,b) \in  H_2(M,\bd M;\Z/2\Z) 
\times H_1(\bd M;\Z) : \bd_* a =   b\bmod 2  \}
$$
where 
$\bd_*: H_2(M,\bd M;\Z/2\Z) \to H_1(\bd M;\Z/2\Z) $ is the connecting 
homomorphism in the long exact sequence of the pair $(M, \bd M)$. 
A careful examination of Neumann's work shows that this isomorphism
is given by the homology map $H :S \mapsto ([S]_2,[\bd S])$ defined above.  
\end{proof}

\begin{rmk} 
Using Poincar\'e duality and the Mayer-Vietoris sequence it follows that
\begin{align*}
\Im H_0 &= \Im ( H_2(M;\Z/2\Z) \to H_2(M,\bd M;\Z/2\Z)) \\ 
&\cong \Ker (H^1(M;\Z/2\Z) \to H^1(\bd M; \Z/2\Z)) \cong H^1(\hat M; \Z/2\Z).
\end{align*}
where $\hat M$ is the pseudo-manifold given by the triangulation $\T$ 
with its ideal vertices included; thus $\hat M$ is homeomorphic to $M$ with 
a cone attached over each boundary torus.
\end{rmk}

\begin{rmk}
When $M$ has no non-peripheral $\Z/2\Z$ homology (for example, if $M$ is 
a knot or link exterior in a $\Z/2\Z$ homology sphere), then 
$Q_0(\T;\Z) = N(\T;Z)$  consists of all closed normal classes.
Further, ${\mathcal K} =\Ker(H_1(\bd M;\Z) \to H_1(M;\Z/2\Z))$ is a 
subgroup of index $2^r$ in $H_1(\bd M;\Z)$ where $r$ is the number of 
components of $\bd M$.

In general, $Q_0(\T;\Z)$ is a finite index submodule of $N(\T;Z)$, and for 
any normal surface class $S \in N(\T;Z)$
its ``double'' $2S$ lies in $Q_0(\T;\Z)$.
\end{rmk}

\subsection{Geometric generators for integer Q-normal classes}
\label{subsect:gen_q_normal}

Next we describe explicit generators for $N(\T;\Z)$. First, recall that 
$Q_0(\T;\Z) = \EE +\TT$ is generated by the edge solutions and tetrahedral 
solutions described by Kang-Rubinstein \cite{KR1}. 
It follows from Theorem \ref{ZQNT_struct_thm} and the previous  remark that 
$$
N(\T;\Z)/Q_0(\T;\Z) \cong \Im H_0 \cong H^1(\hat M; \Z/2\Z).
$$

Given a mod 2 class  
$a \in \Im H_0 = \Im ( H_2(M;\Z/2\Z) \to H_2(M,\bd M;\Z/2\Z))$, we can 
construct a closed normal class  $S_0\in N(\T;\Z)$ with $[S_0]_2=a$ and 
$[\bd S_0]=0$ as follows. 
Choose a simplicial 1-cocycle $z \in Z^1(\T;\Z/2\Z)$ representing the dual 
cohomology class in $H^1(\hat M;\Z/2\Z)$.
This assigns an element $\bar 0$ or $\bar 1 \in \Z/2\Z$ to each edge of 
$\T$, and we lift these to integers $0,1 \in \Z$ giving a 1-cochain 
$\tilde z \in C^1(\T;\Z)$.  
Since  $z$ is a mod 2 cocycle, it follows that the sum of integers 
attached to the edges of any triangle in $\T$ is even. This implies that 
the integers attached to edges of any tetrahedra in $\T$ are either:
(i) all 0, (ii) 1 at the three edges of a triangular normal disc, 0 
elsewhere (ii) or 1 at the four edges of a quad, 0 elsewhere. 
Now we can construct an {\em embedded closed normal surface} $S$ by 
taking one triangle in each tetrahedron of type (ii) 
and one quad in tetrahedron of type (iii).   Forgetting the triangle 
coordinates (if any) gives a closed normal class $S_0$ in $N(\T;\Z)$
such that $[S_0]_2=a$ and $[\bd S_0]=0$.

\begin{rmk} 
Changing $z$ by a coboundary changes $S$ by adding vertex linking surfaces. 
But there is a unique closed embedded normal surface of {\em least weight} 
(i.e. having fewest intersections with the edges) obtained by this 
construction.
\end{rmk}

\medskip Given any class $b \in H_1(\bd M;\Z)$, a construction of 
Neumann \cite{N} gives a normal class $S \in Q(\T;\Z)$
with $[\bd S] = 2 b$ as follows. Represent $b$ by an oriented multi-curve 
$\beta$ which is in normal position relative to the induced triangulation 
$\T_\bd$ of $\bd M$.  Each oriented normal arc of $\beta$ lies in a truncated 
tetrahedron $\sigma$ of $\T$ and winds around one of the edges of $\sigma$, 
which faces a quad type $q$ in $\sigma$. Now add up the quads associated 
to all the normal arcs of $\beta$, with signs $+1$ or $-1$ according to 
whether the normal arc winds anticlockwise or clockwise as viewed from the 
cusp. The result is a normal class with $[\bd S] = 2 b$, $[S]_2=0$ and 
$\chi(S)=0$. 

\begin{rmk} 
These ``peripheral curve solutions'' give normal classes in $Q(\T;\Z)$
representing all {\em even classes}  
$b \in H_1(\bd  M; 2\Z) \subset  H_1(\bd  M; \Z) $, i.e. those such that
$ b \bmod 2 = 0$.  Adding these to the above closed normal classes gives 
normal classes representing all $(a,b) \in  H_2(M,\bd M;\Z/2\Z) \times 
H_1(\bd M;\Z)$ such that $\bd_* a =   b \bmod 2  =0$.
We do not know a direct geometric construction for the other (non-even) 
classes $b \in \mathcal K \subset  H_1(\bd  M; \Z)$.
\end{rmk}

\begin{rmk}
\label{half-integer-coefs}
For any $Q$-normal class $S \in Q(\T;\Z)$, its double satisfies 
$[\bd (2S)]\in H_1(\bd M;2\Z)$ and $[2S]_2=0 \in H_2(M,\bd M;\Z/2\Z)$. Hence
$2S$ is an integer linear combination of edge solutions, tetrahedral 
solutions and peripheral curve solutions.
\end{rmk}

%%%%%%%%%%%%%%%%%%%%%%%%%%%%%%%%%%%%%%%%%%%%%%%%%%%%%%%%%%%%%%%%%%%%%%%%%%%
%%%%%%%%%%%%%%%%%%%%%%%%%%%%%%%%%%%%%%%%%%%%%%%%%%%%%%%%%%%%%%%%%%%%%%%%%%%

\section{The 3D-index via normal surfaces}
\label{sect:index_via_normal_surfaces}

We first extend the tetrahedral index $J_\Delta$ to a function
$J: \Z^{3n} \to \Z((q^{1/2}))$ by defining 
\be
J(a_1, b_1, c_1, a_2, b_2, c_2, \ldots , a_n, b_n, c_n) = 
\prod_{j=1}^n \JD(a_j,b_j,c_j),
\ee
and for each $Q$-normal class $S \in Q(\T;\Z)$ we define
\be
I(S)= (-q^{1/2})^{-\chi(S)}  J(S),
\ee
where $\chi$ is the formal Euler characteristic.
Then  $I(S)$ is unchanged by adding tetrahedral solutions:
if $S^* = S + \sum_j m_j T_j$ with $m_j \in \Z$ then
$J(S)=(-q^{1/2})^{\sum_j m_j} J(S^*)$ and 
$-\chi(S)= -\chi(S^*)-\sum_j m_j ,$
so $I(S)=I(S^*)$.
So there is a well-defined function on the quotient group
\be
I: Q(\T;\Z)/\TT \to \Z((q^{1/2})),
\ee
where $\TT= \sum_j \Z T_j \subset Q(\T;\Z)$ is the subgroup generated by 
the tetrahedral solutions.

Taking homology classes of $Q$-normal classes gives a function 
$S \mapsto ([S]_2,[\bd S])$ which also vanishes  on $ \TT$,
so gives a well-defined homomorphism
\be
h : Q(\T;\Z)/\TT \to H_2(M,\bd M;\Z/2\Z) \times H_1(\bd M;\Z),
\ee
and Theorem \ref{ZQNT_struct_thm} shows that
\be
\Im h =\{ (a,b) \in H_2(M,\bd M;\Z/2\Z) \times  H_1(\bd M;\Z) :  
\bd_* a = b \bmod 2\}
\ee
and 
\be
\Ker h = Q_0(\T;\Z)/\TT = (\EE + \TT)/\TT
\ee
where $\EE = \sum_i \Z E_i\subset Q(\T;\Z)$ is the subgroup generated 
by the edge solutions.

\begin{defn}
\label{def.Iab}
For each $(a,b) \in \Im h$ we define an extended version of the 3D-index by 
\be I_\T^a(b) = \sum_{[S] \in h^{-1}(a,b)}  I([S]).
\label{index_from_normal_surfaces}
\ee
\end{defn}

To compute this, we can choose a normal class $S_0 \in Q(\T;\Z)$ with 
$[S_0]_2=a,[\bd S_0]=b$ and choose a set of $n-r$ edge solutions (as 
explained in \cite{GHRS}), say $E_1, \ldots, E_{n-r}$, whose integer 
linear combinations form a complete set of coset representatives for 
$(\EE+\TT)/\TT$. Then
\be 
I_\T^a(b) = \sum_{\kk=(k_1,\ldots,k_{n-r}) \in \Z^{n-r}}  
I(S_0+ \textstyle{\sum_i k_i E_i}).
\label{explicit_index_computation}
\ee

\begin{cor}
\label{cor.Izero}
In particular, the index $I^0_\T(0)$ is a sum over closed normal classes 
in $Q_0(\T;\Z)$ modulo tetrahedral solutions: 
\be
I_\T^0(0) =  \sum_{[Q] \in Q_0(\T;\Z)/\TT}   I([S]) 
=\sum_{\kk=(k_1,\ldots,k_{n-r}) \in \Z^{n-r}}  I(\textstyle{\sum_i k_i E_i}).
\label{index0_from_normal_surfaces}
\ee
\end{cor}

\begin{rmk} 
It follows immediately from the definition that $I_\T^a(-b)=I_\T^a(b)$ 
for all $a,b$.
\end{rmk}

\begin{rmk}
{\bf Note on Notation.} The previous definition of 3D-index from 
\cite{GHRS} only applies to the cases where $a=0$ and 
$b \in H_1(\bd M; 2 \Z)$. We then have $I^0_\CT(b) = I_\CT(b/2)$ in the 
notation of \cite{GHRS}. For the complement $M$ of a knot in $S^3$ with 
standard meridian $\mu$ and longitude $\lambda$, the index $I_\CT(m,e)$ 
in \cite{DGG1} is denoted  $I_\CT(e \mu -m \lambda/2)$ in \cite{GHRS} and 
$I^0_\CT(2e \mu - m\lambda)$ here.
\end{rmk}

\subsection{Convergence of the index sum}

To understand the convergence of the index sum (as a formal Laurent series) 
we need to examine the lowest degree of the terms in this sum.  Given any 
coset $[S]=S+ \TT$, we can choose its ``minimal non-negative coset 
representative'' $S^*$. Explicitly,
given 
$$
S=(a_1, b_1, c_1, a_2, b_2, c_2, \ldots , a_n, b_n, c_n) \in  Q(\T;\Z)
$$
we have
$$
S^*=(a^*_1, b^*_1, c^*_1, a^*_2, b^*_2, c^*_2, \ldots , a^*_n, b^*_n, c^*_n),
$$
where 
$(a_j^*,b_j^*,c_j^*)=(a_j-m_j,b_j-m_j,c_j-m_j)$ and $m_j=\min\{a_j,b_j,c_j\}$.

Now $J(S^*)$ has leading term $q^{\frac{1}{2} \delta(S^*)}$ of degree 
$\frac{1}{2} \delta(S^*)$. So each surface $[S]$ occurring in the index 
sum (\ref{index_from_normal_surfaces}) contributes a term 
$I([S]) = (-q^{1/2})^{-\chi(S^*)}  J(S^*)$ of lowest $q^{1/2}$-degree 

\be
d([S])=d(S^*) =-\chi(S^*) + \delta(S^*)
\label{index_deg_surface}
\ee
and leading coefficient $(-1)^{-\chi(S^*)}$.

\begin{rmk}
Note that $S^*$ need not give an embedded normal surface, but there are 
{\em at most two} non-zero quad coordinates in each tetrahedron. In the 
next section we will show how to replace such a normal class by a unique 
{\em embedded generalised normal surface} $\tilde S$.
\end{rmk}

From this, it is easy to analyse the convergence of the index sums. First 
we give a new proof that the sum $I_\T^0(0)$ converges (as formal Laurent 
series) if and only if $\T$ contains no embedded normal surfaces of 
Euler characteristic $\ge 0$.

\begin{thm} 
\label{index0_cgce}
The index sum for $I_\T^0(0)$  converges (as a formal Laurent series) if 
and only if the triangulation $\T$ is 1-efficient.
\end{thm}

\begin{proof}
($\Rightarrow$) If the triangulation $\T$ is not 1-efficient then it 
contains a closed embedded normal surface  with non-negative Euler 
characteristic. Doubling this, if necessary, gives a closed embedded 
normal surface $S \in Q_0(\T;\Z_+)$ with
$\chi(S)\ge 0$, $\delta(S)=0$ and $S^*=S \ne 0$. 
Then for each integer $k\ge0$, $(kS)^* = k S$ contributes a term to the sum
(\ref{index0_from_normal_surfaces}) for $I_\T^0(0)$
with $q^{1/2}$-degree $d(kS) = -k\chi(S) \le 0$. So the index sum
cannot converge as a formal Laurent series in $q^{1/2}$.
 
($\Leftarrow$) Assume that  $\T$ is a 1-efficient ideal triangulation.  
First we note the following general fact.

\begin{lemma} 
\label{subadditive}
The function $d = -\chi + \delta: Q(\T;\R_+) \to \Z$ is superadditive, i.e.
$$
d(S+S') \ge d(S)+d(S').
$$
Hence $d(\sum_i n_i S_i) \ge \sum_i d(n_i S_i) \ge \sum_i n_i d(S_i)$ for 
all $S_i \in Q(\T;\R_+)$ and all integers $n_i \ge 0$.
\end{lemma} 

\begin{proof}
This follows since $\chi$ is a linear function and $\delta(S+S') 
\ge \delta(S)+\delta(S')$ provided all quad coordinates of $S$ and $S'$ 
are non-negative.
\end{proof}

Next we observe that $N(\T;\Z_+)$ is the set of integer points in a cone 
defined by homogeneous linear inequalities with integer coefficients. 
So we can choose a set of {\em fundamental solutions} 
(or {\em Hilbert basis} in integer linear programming theory) 
$F_1, \ldots,  F_m$  for  $N(\T;\Z_+)$. This means that every 
$S \in N(\T;\Z_+)$ can written (not necessarily uniquely)
as a finite sum $S= \sum_{i=1}^m n_i F_i$ 
where $n_i \in \Z_+$ and each $F_i$ is irreducible, i.e. has no 
decomposition $F_i=x+y$ where $x,y \in N(\T;\Z_+)$ and $x,y \ne 0$.

Now consider $f_i(x) = d(x F_i) = -\chi(F_i)x+\delta(F_i) x^2$, where $x \ge 0$.
Since $\T$ is 1-efficient, for each $i=1,\ldots,m$, we have either
(i) $\delta(F_i)=0$ and $-\chi(F_i) \ge 1$, or (ii) $\delta(F_i) \ge 1$.
Hence $f_i(x) \to \infty$ as $x \to +\infty$,  and in fact there are 
constants $a_i,b_i  \in \R$ with $a_i >0$ such that
$f_i(x) \ge a_i x - b_i$ for all $x \ge 0$.

Then each $S\in N(\T;\Z)$ can be written as a linear combination 
$S= \sum_{i=1}^m x_i F_i$ with $x_i \in \Z_+$
and, using Lemma  \ref{subadditive},  its $q^{1/2}$-degree satisfies 
$$d(S) \ge \sum_{i=1}^m d(x_i F_i) \ge \sum_i (a_i x_i -b_i).$$
Thus, $d(S) \le D$ implies $\sum_i a_i x_i \le D + \sum_i b_i$,
which has only finitely many solutions with $x_i \in \Z_+$. Hence $I_\T^0(0)$  
converges as a formal Laurent series.
\end{proof}

A similar argument gives the following

\begin{thm} 
\label{1-eff_cgce_I_T(b)}
If $\T$ is 1-efficient, then $I_\T^a(b)$ converges for all $a$ and $b$; 
in fact $\sum_a I_\T^a(b)$ is convergent for all $b$.
\end{thm}

\begin{proof} 
We may assume that  $b \ne 0$. Now consider  the set $Q_b(\T;\Z_+)$ of 
Q-normal classes $S \in Q(\T;\Z_+)$ whose boundary $\bd S$ is a 
non-negative integer multiple of $b$. This is the set of integer points 
in a cone defined by a set of homogeneous linear equations and inequalities 
with integer coefficients, so we can choose a finite set $F_1, \ldots, F_m$  
of fundamental solutions for $Q_b(\T;\Z_+)$ where $\bd F_i = s_i b$ with 
$s_i \in \Z_+$.

Given $S \in Q_b(\T;\Z_+)$ we can write $S = \sum x_i F_i$ where 
$x_i \in \Z_+$ and  $\bd S = (\sum s_i x_i )b$.
Let $I:=\{1,2,\ldots,m\}$ and write  $I= I_0  \cup I_1$ where 
$I_0 = \{ i \in I : s_i=0 \}$ and $I_1 = \{ i \in I : s_i \ge 1 \}$.
Then $S = S_0 + S_1$ where $S_0 = \sum_{i \in I_0} x_i F_i$ and 
$S_1 = \sum_{i \in I_1} x_i F_i$ and $\bd S = b$ if and only if 
$\sum_{i \in I_1} s_i x_i = 1$. So $S_1$ belongs to the finite set 
$\{ F_i : i \in I_1 \text{ and } s_i=1\}$, and $S_0 \in  N(\T;\Z)$. 

As in the proof of Theorem \ref {index0_cgce} there are constants 
$a_i,b_i  \in \R$ with $a_i >0$ such that $d(x_i F_i) \ge a_i x_i - b_i$ for 
each $i \in I_1$. Now
$$
d(S) \ge d(S_1)+d(S_0) \ge d(S_1) + \sum_{i \in I_1} d(x_i F_i) \ge 
d(S_1) + \sum_{i \in I_1}(a_i x_i - b_i).
$$
Hence given any $D \ge 0$ there are at most finitely many 
$S \in Q_b(\T;\Z_+)$ with $\bd S = b$ and  $d(S) \le D$. 
This implies the result.
\end{proof}

\begin{thm}
If $\T$ is {\em spun 1-efficient}, i.e. contains no embedded spun normal 
surface $S\ne 0$ with $\chi(S)\ge 0$, then the ``total index'' 
$I^{tot}_\T = \sum_{\text{all } (a,b)}  I_\T^a(b)$ converges.
\end{thm}

\begin{proof} 
Let $F_1, \ldots, F_m$ be a set of fundamental solutions for the set 
$Q(\T;\Z_+)$ of non-negative integer solutions to the Q-matching equations.   

Now consider $f_i(x) = d(x F_i) = -\chi(F_i)x+\delta(F_i) x^2$, where $x \ge 0$.
Since $\T$ is spun 1-efficient, for each $i=1,\ldots,m$, we have either
(i) $\delta(F_i)=0$ and $-\chi(F_i) \ge 1$, or (ii) $\delta(F_i) \ge 1$.
Hence $f_i(x) \to \infty$ as $x \to +\infty$, and in fact there are 
constants $a_i,b_i \in \R$ with $a_i >0$ such that
$f_i(x) \ge a_i x - b_i$ for all $x \ge 0$.

Then each $S\in Q(\T;\Z)$ can be written as a linear combination 
$S= \sum_{i=1}^m x_i F_i$ with $x_i \in \Z_+$
and, using Lemma \ref{subadditive},  its $q^{1/2}$-degree satisfies 
$$
d(S) \ge \sum_{i=1}^m d(x_i F_i) \ge \sum_i (a_i x_i -b_i).
$$
Thus, $d(S) \le D$ implies $\sum_i a_i x_i \le D + \sum_i b_i$,
which has only finitely many solutions with $x_i \in \Z_+$. Hence the total
index sum converges.
\end{proof}

\begin{thm}
Every 1-efficient ideal triangulation of an anannular 3-manifold other 
than the solid torus or solid Klein bottle is spun 1-efficient. 
\end{thm}

\begin{proof} 
This follows by a simple barrier argument. For more details see 
Jaco-Rubinstein \cite{JR} and Section \ref{barriers}. Assume that $\T$ 
is a 1-efficient triangulation of an anannular 3-manifold $M$, which is 
not spun 1-efficient. So there is an embedded spun normal surface $S$ 
with $\chi(S)\ge 0$ in $M$. Clearly then $S$ is a M\"obius band, annulus 
or disk. If $S$ is a disk, $M$ must be a solid torus or solid Klein bottle, 
which are excluded by assumption. So $S$ must be an annulus or M\"obius band. 
On the other hand, if $S$ is an annulus or M\"obius band, since $M$ is 
anannular, $S$ must be topologically boundary parallel, so is contained in 
a collar of a boundary component of $M$. (The M\"obius band case only occurs 
if $M$ is non-orientable and the corresponding cusp is a Klein bottle).
We can then push the appropriate boundary surface across this collar to 
give a torus or Klein bottle $T$ which is boundary parallel. Moreover the 
product region bounded by $T$ and a cusp contains $S$. 

We now wish to construct a barrier using $S$. Choose a normal torus or 
Klein bottle $T_0$ which consists entirely of triangular disks and is 
parallel to the cusp in the previous paragraph. We can assume that $S$ 
and $T_0$ are chosen to intersect transversely.  $S \cup T_0$ separates 
$M$ into various regions, with one such region $R$ homeomorphic to $M$. 
The boundary of $R$ contains a `piecewise normal' torus or Klein bottle 
and this is a barrier for normalisation of surfaces in $R$ in the sense 
of Jaco-Rubinstein \cite{JR}. So we can normalise the essential torus or 
Klein bottle $T_0$ in $R$ to give an embedded normal torus or Klein bottle 
which is not normally peripheral, contradicting  the assumption that $\T$ 
is 1-efficient. 
\end{proof}

\begin{rmk}
A 1-efficient ideal triangulation of an (open) solid torus can contain an 
embedded spun normal disk. In fact this happens in 
Example \ref{solid_torus_ex} below. 
\end{rmk}

A discussion of angle structures and angle structure with rotational 
holonomy zero on each peripheral curve follows in 
Appendix \ref{gen_angle_euler}.

\begin{cor}
If $\T$ admits a strict angle structure with rotational holonomy zero on 
each peripheral curve, then it is spun 1-efficient. In particular, this 
applies to any geometric triangulation of a cusped hyperbolic 3-manifold. 
\end{cor}

\begin{proof}
The Euler characteristic of any embedded spun normal surface $S$ can be 
calculated as $\sum_{q \in \quads} -\alpha(q)/\pi$,  if 
$\alpha: \quads \to (0,\pi)$ is an generalised angle structure 
with rotational holonomy zero on each peripheral curve. If $\alpha$ is a 
strict angle structure then the sum is negative, unless $S=0$.
\end{proof}

\begin{ex}
The 2-tetrahedron triangulation $\T$ of the trefoil knot complement in 
Example \ref{trefoil_ex}, is not spun 1-efficient. Here 
$I^0_\T(x \mu + y \lambda) =  \delta_{0,x+6y}$ for 
$x \in \Z, y \in \frac{1}{2} \Z$ so the total index 
$\sum_{x,y} I^0_\T(x \mu + y \lambda)$ diverges.
\end{ex}

\subsection{Non 1-efficient ideal triangulations}

When the triangulation $\T$ is not 1-efficient, it turns out that the 
series $I_\T^a(b)$ may converge for some $b \ne {\bf 0}$
even when the series $I_\T^0(0)$ diverges.  See Example \ref{toroidal_ex} 
below. The following result gives an obstruction to convergence.

\begin{lemma} 
Assume that an ideal triangulation $\T$ contains a Q-normal class 
$S \in Q(\T;\Z_+)$ with $[S]_2=a$ and $\bd S$ satisfying $[\bd S]=b$,
and there is a closed normal surface $0 \ne S' \in N(\T;\Z_+)$ such that 
$S'$ is embedded and $-\chi(S')+ 2\delta(S,S') \le 0$.
Then the index sum for $I_\T^a(b)$ diverges.
\end{lemma}

\begin{proof}
By replacing $S'$ by $2 S'$, if necessary, we can assume that 
$S' \in Q_0(\T;\Z_+)$. Then for all $k \in \Z_+$, the normal class 
$S_k=S+k S'$ contributes a term in the index sum 
(\ref{explicit_index_computation}) for $I_\T^a(b)$ with $q^{1/2}$-degree 
\begin{align}
d (S_k) &= -\chi(S+k S') + \delta(S+k S') \nonumber \\
&= -\chi(S)-k\chi(S') +\delta(S)+2k \delta(S,S')+k^2\delta(S') \nonumber \\
& = d(S)+k\left( -\chi(S')+2 \delta(S,S')\right)+k^2\delta(S').
\label{degree_S_k}
\end{align}
Hence $d(S_k)$ remains bounded above as $k \to +\infty$ if $\delta(S')=0$ 
and $-\chi(S')+2 \delta(S,S')\le 0$. Thus the index sum diverges.
\end{proof}

In general, we have the following converse.

\begin{thm} 
\label{cgce_I_T(b)} 
Assume that an ideal triangulation $\T$ does not contain a Q-normal class 
$S \in Q(\T;\Z_+)$ with $\bd S = b$, and a closed normal surface 
$S' \in N(\T;\Z_+)$ such that $S'$ is embedded and 
$-\chi(S')+ 2\delta(S,S') \le 0$. Then the index sum $I_\T^a(b)$ converges.
\end{thm}

\begin{proof}
We start by following the proof of Theorem \ref{1-eff_cgce_I_T(b)}. This 
shows that there exist finitely many normal classes
$S_1, \ldots, S_k \in Q(\T;\Z_+)$ with $\bd S_i=b$ and finitely many 
normal classes  $F_1, \ldots, F_m \in Q(\T;\Z_+)$ with $\bd F_i=0$ 
such that every $S \in Q_b(\T;\Z_+)$ with $\bd S = b$  can be written as
$$
S= S_j + \sum_{i=1}^m x_i F_i, \text{ where } 1\le j \le k 
\text{ and } x_i \in \Z_+.
$$

Now 
\begin{align*}
d(S) 
&= d(S_j)+ \sum_{i} (-\chi(F_i)+ 2 \delta(F_i,S_j))x_i 
+ \sum_{i ,j } \delta(F_i,F_j) x_i x_j \\
&\ge  d(S_j)+ \sum_{i} (-\chi(F_i)+ 2 \delta(F_i,S_j))x_i  
+ \delta(F_i) x_i ^2 .
\end{align*}
For each $i \in I_0$ we have, by assumption,  either $\delta(F_i)\ge 1$ 
or $-\chi(F_i)+ 2 \delta(F_i,S_1) \ge 1$ so it follows that 
$$
f_i(x):=(-\chi(F_i)+ 2 \delta(F_i,S_1))x  + \delta(F_i) x^2
$$
is bounded below for $x \ge 0$ and approaches 
$+\infty \text{ as } x\to +\infty.$
Since  there are only finitely possibilities  for $S_1$,
it follows that there are only finitely
many $S \in Q_b(\T;\Z)$ such that $\bd S = b$ and $d(S) \le D$. Hence the 
index sum for $I_\T^a(b)$ converges.
\end{proof}

\medskip

\begin{rmk}
\begin{enumerate}
\item 
When $b=0$ this reduces to Theorem \ref{index0_cgce}, as we can take $S=0$ 
in the above theorem.
\item 
If the triangulation $\T$ is {\em $0$-efficient} then 
Theorem \ref{cgce_I_T(b)} simplifies to the following: \\
$I_\T^a(b)$ converges if and only if $\T$ 
does not contain a Q-normal class $S \in Q(\T;\Z_+)$ with $\bd S = b$,  
and an \emph{embedded} closed normal surface $S'$ disjoint from $S$ such 
that  $\chi(S') =0$.
\end{enumerate}
\end{rmk}

%%%%%%%%%%%%%%%%%%%%%%%%%%%%%%%%%%%%%%%%%%%%%%%%%%%%%%%%%%%%%%%%%%%%%%%%%%%
%%%%%%%%%%%%%%%%%%%%%%%%%%%%%%%%%%%%%%%%%%%%%%%%%%%%%%%%%%%%%%%%%%%%%%%%%%%

\section{Invariance under 2-3 and 0-2 moves}  

The arguments of \cite{GHRS} extend easily to prove the following. 

\begin{thm} 
Let $\CT$ be an ideal triangulation of $M$ and let $\tilde \CT$ be obtained 
from $\CT$  by a 2-3 move. Then for all $a,b$ as above,
$I_{\CT}^a(b)=I_{\tilde\CT}^a(b)$ provided both sides are defined.
\end{thm} 

\begin{proof} 
We follow the argument from \cite[Theorem A.1]{GHRS}, indicating the main 
changes needed in the current setting. Consider a 2-3 move from $\CT$ to 
$\tilde \CT$; this occurs in a bipyramid consisting of $2$ tetrahedra in 
$\CT$ which are replaced by 3 tetrahedra sharing an edge of order 3 in 
$\tilde \CT$.

First note that each $Q$-normal class $S \in Q(\CT;\S)$  can be written 
as a linear combination of edge solutions, tetrahedral solutions and 
peripheral curve solutions with coefficients in $\frac{1}{2} \Z$, by the 
results of Section 6 (see Remark \ref{half-integer-coefs}).
As in section 4.4 of \cite{GHRS}, we can represent $S$ by a linear 
combination $\omega$ of oriented normal arcs in the induced triangulation 
$\T_\bd$ of $\bd M$, chosen so that each quad coordinate for $S$ is a sum 
of 4 ``turning numbers'' of normal arcs, coming from the 4 corners of the 
tetrahedron containing the quad (see figure 7 of \cite{GHRS}).

Explicitly, for each edge solution we choose a small normal linking
circle in  $\bd M$ around one end of the edge in $\T$,  we choose oriented 
normal curves representing each peripheral class, and for each tetrahedral 
solution we choose three normal arcs  oriented anticlockwise in a triangle 
coming from one corner of the tetrahedron. (If tetrahedral solutions are 
needed in the bipyramid, we choose the two triangles at the degree 3 
vertices at the top and bottom.) One difference from \cite{GHRS}  is that 
the coefficients of the normal arcs here can be {\em half-integers}; 
but each sum of turning numbers giving a quad coefficient is still an 
integer. 

To compute the index $I_\T^a(b)$, we choose a normal class $S_0$ with 
$[S_0]_2=a$ and $[\bd S_0]=b$, and have a contribution from each normal class
$$
S_{\kk} = S_0 + \sum k_i E_i,
$$ 
where $\kk=(k_i) \in \Z^{n-r}$ is a set of integer weights on a set
of $n-r$ ``basic edges''.  (The excluded edges are taken from a 
``maximal tree with 1- or 3- cycle for $\T$'' as in Theorem 4.3 
of \cite{GHRS}; this ensures that the map 
$\kk=(k_i) \mapsto \sum k_i E_i + \TT$ 
is an isomorphism from $\Z^{n-r}$ onto $(\EE+\TT)/\TT$.)
We can represent $S_0$ by a linear combination $\omega$ of oriented 
normal arcs as described above.
By adding $k_i$ small linking circles around one end of the $i$th edge 
in $\T$ we obtain a linear combination $\omega_{\kk}$ of oriented normal 
arcs which gives the 
quad coordinates of $S_{\kk} = S_0 + \sum k_i E_i$.

The construction of \cite{GHRS} shows how to replace $\omega$ by a linear 
combination $\tilde\omega$ of oriented normal arcs in $\tilde\T_\bd$ 
representing a normal class $\tilde S_0$ in $\tilde \T$ and this satisfies 
$[\tilde S_0]_2=[S_0]_2$ and $[\bd S_0] = [\bd \tilde S_0]$.   
(Essentially, $\tilde S_0$ is obtained by subdividing the normal disks of 
$S_0$, so its homology classes do not change.) 
Then the index sum for $I_{\tilde \T}^a(b)$ has a contribution for each 
normal class 
$$
\tilde S_{\tilde \kk} = \tilde S_0 + k_0 \tilde E_0 + \sum k_i \tilde E_i,
$$
where $k_0$ is the integer weight on the new edge class $\tilde E_0$ in 
$\tilde \T$, and $\kk=(k_i)\in\Z^{n-r}$ as above give the weights
on the edge classes $\tilde E_i$ coming from $\T$.
Adding small linking circles around one end of each  edge in $\T$ gives a 
linear combination of normal arcs  $\tilde\omega_{\tilde \kk}$
which gives the quad coordinates of $\tilde S_{\tilde \kk}$.

The proof of the Pentagon Equality (Lemma A.3 in \cite{GHRS}) now goes 
through verbatim; this gives the result. The only difference in the 
current setting is that coefficients of normal arcs and turning numbers 
can now be half-integers. However since the sums giving quad coefficients 
used in the tetrahedral index functions are integers,
the arguments of \cite{GHRS} go through without change.  
\end{proof}

The arguments from \cite{GHRS} also extend to give the following.

\begin{thm} 
\label{0-2invariance}
Let  $\CT$ be an ideal triangulation of $M$ and let $\tilde \CT$ be 
obtained from $\CT$  by a 0-2 move. Then for all $a,b$ as above,
$I_{\CT}^a(b)=I_{\tilde\CT}^a(b)$ provided both sides are defined.
\end{thm} 

\begin{proof}
Here the proof of Theorem 5.1 in  \cite{GHRS} applies, if we start with a 
linear combination $\omega$ of oriented normal arcs in $\T_\bd$
with $\frac{1}{2} \Z$ coefficients  giving the quad coordinates of $S_0$ 
with homology classes $[S_0]_2=a$ and $[\bd S_0]=b$.
\end{proof}

%%%%%%%%%%%%%%%%%%%%%%%%%%%%%%%%%%%%%%%%%%%%%%%%%%%%%%%%%%%%%%%%%%%%%%%%%%%
%%%%%%%%%%%%%%%%%%%%%%%%%%%%%%%%%%%%%%%%%%%%%%%%%%%%%%%%%%%%%%%%%%%%%%%%%%%

\section{Generalised normal surfaces}
\label{gen_normal_surfaces}

{\em Embedded normal surfaces} were introduced by Kneser \cite{Kn} in 1929. 
In the 1950s and 1960s, Haken \cite{Ha1,Ha2,Ha3} extensively developed the 
theory of normal surfaces (via handle decompositions and also via 
triangulations) and applied this to basic algorithmic questions in topology.  
The solutions to Haken's matching equations include {\em branched} and 
{\em immersed} normal surfaces, which are the result of gluing together 
finitely many normal disks in tetrahedra of a triangulation of a closed 
manifold or  an ideal triangulation of a compact manifold with boundary.
In the 1980s  Thurston suggested the theory of 
{\em spun normal surfaces }(see \cite{Ti,Wa}). These are obtained by 
gluing finitely many quadrilaterals but possibly infinitely many 
triangular disks, as long as there are at most a finite number of such 
disks which are not contained in `tails' which are infinite annuli 
spiralling around a cusp. We will again specify if a spun normal surface 
is embedded --- in general it can be branched or immersed. 
Moreover spun normal surfaces may contain no tails, in which case they are 
(closed) normal surfaces. 

In this section we discuss {\em generalised} normal surfaces and 
normal classes. The former have been called $k$-normal surfaces, 
spinal surfaces and helicoidal surfaces in the literature 
(see \cite{Ma2,FoMa,FrKa,Ba}). 
Our aim is to show the interesting connections both with normal surface 
theory and with the duality between homology and cohomology. 
Moreover the lattices over which the 3D-index is evaluated can be 
represented as collections of embedded generalised normal surfaces, 
lying in a particular $\Z/2\Z$-homology class. Finally generalised normal 
surface theory can be viewed as a projection of normal surface theory, 
with the subspace of solutions spanned by the tetrahedral solutions 
quotiented out. 

\subsection{Generalised normal and spun normal surfaces}

Firstly, we discuss normal curve theory (with real coefficients) 
on the 2-sphere equipped with its four simplex triangulation as the boundary 
of a tetrahedron $\Delta^3$. We just recall the key points 
(see \cite[Section 3.2]{Ma2} for a detailed exposition).  

\begin{itemize}
\item 
There are $12$ normal arc types in $\partial \Delta^3$ and so normal 
curves are defined by non-negative integer vectors in $\R^{12}$ satisfying 
$6$ matching equations. 
It is easy to see that edge weights of normal curves determine the numbers 
of normal arcs and there are six independent edge weights. 
Hence the solution space $W(\R)$ to the matching equations is $6$-dimensional. 
\item 
For the integer solution space $W(\Z)$ to the matching equations, 
the only additional constraints on the edge weights are mod 2 conditions. 
Namely, the three edges weights at a face are either all even or two are 
odd and one even. 
\item 
There are precisely $7$ vertex classes of the projective solution space 
$\calP$, which is a $5$-dimensional polytope. 
These vertices are the classes of the $4$ boundaries of triangular normal 
disks and the $3$ boundaries of quadrilateral normal disks. 
They satisfy precisely one relation: that the sum of the triangular curves 
is the sum of the quadrilateral curves. 
\item 
There are $21$ facets of $\calP$ which are all $4$-simplices. Namely, there 
are $12$ facets containing $3$ quadrilateral curves and $2$ triangular 
curves as vertices, $3$ facets containing $4$ quadrilateral curves and 
$1$ triangular curve, and $6$ facets containing $2$ quadrilateral curves 
and $3$ triangular curves. 
\end{itemize}

Next, we want to analyze the quotient space $N(\T;\Z)/\TT$ of integer 
normal classes modulo tetrahedral solutions, working in quadrilateral space. 

For each tetrahedron, this is equivalent to studying the normal arc 
solution space $W(\Z)$ modulo the $\Z$-submodule spanned by the $4$ 
triangular curves in the boundary of the tetrahedron. This is a 
2-dimensional space generated by the coset representatives of any 2 
of the 3 quadrilateral curves. So elements can be viewed as integer 
linear combinations of any 2 quadrilateral curves.

We can represent the tetrahedron as a square pillowcase, triangulated 
with 2 triangles on the top and 2 triangles on the bottom, and choose 
normal coordinates so that the 3 quadrilateral curves correspond to 
$(1,0), (0,1), (1,1) \in \Z^2$. Now it is easy to see what happens when 
we take a linear combination $p(1,0)+q(0,1)=(p,q)$ where $p,q \in \Z_+$. 
If $p$ and $q$ are relatively prime, then the $(p,q)$ normal class can 
be represented by a simple closed geodesic of slope $q/p$ 
on a square pillowcase with edge weights $p,q,p+q$ as shown in 
Figure \ref{square_pillow} below. In general, if $p,q$ are not relatively 
prime, we get $d$ parallel copies of a simple closed normal curve where 
$d = \gcd(p,q)$. 
 
\begin{figure}[htpb]
\begin{center}
\includegraphics[width=0.3\textwidth]{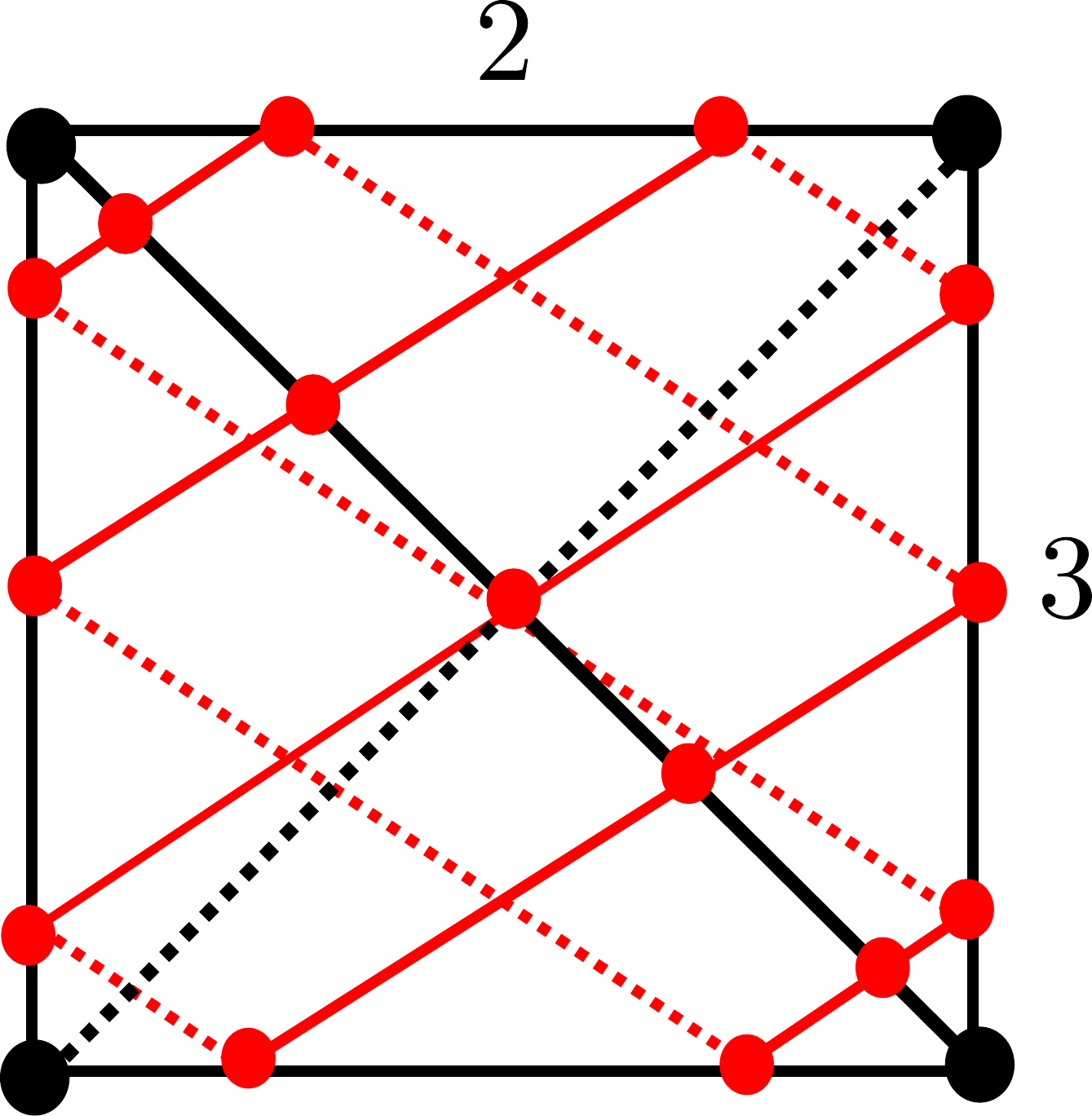}
%\resizebox{6cm}{!}{\input{pillowcase.pdf_tex}}
\end{center}
\caption{A $(3,2)$ normal curve on the square pillowcase, with edge 
weights $3,2$ and $5$ along the tetrahedron edges.}
\label{square_pillow}
\end{figure}

Alternatively, if we lift to the 2-fold branched covering of the 2-sphere 
over the 4 vertices, the result is a 2-torus. Then the $(p,q)$ normal 
curve lifts to simple closed $(p,2q)$ curves on the torus.

\begin{defn}
A {\em generalised normal disk} in a tetrahedron is a properly embedded 
disk whose boundary is an embedded simple closed normal curve.
\end{defn}
  
\begin{rmk} 
These generalised normal disks arise naturally in the theory of Dehn 
filling of tangles, introduced by Montesinos in \cite{Mont}. Namely, 
the generalised normal disks are the disks that separate the two strands 
of a rational tangle. 
\end{rmk}

\begin{defn}
A {\em generalised normal surface} in an ideal triangulation $\T$ of a 
compact 3-manifold is an embedded surface which meets each tetrahedron of 
$\T$ in a finite collection of disjoint generalised normal disks.
\end{defn}
 
\begin{rmk}
Special cases of this notion have been considered previously, for example 
in \cite{Ma2,FoMa,FrKa,Ba}. 
\end{rmk}

\begin{lemma}
\label{gen_closed_class}
Given a closed normal class $S$ in $\tilde N(\T;\Z_+)$ with at most two 
non-zero quadrilateral coordinates in each tetrahedron, there is an 
embedded generalised normal surface $\tilde S$ with the same edge weights, 
which is unique up to normal isotopy.
\end{lemma}

\begin{proof} 
This is essentially the same argument as the construction of a closed 
normal surface from its triangle and quad coordinates
(see for example \cite{Ma2}).
Any solution $S \in \tilde N(\T;\Z_+)$  to Haken's matching equations 
gives an embedded collection of 
normal arcs in the 2-skeleton of the triangulation which is unique up to 
normal isotopy: 
first choose $w(e)$ disjoint points in the each edge $e$ where the ``weight'' 
$w(e)$ is the number of normal disks meeting edge $e$, 
and extend this to an embedded collection of normal arcs in each triangle. 
The arc pattern on the boundary of each tetrahedron can then can be filled 
in uniquely (up to normal isotopy) 
by adding generalised normal disks  and normal triangles. The matching 
equations imply that these fit together
to give a generalised normal embedded surface.
\end{proof}

We can also consider {\em generalised spun normal surfaces} in an ideal 
triangulation $\T$, which meet each ideal tetrahedron in a finite 
collection of disjoint generalised normal disks and possibly infinitely 
many triangular normal disks.

\begin{lemma}
\label{gen_spun_class}
Given a Q-normal class $S$ in $Q(\T;\Z_+)$ with at most two non-zero 
quadrilateral coordinates in each tetrahedron, there is a unique embedded 
generalised spun normal surface $\tilde S$ such that the edge weights
of the non-triangular disks in $S$ and $\tilde S$ agree.
\end{lemma}

\begin{proof}
This is analogous to the construction of a spun normal surface from a 
solution $S$ to the Q-matching equations, as described in \cite{Ti} 
(see also \cite{DuGa}). We first choose $w(e)$ disjoint points in the 
each edge $e$ where $w(e)$ is the weight of quadrilateral normal disks 
of $S$ meeting edge $e$, and extend this to an embedded collection of 
normal arcs in each triangle. Next add  infinitely many parallel arcs 
at each corner of each triangle.  The arc pattern on the boundary of 
each tetrahedron can then can be filled in uniquely (up to normal isotopy) 
by adding generalised normal disks  and infinitely many normal triangles.
There is a unique way to glue together the faces of tetrahedra in pairs 
in the given combinatorial pattern so the normal arcs match, 
and the $Q$-matching equations show that the generalised normal disks fit 
together consistently (without any ``shearing'' ) around each edge class. 
Removing any boundary parallel normal surface components  gives the 
desired embedded generalised spun normal surface $\tilde S$.
(If $\bd S =0$, the result is an embedded closed normal surface.)
\end{proof}

\begin{rmk}
In the above two lemmas, $\tilde S$ is connected if the class $S$ is 
irreducible.
\end{rmk}

The normal classes needed for the summation to get the 3D-index correspond 
to embedded generalised normal surfaces. The correspondence is unique in 
the sense that the process produces a unique embedded generalised normal 
or spun normal surface with at most two non-zero quadrilateral coordinates 
in each tetrahedron. This gives bijections
$$
[S]=S+\TT \in Q(\T;\Z)/\TT  \leftrightarrow  S^* \in Q(\T;\Z_+) 
\leftrightarrow \tilde S \in 
\{ \text{embedded generalised spun normal surfaces} \}
$$

\subsection{Degrees of index terms}

Recall, from equation (\ref{index_deg_surface}), that the $q^{1/2}$-degree 
of a term $I([S])$  in the $q$-series for the 3D-index is given by the 
expression 
$$
d(S)=d(S^*) =-\chi(S^*) +\delta(S^*),
$$
where $S^*$ is the minimal positive coset representative for $[S]=S+\TT$. 
Then $S^*$ is a normal class with at most two quadrilateral types in each 
tetrahedron, and 
$$
d(S^*)=-\chi(S^*) +\sum_i p_iq_i,
$$
where $p_i,q_i\ge0$ are the numbers of quadrilaterals for $S^*$ in the 
$ith$ tetrahedron, and $\chi(S^*)$ denotes the formal Euler characteristic 
of $S^*$.

\begin{lemma} 
\label{deg_for_gen_norm_surf}
Let $S \in Q(\T;\Z)$ be a Q-normal class 
and let $\tilde S$ be the embedded generalised normal surface 
$\tilde S$ corresponding to $S^*$. Then
$$d(S)= -\chi(\tilde S)+ \sum_i (p_i q_i - p_i - q_i + \gcd(p_i,q_i)),$$
where $p_i,q_i\ge0$ are the numbers of quadrilaterals for $S^*$ in the 
$ith$ tetrahedron.

Hence
$d(S) \ge -\chi(\tilde S)$ with equality if and only if $\tilde S$ is a  
``special'' embedded generalised normal surface, i.e. in each tetrahedron one 
of the coefficients $p_i$ or $q_i$ is either $0$ or $1$. 
\end{lemma}

\begin{proof} 
We can choose a generalised angle structure on $\T$ with vanishing 
rotational peripheral holonomy by 
Proposition \ref{gen_ang_zero_periph_rot_hol},  and use this to compare 
the Euler characteristic of $\tilde S$ 
with the formal Euler characteristic of $S^*$, using 
Proposition \ref{comb_Gauss-Bonnet} and 
Definition \ref{formal_euler_from_ang_str} from Appendix \ref{gen_angle_euler}.

In both cases we get corner terms of $(\text{formal angles})/2\pi$ which 
add to $1$ for each vertex. 
Moreover the number of edges in $\tilde S$ is the same as for $S^*$, but 
the number of disks is reduced from $\sum_i (p_i+q_i)$ to 
$\sum_i \gcd(p_i,p_i)$.  So
$\chi(\tilde S)= \chi(S^*) + \sum ( \gcd(p_i,q_i) - p_i -q_i )$. Hence
$d(S^*) = - \chi(S^*) + \sum_i p_i q_i = -\chi(\tilde S)
+ \sum_i (p_i q_i - p_i - q_i + \gcd(p_i,q_i)).$

\smallskip
The final result now follows from the observation that for any integers 
$p, q \ge 0$ we have
\begin{itemize}
\item $pq - p - q + \gcd(p,q) = 0$ if $p$ or $q$ is $0$, and otherwise
\item $pq - p - q + \gcd(p,q) \ge pq-p-q+1 =(p-1)(q-1) \ge 0$ with equality 
if and only if $p$ or $q$ is $1$.
\end{itemize}
\end{proof}

\begin{rmk}  
A very similar function is studied in \cite[Section 4.1]{FrKa}.
\end{rmk}

\begin{rmk} 
We will be particularly interested in terms in the 3D-index sum which 
have non-positive degrees. 
\end{rmk}

\subsection{Normalisation and barriers}
\label{barriers}

There is a well-known procedure going back to Kneser \cite{Kn} for 
`normalisation' of an embedded closed surface $S$ in a triangulation $\T$ 
of a manifold $M$. We can assume initially that the surface $S$ is 
transverse to the triangulation.There are five basic moves: 

\begin{enumerate}
\item 
Compress $S$ along an embedded disk $D$ in the interior of a tetrahedron 
$\D$, which meets $S$ in $C=\partial D$, where  either $C$ is essential in 
$S$ or a disk bounded by $C$ in $S$ meets $\partial \D$. 
\item 
Isotope a disk $D^\prime$ of intersection of $S$ with a tetrahedron across a 
face of the tetrahedron, where $\partial D^\prime$ lies in a face and is an 
innermost curve in that face. 
\item 
Isotope an innermost arc of intersection of $S$ with a face of the 
triangulation across an edge. Here the arc has both ends on the edge and 
there are no intersections of $S$ with the bigon cut off by the arc in 
the face. 
\item 
Boundary compress the intersection of $S$ with a tetrahedron. This occurs 
along a bigon which intersects $S$ in one arc of its boundary with the 
other arc contained in an edge of the tetrahedron. 
\item
Finally, any component which is a sphere contained in the interior of a 
tetrahedron is removed.
\end{enumerate}

Kneser \cite{Kn} and Haken \cite{Ha1} show that any surface can either be 
{\em normalised}, i.e. either converted to a non-empty normal surface 
by a sequence of such moves,  or else converted to the empty surface.

\begin{rmk}
If an embedded sphere $S$ or an embedded torus $T$ is converted to the 
empty surface when normalised in one of its complementary regions,
then that region is a 3-ball or solid torus in $M$ bounded by $S$ or $T$ 
respectively. 

In fact, if $S$ or $T$ did not bound a 3-ball or solid torus respectively 
then there is no isotopy or a disk compression followed by an isotopy, 
shrinking $S$ or $T$ respectively to a point. 
\end{rmk}

\begin{defn}
An embedded surface $S^\prime$ is a {\it barrier} for the normalisation of 
$S$ if it satisfies the following conditions. Firstly, 
$S^\prime \cap S = \emptyset$. Secondly, $S$ can be normalised by moves 
in the complement of $S^\prime$. 
\end{defn}

\begin{thm} 
Let $\T$ be an ideal triangulation  of compact manifold with boundary 
consisting of tori. Suppose that $S^\prime$ is a ``simple'' embedded 
generalised normal surface in $\T$, i.e. has at most one generalised 
normal disk in each tetrahedron. Then $S^\prime$ is a barrier for 
normalisation of any surface in its complement. 
\end{thm}

\begin{proof}
Consider the two regions $R, R^\prime$ obtained by splitting a tetrahedron 
$\D$ open along a generalised normal disk $D$ in $S'$. 
The faces of these regions consist of a copy of $D$ and the result of 
splitting the faces of $\D$ open along $\partial D$. 
The latter produce faces which are either triangles, quadrilaterals or 
pentagons. (Hexagons cannot arise since a generalised normal disk has at 
most two different arc types in any triangular face of $\D$.)

Choose a maximal family of disjoint non parallel properly embedded bigon 
compressing disks for each of the regions $R,R'$, as in the fourth 
normalisation move.  Thus, a quadrilateral face  of $R$ or  $R^\prime$ 
has two associated bigon disks, one along each edge of $\D$ in its 
boundary, whereas a pentagon has one such bigon. (Figure \ref{fig:12gon} 
exhibits the case where $D$ is a 12-gon, with $\bd \D$ stereographically 
projected to the plane.)

\begin{center}
\begin{figure}
\includegraphics[width=0.3\textwidth]{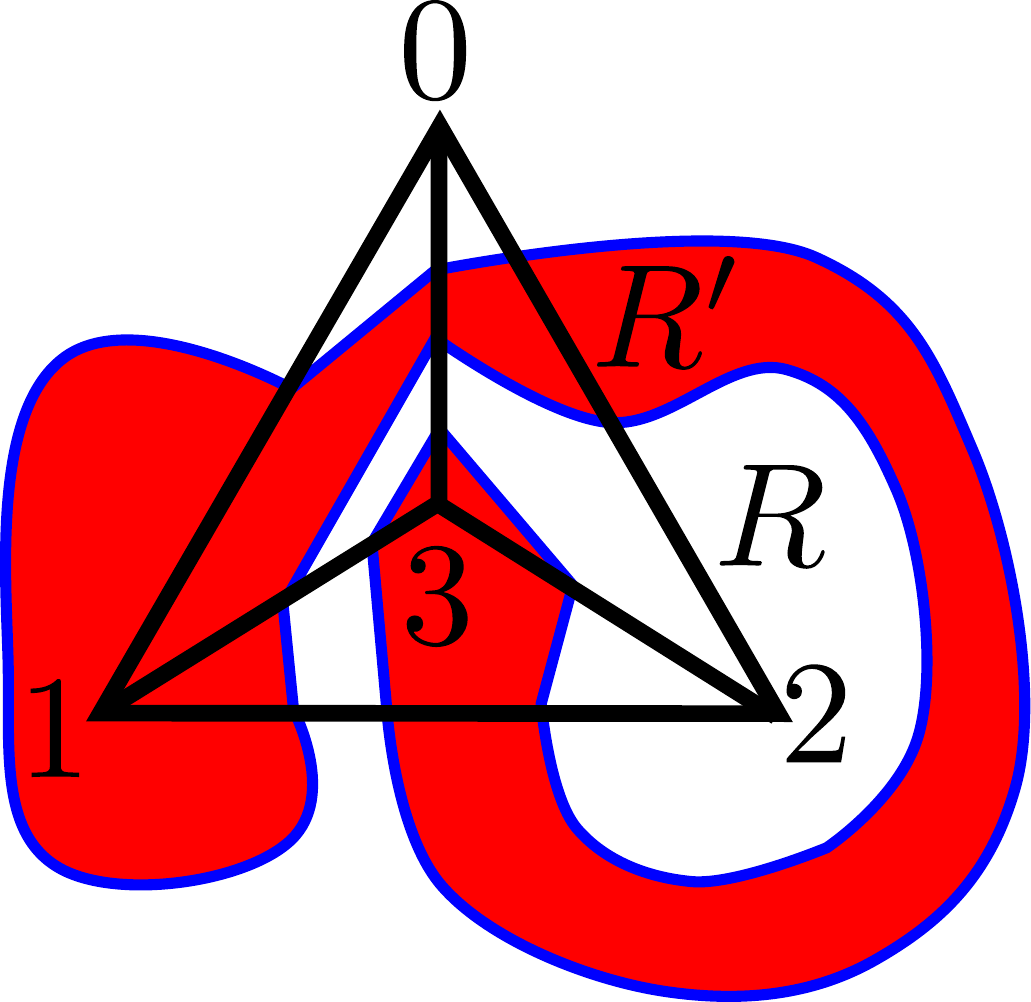} \qquad\qquad
\includegraphics[width=0.3\textwidth]{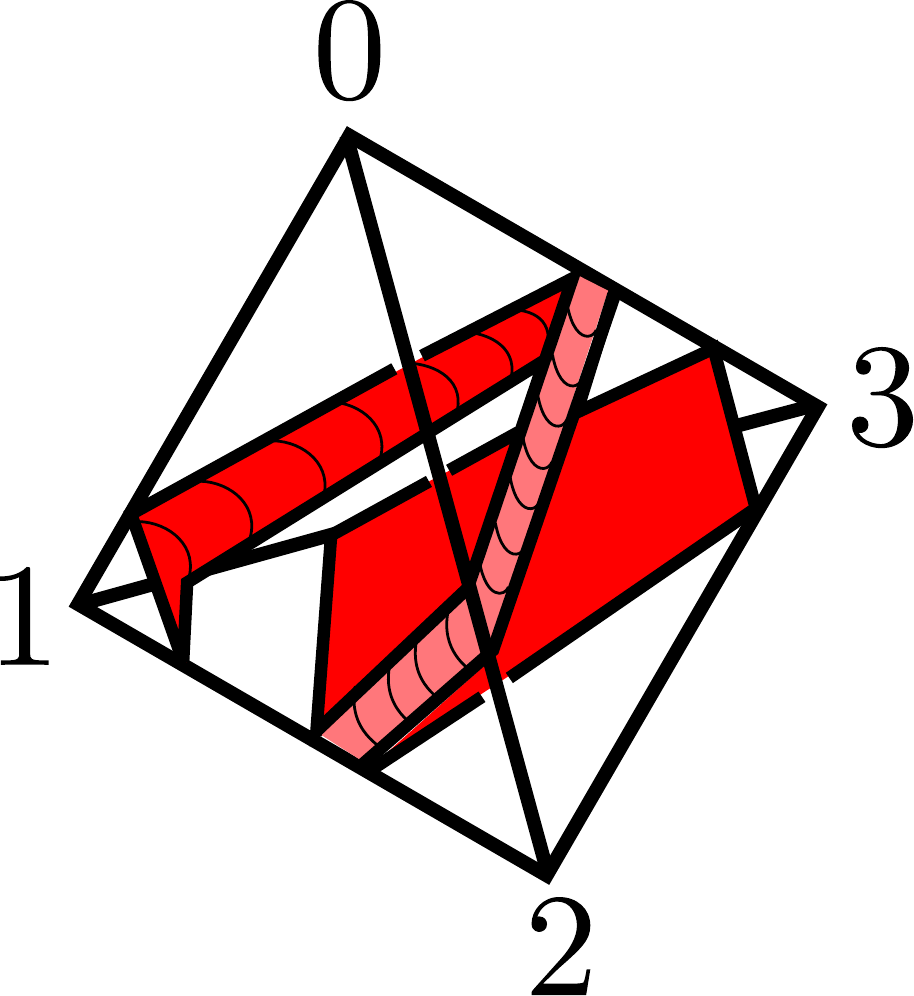}
%\resizebox{6 cm}{!}{\input{12gon.pdf_tex}}
%\resizebox{6 cm}{!}{\input{12gon_in_tet.pdf_tex}}
\caption{\label{fig:12gon} (left) A schematic diagram of the 12-gon 
determined by coordinates $(2,1,0)$. (right) An embedding of this 12-gon. 
The region $R$ faces the dark (red) shaded side, while $R'$ faces the 
lighter side.}
\end{figure}
\end{center}

If either region is compressed along all these disks, each of the 
resulting components is either a tetrahedron, where the compression of $D$ 
gives a face of the tetrahedron, or a 3-ball with boundary a pair of disks, 
where one disk is the result of compressing $D$ and the other disk is in 
a face of $\D$. 

Suppose that $S$ is a surface disjoint from $S^\prime$. If $S$ meets either 
region $R, R^\prime$, then notice that the intersection of $S$ with the 
above families of bigon disks can be used to perform normalisation moves 
on $S$ in the complement of $S^\prime$. For if there is a loop $C$ in 
$S \cap  D_i$ where $D_i$ is one of the bigon disks and either $C$ is 
essential on $S$ or a disk bounded by $C$ in $S$ meets $\partial \D$, 
then this gives a first normalisation move.  On the other hand, if there 
are arcs in $S \cap  D_i$, these give third or fourth normalisation moves. 

After performing all such normalisation moves, the resulting surface 
again denoted by $S$ is still disjoint from $S^\prime$ 
but also from all the bigon disks in  $R, R^\prime$. It is easy to see that 
we can perform additional first or second normalisation 
moves  so that $S \cap  \D$ consists entirely of triangular normal disks 
and spheres in interior $\D$ and after discarding any 
sphere components  we have normalised $S$ in $\D$. 

This argument can clearly be performed in all tetrahedra containing a 
generalised normal disk, and in tetrahedra containing only normal disks, 
the argument in \cite{JR} applies. This completes the proof. 
\end{proof}

\begin{cor}
\label{barrier}
If $S^*$ is a two-sided embedded ``simple'' generalised normal surface 
(i.e. has at most one generalised normal disk in each tetrahedron), then 
it can be pushed off itself to either side and normalised in its complement.
\end{cor}

\begin{rmk}
\label{surface_dies}
If we consider an arbitrary generalised normal surface $S^*$ which is not 
necessarily simple, the same idea shows that any surface $S$ in the 
complement of $S^*$ can be transformed into a {\em generalised} normal 
surface, by moves similar to the ones above for normalisation. 
The difference in this case is that if $S^*$ has parallel generalised 
normal disks bounding a product region $R$ in a tetrahedron and $S$ meets 
$R$,  then $S$ might be transformed to intersect $R$ in 
parallel copies of these generalised normal disks. 
\end{rmk}

\begin{thm}
\label{euler_for_gen_norm_surf}
Let $\T$ be a 1-efficient ideal triangulation of a cusped, orientable 
3-manifold $M$, and assume that $M$ is not a solid torus or $T^2 \times I$.
Then $\T$ contains no closed embedded simple generalised normal surface 
$\tilde S$ with $\chi(\tilde S) \ge 0$.
\end{thm} 

\begin{proof} 
Recall that $\T$ 1-efficient implies that $M$ is irreducible and atoroidal. 
Suppose first that $\tilde S$ is two-sided in $M$. Now by the barrier 
argument in Corollary \ref{barrier}, we can push the simple generalised 
surface $\tilde S$ off itself on either side, and normalise it. The result 
is either (i) a normal sphere, or (ii) a normal torus, or else
(iii) the surface becomes empty. 
 
The first case clearly contradicts the assumption of 1-efficiency. 

In the second case, note that the surface may normalise to a peripheral 
torus at a cusp when it is pushed off itself to a side containing a cusp. 
Since we are assuming that $M$ is not $T^2 \times I$, it follows that 
normalisation to a peripheral torus cannot occur on both sides of a 
generalised normal torus. But then, pushing off to at least one side, we 
would get a normal torus which is topologically but not normally parallel 
into a cusp, again contradicting 1-efficiency.  

So we are left with the case that the surface becomes empty. If $\tilde S$ 
is a sphere, it cannot bound a ball on both sides 
and hence yields a normal sphere on one side or the other by 
Remark \ref{surface_dies}, contrary to assumption. 

If $\tilde S$ is a torus, it could be compressible. In this case, the 
torus bounds a solid torus or cube-with-knotted-hole. 
In the latter case, the sphere resulting from compressing $\tilde S$ 
does not bound a ball in the complement of $\tilde S$ 
so can be normalised giving a normal sphere, contrary to assumption. 
So we are left with the final case, where $\tilde S$  is a torus bounding 
a solid torus. The other side of $\tilde S$ must contain a cusp of $M$. 
But then if $\tilde S$ is compressible on this side, 
the result is a sphere which can be normalised, contradicting the 
assumption of 1-efficiency.  On the other hand, if $\tilde S$ is 
incompressible on the side containing a cusp, either $\tilde S$ normalises 
to give a non-peripheral normal torus, contrary to 1-efficiency, 
or $M$ is a solid torus, which is excluded by our hypothesis. So the case 
where $\tilde S$ is two-sided is complete. 

If the generalised normal surface $\tilde S$ is one-sided it is a 
projective plane or Klein bottle. The first case contradicts 
irreducibility. If  $\tilde S$ is a Klein bottle, the boundary of a small 
regular neighbourhood of $\tilde S$ is either an  incompressible torus, 
contradicting the atoroidal assumption, or compressible and $M$ is a 
closed prism manifold, contradicting the hypothesis that $M$ has a cusp. 
This completes the proof. 
\end{proof}

\begin{lemma} 
\label{solid_torus_gen_norm}
Let $\T$ be a 1-efficient ideal triangulation of $M$, which is either a 
solid torus or $T^2 \times I$. Then $\T$ contains no closed embedded 
simple generalised normal surface $\tilde S$ with $\chi(\tilde S) > 0$,
and no closed embedded normal surface $\tilde S$ with $\chi(\tilde S) = 0$ 
which is non-peripheral. However, $\T$ always contains a closed embedded 
generalised normal $\tilde S$ with $\chi(\tilde S) = 0$.
In fact, there is an almost normal torus $\tilde S$ containing one 
octagonal disk.
\end{lemma}

\begin{proof} 
If there was a simple generalised normal surface $\tilde S$ with 
$\chi(\tilde S) > 0$, then $\tilde S$ would be a sphere or projective 
plane. A projective plane can be normalised and in the case of a sphere, 
$\tilde S$ could be normalised to the side which was not a ball. Both cases 
contradict 1-efficiency. Similarly there cannot be a non-peripheral normal 
surface $\tilde S$ with $\chi(\tilde S) = 0$ for the same reason. 

By a standard sweepout argument, starting with a peripheral normal torus 
at a cusp and sweeping to a core circle or the other cusp, there must be 
an intermediate almost normal torus $\tilde S$ - see Rubinstein \cite{Ru} 
or Stocking \cite{St}. If $\tilde S$ was obtained by attaching a tube to 
a normal sphere, this would contradict 1-efficiency. Hence $\tilde S$ must 
have one octagonal disk as claimed. 
\end{proof}

\begin{rmk} 
The proof of Theorem \ref{euler_for_gen_norm_surf} shows that a 1-efficient 
triangulation $\T$ of a closed orientable 3-manifold $M$ 
{\em other than $S^3$ or a lens space} contains no closed embedded simple 
generalised normal surface $\tilde S$ with $\chi(\tilde S) \ge 0$. 
(Here 1-efficient means there are no embedded normal spheres, projective 
planes, Klein bottles or tori, except for vertex linking spheres and 
edge linking tori. The former are boundaries of a small regular 
neighbourhood of a vertex, and the latter are obtained similarly from an 
edge which is a loop.) 
\end{rmk}

\begin{cor}
\label{1-efficient_degrees}
Let $\T$ be a 1-efficient ideal triangulation of a manifold $M$ other 
than the solid torus or $T^2 \times I$, and let $0\ne S \in N(\T;\Z)$ be 
an irreducible closed normal class. Then the contribution of $S$ to the 
3D-index sum has $q^{1/2}$-degree $d(S) \ge 1$.
\end{cor}

\begin{proof} 
Since $S$ is irreducible, it contains no tetrahedral solution summand, 
so $S=S^*$ is  the minimal non-negative coset representative of $[S]=S+\TT$. 
Then $S^*$  contains at most two quad types in each tetrahedron,
so can be replaced by an embedded generalised normal surface $\tilde S$ 
by Lemma \ref{gen_closed_class}. Now $\tilde S$ is connected since $S$ 
is irreducible, and $d(S) \ge - \chi(\tilde S)$ by 
Lemma \ref{deg_for_gen_norm_surf}.
Assume first that $\tilde S$ is not simple, so there are at least two of 
two different quadrilateral types in some tetrahedron for $S^*$. 
But then Lemma \ref{deg_for_gen_norm_surf} implies that $d(S) \ge 2$.
On the other hand,  if $\tilde S$ is simple and $M$ is not a solid torus, 
then $\chi(\tilde S)<0$ by Theorem \ref{euler_for_gen_norm_surf}, hence 
$d(S) \ge 1$. This completes the proof. 
\end{proof}

\begin{rmk}
\label{solid_torus_T2xI_degrees}
This result fails for 1-efficient triangulations of the solid torus and 
$T^2 \times I$. Here every normal class $S$ contributing to the index sum 
has degree $d(S)\ge 0$, but an almost normal torus $S$ gives a 
contribution $d(S)=0$. (Compare Example \ref{solid_torus_ex} below.)
\end{rmk}

One consequence of Corollary \ref{1-efficient_degrees} and 
Remark \ref{solid_torus_T2xI_degrees} is the following basic fact about 
the 3D-index.

\begin{cor}
If $\T$ is a 1-efficient ideal triangulation then the 3D-index $I_\T^0(0)$
is a formal power series in $q^{1/2}$, not just a formal Laurent series 
in $q^{1/2}$.
\end{cor}

In general, we expect that the minimum $q^{1/2}$-degree $d_{min}$ of 
$I^0_\T(kb)$ grows quadratically as $k \to +\infty$. However this is 
no longer true if $b$ is the boundary of an embedded spun normal surface.

\begin{lemma}
\label{linear_growth}
Assume $\T$ is a 1-efficient triangulation of a 3-manifold other than 
the solid torus or $T^2 \times I$. Assume that $\T$  contains an embedded 
spun normal surface $S$ with $\bd S=b \ne 0$, and 
that $S$ is the only irreducible spun normal class in $Q(\T;\Z_+)$ whose 
boundary is a positive integer multiple of $b$.
Then the minimum $q^{1/2}$-degree $d_{min}$ of $I^0_\T(kb)$ grows at 
mostly linearly as $k \to +\infty$.
In fact, $d_{min} I^0_\T(kb) \le -\chi(S)k$ for all even $k \ge 0$.   
\end{lemma}

\begin{proof} 
For each even $k\ge 0$, $k S$ has $\bd(k S) = k b$ and $[k S]=0$ so 
contributes a term to  $I^0_\T(kb)$
with $d(kS)^*)= d(kS) = -\chi(S)k+\delta(S)k^2 = -\chi(S)k.$ 
Now, by our assumption, any $S'$ with $\bd S'=b$ can be written as 
$S'=kS + S_0$ where $\bd S_0=0$. Hence $d(S') \ge d(kS) + d(S_0) > d(kS)$ 
unless $S_0=0$ by Corollary \ref{1-efficient_degrees}. 
\end{proof}

\begin{ex}
For the figure eight knot complement, with its canonical triangulation 
$\T$ given in Example \ref{ex:fig8}, there are embedded, essential, 
spun normal once-punctured Klein bottles with boundary
$b = \pm 4 \mu \pm \lambda$ and Euler characteristic $\chi= -1$ giving 
a term in $I^0_\T(k b)$ of degree $-\chi k =k$ for each $k \ge 0$. Further 
this is the {\em only} irreducible spun  normal class
with boundary a positive multiple of $b$, so by Lemma \ref{linear_growth}
this is the minimum $q^{1/2}$-degree of $I^0_\T(k b)$.
(Compare the calculations for 
$I_\T^0(8\mu+2\lambda) = I_{\CT}(4\mu+\lambda )$ and 
$I_\T^0(4\mu+\lambda )= I_{\CT}(2\mu+\frac{1}{2} \lambda)$
given in Example \ref{ex:fig8}.)

For the trefoil knot complement, with its 2 tetrahedra triangulation $\T$ 
given in Example \ref{trefoil_ex}, 
we have $I^0_\T((x \mu + y \lambda)) =  \delta_{0,x+6y}$.
Here there is an embedded, essential, spun normal M\"obius strip $S$ with 
boundary $b= \pm(6 \mu - \lambda)$
giving a contribution to  $I^0_\T(k b)$ of degree $0 = -\chi(S)k$ for all 
$k \in \Z_+$. Further, this is the {\em only} irreducible spun  normal class
with boundary a positive multiple of $b$, so  Lemma \ref{linear_growth} 
again applies.
\end{ex}

%%%%%%%%%%%%%%%%%%%%%%%%%%%%%%%%%%%%%%%%%%%%%%%%%%%%%%%%%%%%%%%%%%%%%%%%%%%
%%%%%%%%%%%%%%%%%%%%%%%%%%%%%%%%%%%%%%%%%%%%%%%%%%%%%%%%%%%%%%%%%%%%%%%%%%%

\section{Some examples}

In this section we give some examples of index computations for ideal 
triangulations with 2 tetrahedra, taken from the census
of triangulations described in Section \ref{PachnerGraph} below.

\subsection{Solid torus}
\label{solid_torus_ex}

The (open) solid torus has a 1-efficient triangulation $\T$  by 2 
tetrahedra with isomorphism signature `cMcabbgds', 
and gluing equation/holonomy coefficient matrix (from SnapPy) given by:
$$\begin{bmatrix} 2 & 1 & 2 & 2 & 2 & 2 \\ 0 & 1 & 0 & 0&0&0\\
0 & 1 & 0 & 0 & -1 & 0\\ 0 & 1 & 0 & 0 & 0& 2
\end{bmatrix}$$
(Here the peripheral curves $\mu,\lambda$ correspond to the meridian and 
longitude for the unknot in $S^3$.)
Summing over integer weights $k$ on edge 2 gives the 3D-index 
\begin{align*}
I_\T(x \mu + y \lambda) & =\sum_{k \in \Z} q^k \JD(0,k+x+y,0)\JD(0,-x,2y) \\
& =J(0,-x,2y) q^{-x-y} \sum_{k\in\Z} q^kJ(0,k,0)=0
\end{align*}
since
$$
\sum_{k\in\Z} q^k \JD(0,k,0) = 
\sum_{k\in\Z} \ID(0,k) q^k = I(0,q,q)= \frac{(q q^{-1};q)_\infty}{(q)_\infty;q}=0
$$
using equation (\ref{ID_gen_func}). 

It follows that the 3D-index {\em vanishes identically} for
all 118753 1-efficient ideal triangulations of the solid torus with 
at most 6 tetrahedra, since we have checked that these are all connected to 
$\T$ by 2-3, 3-2, 0-2 and 2-0 moves preserving 1-efficiency.

\subsection{Trefoil complement}
\label{trefoil_ex}

The right-handed trefoil complement (`L103001' in SnapPy) has a 
2-tetrahedron triangulation $\T$ with gluing equation/holonomy coefficients 
given by:
$$
\begin{bmatrix} 1 & 2 & 2 & 2 & 1 & 2 \\ 1 & 0 & 0 & 0&1&0\\
0 & 0 & -1 & 1 & 0 & 0\\ 1 & 0 & -4 & 4 & -1 & 0
\end{bmatrix}
$$

Hence, summing over edge 2, the 3D-index is
\begin{align*}
I_\T(x \mu + y \lambda) &= \sum_{k \in \Z} \JD(k+y,0,-x-4y)\JD(x+4y,k-y,0) \\
&=\sum_k q^k \ID(x+4y,k+y)(-q^{1/2})^{-x-4y}\ID(x+4y,k-x-5y) \\
&= \sum_{k'} q^{k'} (-q^{1/2})^{-x-6y} \ID(m,k') \ID(m,k'-x-6y) \\
&= \sum (-q^{1/2})^{-x-6y} \delta_{0,-x-6y} = \delta_{0,x+6y}.
\end{align*}
where we put $k'=k+y$ and $m=x+4y$ and used the quadratic identity.

Note that $I_\T(\gamma)$ is non-trivial exactly when $\gamma$ is a multiple of
the boundary curve $-6\mu+\lambda$ of the essential annulus 
(and M\"obius strip) in the trefoil knot exterior.

\subsection{$\bf{T^2 \times I}$}  
\label{t2xI_ex}

There is a unique triangulation of $T^2 \times I$ with 3 tetrahedra. 
Regina's isomorphism signature of the triangulation is `dLQacccbjkg'. This 
triangulation turns out to be 1-efficient and has the following 
gluing equation/holonomy coefficients (from SnapPy):

$$
\begin{bmatrix} 
1 &   2 &   2 &   1 &   1 &   1 &   1 &   1 &   1\\
1 &   0 &   0 &   0 &   0 &   0 &   0 &   0 &   0\\
 0 &   0 &   0 &   1 &   1 &   1 &   1 &   1 &   1\\
0 &   0 &   0 &   0 &   0 &   1 &   1 &   0 &   0\\
-1 &   0 &   0 &   0 &   1 &   0 &   0 &   0 &   1\\
0 &   0 &   0 &   1 &   0 &   0 &   0 &  -1 &   0\\
 0 &   0 &   0 &   0 &   0 &   1 &  -1 &   0 &   0
\end{bmatrix}
$$

Because both this triangulation and the solid torus triangulation have a 
degree one edge, the computation is very similar to that of 
Example \ref{solid_torus_ex}.

Let $\omega = (x_1 \mu_1+ y_1 \lambda_1, x_2 \mu_2+ y_2 \lambda_2)$. Then
\begin{align*}
I_{\CT}({\omega})(q) &= \sum_{k_2 \in \Z} 
q^{k_2} \JD(k_2-y_1,0,0) \JD(x_2, y_1,x_1+y_2) \JD(x_1-y_2,-x_2,y_1)\\
&=\JD(k_2-y_1,0,0) \JD(x_2, y_1,x_1+y_2) q^{y_1} 
\sum_{k_2 \in \Z} q^{k_2-y_1} \JD(k_2-y_1,0,0)=0,
\end{align*}
since
$$
\sum_{\ell\in\Z} q^\ell \JD(\ell,0,0) = 
\sum_{\ell\in\Z} \ID(-\ell,0) q^\ell= \sum_{\ell\in\Z} \ID(0,\ell) q^\ell
= I(0,q,q)= \frac{(q q^{-1};q)_\infty}{(q)_\infty;q}=0.
$$
 
\subsection{A toroidal example}  
\label{toroidal_ex}

The 2-tetrahedron ideal triangulation $\T$ with isomorphism signature
`cPcbbbdei' is not 1-efficient.  In fact it gives a manifold containing an 
incompressible torus which splits the manifold in 
two Seifert fibres pieces $SFS[D^2(2,1)(3,1)]$ and $SFS[A^2(2,1)]$. The 
gluing equation/holonomy coefficients (from SnapPy) are given by:
$$
\begin{bmatrix} 1 & 1 & 0 & 1 & 0 & 1 \\ 1 & 1 & 2 & 1&2&1\\
0 & -1 & 0 & 1 & 0 & 0\\ 0 & 0 & 2 & 0 & 0& 0
\end{bmatrix}
$$

Hence, summing over edge 1,  the 3D-index is given by
$$
I_\T(x \mu + y\lambda ) =\sum_{k \in \Z} q^k \JD(k,k-x,2y) \JD(k+x,0,k)
= \sum_{k \in \Z} \ID(x,2y-k) \ID(-x,-k).
$$

If $x=0,y=0$ then the above sum $\sum_k \ID(0,-k)^2$ 
is divergent, as expected by Theorem \ref{index_structure_iff_1-efficient}, 
since each term with $k \le 0$ has degree $0$.

However for $x \ne 0$ the sum converges, and experimentation suggests the sum 
simplifies to the following geometric series:
$$
I_\T(x \mu + y \lambda) = \frac{ (-1)^x q^{|x|(|y + x/2| + 1/2)}}{(1 - q^{|x|})}, 
\text{ for } x\ne 0.
$$

Putting 
$$
x \mu + y \lambda = a \mu +(b-a/2)\lambda = a(\mu-\lambda/2)+b\lambda =
a \bar \mu + \b \bar \lambda
$$
we have symmetries
$I_\T(a \bar \mu + b \bar\lambda) = 
I_\T( \pm a \bar \mu + \pm b \bar\lambda)$ for $a \ne 0$ 
(by the duality and triality identities) and experimentally it seems that
$$
I_\T(a \bar \mu + b \bar\lambda) = \frac{(-1)^a q^{a(b+1/2)}} {1-q^a},  
\text{ for }a>0, b \ge 0.
$$

\subsection{m009}

The manifold $X=m009$ is the first orientable cusped manifold in the 
SnapPea census with non-peripheral $\Z/2\Z$ homology. It has a 
triangulation using 3 tetrahedra with
gluing equation/holonomy coefficients $E_0,E_1,E_2,M,L$ (from SnapPy) 
given by:
$$
\begin{bmatrix} 
2 & 0 & 0 & 2 & 1 & 0 & 2 & 1 & 0\\
0 & 2 & 0 & 0 & 0 & 2 &  0 & 0 & 2\\
0 & 0 & 2 & 0 & 1 & 0 & 0 & 1 & 0\\
0 & -1& 1 & 0 & 1 & -1& 0 & 0 & -1\\
-1 & 0 & 0 & 1 & 0 & 0 & -2 & 0 & 1
\end{bmatrix}
$$

and the tetrahedral solutions $T_0,T_1,T_2$ have coefficients
$$
\begin{bmatrix} 
1 & 1 & 1 & 0 & 0 & 0 & 0 & 0 & 0\\
0 & 0 & 0 & 1 & 1 & 1 &  0 & 0 & 0\\
0 & 0 & 0 & 0 & 0 & 0 & 1 & 1 & 1\\
\end{bmatrix}
$$

Here,
\begin{itemize}
\item 
$H_1(X;\Z) \cong \Z \times \Z/2\Z$
\item 
$H_2(X,\bd X;\Z/2\Z) \times H_1(\bd X;\Z)  \cong \Z/2\Z \times \Z^2$ with 
$\Z^2$ generated by the 
homology classes $\mu,\lambda$ of the ``meridian'' and ``longitude'' chosen 
by SnapPy.
\item 
${\mathcal K} = \Ker\left(H_1(\bd X; \Z) \to H_1(X;\Z/2\Z)\right)$ is 
spanned by $\mu+\lambda,\mu-\lambda$
\item 
The taut angle structure with angles $\alpha=(0,\pi,0,0,\pi,0,0,\pi,0)$ has 
vanishing peripheral rotational holonomy, 
so can be used to compute Euler characteristics via 
$\chi = \sum -\alpha(q)/\pi$. 
\item 
$S_1= \frac{1}{2} E_1 = 
\begin{bmatrix} 0 & 1 & 0 & 0 & 0 & 1 &  0 & 0 & 1\end{bmatrix}$ has
$([S_1]_2,[\bd S_1]) = (\bar 1, 0 ) \in \Z/2\Z \times \Z^2$ and $\chi(S_1)=-1$
\item 
$N(\T;\Z)$ is spanned over $\Z$ by $E_0,E_1,E_2,T_0,T_1,T_2$ and $S_1$.
\item
$M$ has $([M]_2,[\bd M]) = (0, 2\mu ) \in \Z/2\Z \times \Z^2$ and $\chi(M)=0$
\item
$L$ has $([L]_2,[\bd L]) = (0, 2\lambda ) \in \Z/2\Z \times \Z^2$ and 
$\chi(L)=0$
\item
$S_2 = \frac{1}{2}(M+L+T_0-T_1) = 
\begin{bmatrix} 0 & 0 & 1 & 0 & 0 & -1 &  -1 & 0 & 0\end{bmatrix}$ has
$([S_2]_2,[\bd S_2]) = (0, \mu+\lambda ) \in \Z/2\Z \times \Z^2$ and 
$\chi(S_2)= 0$
\item
$Q(\T;\Z)$ is spanned over $\Z$ by $E_0,E_1,E_2,T_0,T_1,T_2$, $S_1$, $M,L$ 
and $S_2$.
\end{itemize}

To compute the index for the `even' class 
$(\bar 0,2x,2y) \in \Z/2\Z \times \Z^2$ we sum over elements of $(\E+\TT)/\TT$ 
by taking edge coefficients $(0,k_1,k_2)$ with  $k_1,k_2 \in \Z$ giving
\begin{align*}
I^{even}(2x\mu+2y\lambda) &= \sum_{k_1,k_2} q^{k_1+k_2} \JD(k_1E_1+k_2E_2+xM+yL) \\
& = \sum_{k_1,k_2}  q^{k_1+k_2}
\JD(-y,2k_1-x,2k_2)\JD(y,k_2+x,2k_1-x)\JD(-2y,k_2,2k_1-x+y)
\end{align*}
e.g
$$
I^{even}(0,0)=1-q-q^2+6 q^3+9 q^4+12 q^5-5 q^6-34 q^7-79 q^8-118 q^9
-118 q^{10}+ \ldots
$$

To compute the index for the `odd' class 
$(\bar 1,2x,2y) \in \Z/2\Z \times \Z^2$ we sum over elements of the coset 
$(S_1+\E+\TT)/\TT$ by taking edge coefficients $(0,\frac{1}{2}+k_1,k_2)$, 
$k_1,k_2 \in \Z$ giving
\begin{align*}
I^{odd}(2x\mu+2y\lambda) &= \sum_{k_1,k_2} (-q^{1/2})q^{k_1+k_2} 
J(S_1+k_1E_1+k_2E_2+xM+yL) \\
&= \sum_{k_1,k_2}  -q^{1/2+k_1+k_2}\JD(-y,2k_1+1-x,2k_2)  \JD(y,k_2+x,2k_1+1-x)\\
&\hspace{4cm} \times\JD(-2y,k_2,2k_1+1-x+y)
\end{align*} 
e.g.
$$
I^{odd}(0,0)=-q^{1/2}-2 q^{3/2}+2 q^{5/2}+8 q^{7/2}+11 q^{9/2}+6 q^{11/2}
-17 q^{13/2}-57 q^{15/2}-100 q^{17/2}-124 q^{19/2}+\ldots
$$

For the class $(\bar 0, 1,1) \in \Z/2\Z \times \Z^2$ we take edge 
coefficients $(0,k_1,k_2)$ with  $k_1,k_2 \in \Z$  giving
\begin{align*} 
I^{even}(\mu+\lambda) &= \sum_{k_1,k_2} q^{k_1+k_2} J(S_2+k_1E_1+k_2E_2) \\
&= \sum_{k_1,k_2} q^{k_1+k_2}  
\JD(0,2k_1,2k_2+1)\JD(0,k_2,2k_1-1)\JD(-1,k_2,2k_1) \\
&= -q+4 q^3+7 q^4+6 q^5-7 q^6-32 q^7-65 q^8-89 q^9-81 q^{10} + \ldots
\end{align*}

%%%%%%%%%%%%%%%%%%%%%%%%%%%%%%%%%%%%%%%%%%%%%%%%%%%%%%%%%%%%%%%%%%%%%%%%%%%
%%%%%%%%%%%%%%%%%%%%%%%%%%%%%%%%%%%%%%%%%%%%%%%%%%%%%%%%%%%%%%%%%%%%%%%%%%%

\section{Notes on connectedness of the 1-efficient Pachner Graph}
\label{PachnerGraph}

In this section we study the \emph{Pachner graph} for a compact orientable 
3-manifold $M$ with non-empty boundary consisting of tori. 
This is a graph where each vertex corresponds to an ideal triangulation 
(up to the equivalence of relabelling) of the manifold $M$,
 with an edge connecting two vertices if their corresponding triangulations 
can be obtained from one another via 2-3 or 3-2 
 Pachner moves (as shown in Figures \ref{fig:twoThreeBipyramid} and 
\ref{fig:twoThreeBoundary}.)
 We give some general results, then describe a census of all ideal 
triangulations with at most 6 tetrahedra which provides 
a number of examples with interesting properties. 
 
\begin{figure}[htpb]
\centering
\includegraphics[width=0.6\textwidth]{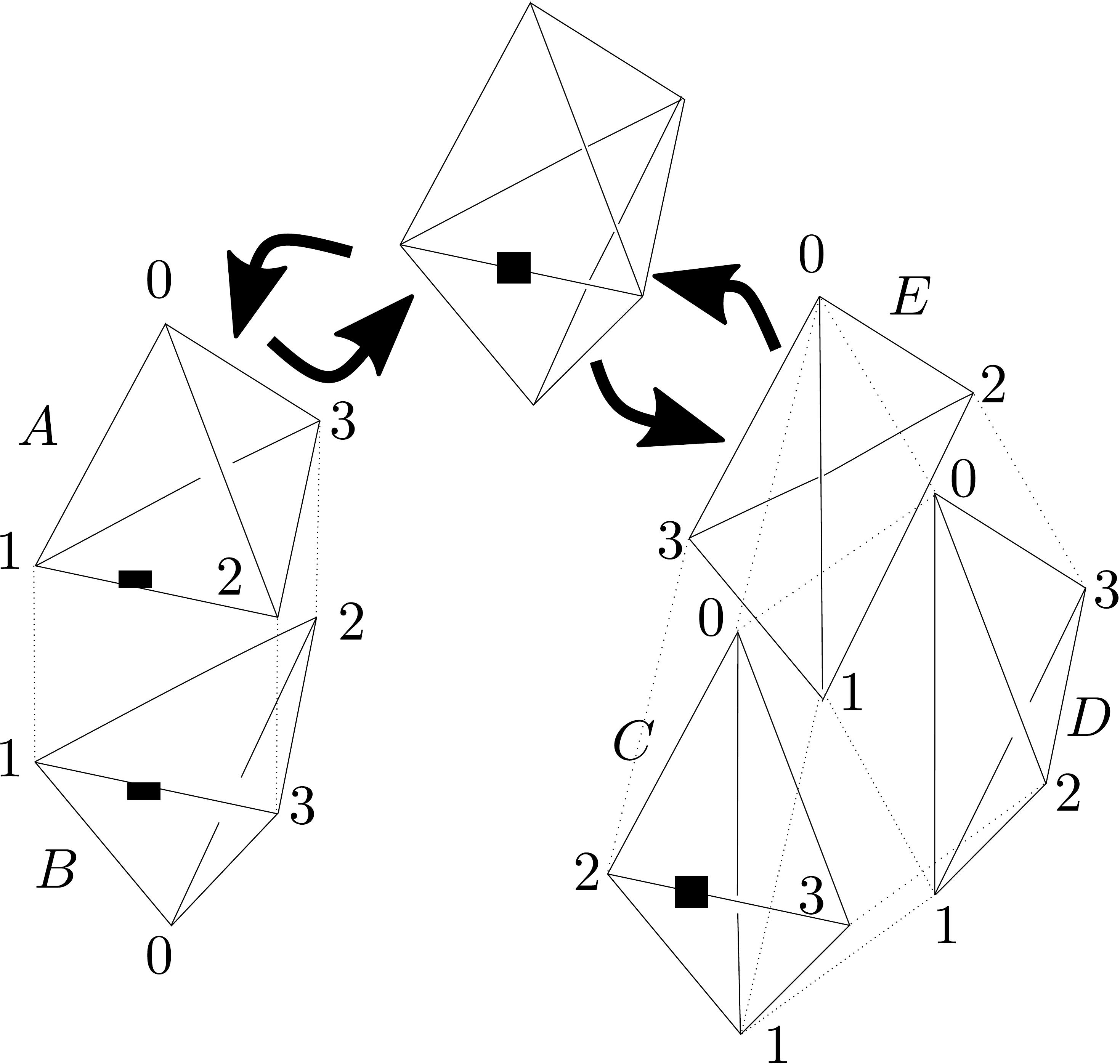}
%\resizebox{10 cm}{!}{\input{twoThreeLabels.eps_tex}}
\caption{\label{fig:twoThreeBipyramid} A 2-3 move on labelled tetrahedra. 
As the black square indicates, edges $A12$ and $B13$ are identified with $C23$}
\end{figure}

\begin{figure}[htpb]
\centering
\includegraphics[width=0.8\textwidth]{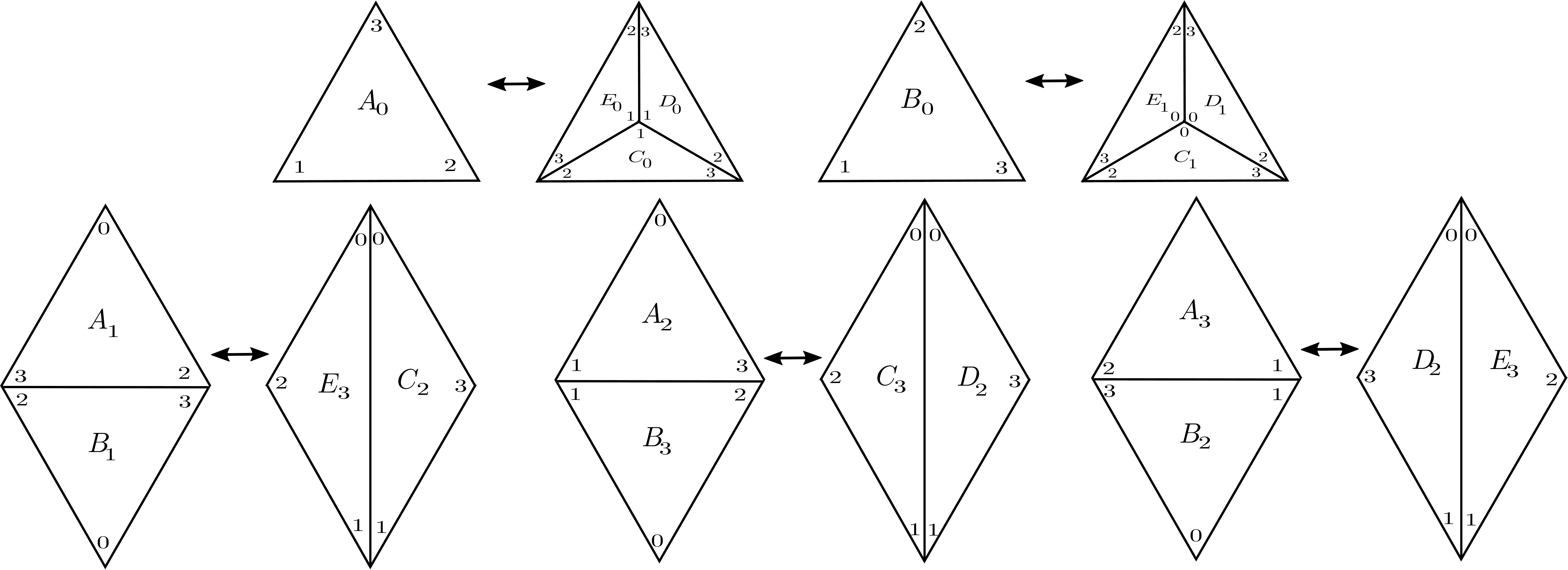}
%\resizebox{14 cm}{!}{\input{twoThreeTriangles.eps_tex}}
\caption{\label{fig:twoThreeBoundary}The effect of the 2-3 move on the 
boundary triangles.}
\end{figure}

Throughout the section, we will use Burton's isomorphism signature 
notation \cite{BurtonIsoSig} to identify triangulations. 
Interested readers can replicate the results of this section by inputing 
these signatures into Regina to construct the relevant triangulations. 
Readers new to this concept will find it helpful to know that if the first 
letter of the signature (of those used in this paper) is the 
$n$-th letter of the alphabet, then the corresponding triangulation is 
comprised of $n-1$ tetrahedra. 

As mentioned in the introduction, the 
the Pachner graph of a $3$-manifold with non-empty boundary is connected, and so it is 
interesting to consider basic properties of subgraphs of this graph: 
number of vertices, connectedness, etc. For example,  we say the 
\emph{geometric Pachner graph} is the subgraph of ``geometric'' 
triangulations such that Thurston's gluing equations have a solution where each 
tetrahedral shape has positive imaginary part.

It has been recently shown that the figure eight knot complement admits 
infinitely many geometric triangulations, however the geometric subgraph 
of the Pachner graph is disconnected 
(see \cite[Theorem 1.1 and Remark 3.3]{DaddDuan}). 

One can also define the \emph{1-efficient Pachner graph} as the subgraph 
of the Pachner graph containing only 1-efficient triangulations and the 
edges between them. 

Before discussing the 1-efficient Pachner graph, we mention that our 
notation (used here and in the accompanying code) for Seifert fibred 
spaces with non-empty boundary, is consistent with Hatcher \cite{Hat}. 
Namely, a Seifert fibre space over the surface $F$ with exceptional fibres 
given by the parameters $\{(a_i,b_i)\}_{i=1} ^k$ (with $a_i>b_i>0$ if $F$ is 
punctured) will be denoted by $SFS[F(a_1,b_1)...(a_k,b_k)]$. 

First, we point out that for many 3-manifolds 
the 1-efficient Pachner graph is infinite. Although this is presumably 
known to the experts, we include it for the sake of completeness. 

\begin{prop}
Let $M$ be an cusped irreducible, atoroidal orientable 3-manifold other 
then the solid torus  $S^1\times D^2$. Then $M$ admits infinitely many 
1-efficient triangulations.
\end{prop}

The key idea for this proof is to exhibit ideal triangulations that 
support a {\em taut angle structure}, i.e.
a semi-angle structure where all angles are either 0 or $\pi$. 
(A semi-angle structure is an assignment of non-negative angles to the 
edges in the ideal tetrahedra, so that the angle sum at each ideal vertex 
in such a tetrahedron is $\pi$.)
Also, in \cite[Theorem 1]{La1}, Lackenby shows that an-annular cusped 
manifolds  admit taut ideal triangulations,
 which give such a  taut angle structure.
Given the hypotheses of the Proposition, we point out that an-annular is 
shorthand for Seifert fibre spaces over the disk 
over than $S^1\times D^2$ or $SFS[D(2,1)(2,1)]$, and hyperbolic manifolds. 

\begin{proof}
If $M$ is hyperbolic or a Seifert fibre space over the disk (other than 
$S^1\times D^2$ or $SFS[D(2,1)(2,1)]$), then by \cite[Theorem 1]{La1}, 
$M$ admits a taut ideal triangulation, say with $n$ tetrahedra. 
In this triangulation, about any edge we can find two faces of the 
triangulation incident 
to that edge which (in an open neighbourhood of the edge) separate the 
two angles labelled by $\pi$. Performing a 0-2 move along these two 
faces produces a new triangulation with $n+2$ tetrahedra. To see that 
this triangulation supports a taut angle structure, label each edge in the 
new edge class of degree 2 by $\pi$, the opposite edges in the tetrahedron 
by $\pi$, and the remaining edges $0$. As this is a semi-angle 
structure, this new triangulation is 1-efficient since $M$ is atoroidal 
(see \cite[Theorem 2.6]{KR2}). 

If $M$ is Seifert fibred over an annulus with exactly one exceptional 
fibre, then $M$ is fibred over the circle with fibre a surface of negative 
Euler characteristic. Therefore, $M$ admits a layered triangulation 
obtained from expressing the monodromy of its fibration in terms of edge 
flips on an ideal triangulation of the fibre (see for example 
\cite[Section 2]{La1}). Using this ideal layered triangulation, we again 
obtain a taut angle structure and complete the proof in the same manner 
as above. 
\end{proof}

The following remark uses a dual notion to 1-efficiency, introduced by 
Garoufalidis \cite{Gar}, to obtain a similar result. Let $\T$ be an ideal 
triangulation with $n$ tetrahedra and $\Lambda=\{\lambda_i\}_{i=1}^{3n}$ a 
set of $3^n$ generalised angle structures. Then $\Lambda$ is a 
\emph{index structure} on $\T$ if for each $Q=(Q_1, ..., Q_n)$ where $Q_j$ 
is a choice of quadrilateral type in tetrahedron $j$,  there exist 
$\lambda_i$ such that for all $j$ the edges opposite $Q_j$ in $\T$ have 
positive angles in $\lambda_i$. This is an obstruction to having embedded 
surfaces of non-negative Euler characteristic, and in fact it is known to 
be equivalent to $\T$ being 1-efficient by \cite[Theorem 1.2]{GHRS}. 

\begin{rmk}
An ideal triangulation of a solid torus will not admit a taut angle 
structure (by \cite{KR2}, the existence of an embedded generalised 
normal torus in a solid torus is an obstruction to existence of a taut 
angle structure). Nevertheless, an analogous method can be used to 
obtain infinitely many 1-efficient triangulations of the solid torus. For 
example, the triangulation `dLQacccbnbb' has a gluing matrix: 
$$
\begin{pmatrix} 
1& 0& 0& 0& 0& 0& 0& 0& 0\\
1& 2& 2& 1& 2& 2& 1& 2& 2\\
0& 0& 0& 1& 0& 0& 1& 0& 0\\
\end{pmatrix}
$$
and so $\{(2\pi, \beta_0, -\pi- \beta_0, \pi, \beta_1, 
-\beta_1, \pi, \beta_2, -\beta_2)\mid \beta_0,\beta_1,\beta_2\in \R \}$ 
is a subset of the generalised angle structures on `dLQacccbnbb'. 
Furthermore, one has the freedom to adjust these parameters to find an 
index structure. In fact, after performing a 0-2 move along the two faces 
incident to the degree two edge, we see that a new index structure is 
obtained since we can obtain new generalised angle structures by just 
introducing two new tetrahedra subject to same constraint as the last two 
tetrahedra above, i.e. the angles are $(\pi, \beta_j, -\beta_j)$. 
Furthermore, one can then perform 0-2 moves along the faces incident to 
any degree 2 edge in the resulting triangulations to obtain more 
1-efficient triangulations, and so by an inductive argument there are 
infinitely many 1-efficient triangulations of the solid torus.
\end{rmk}

\subsection{A census of ideal triangulations}

These examples were obtained by first compiling a list of all ideal 
triangulations with at most six tetrahedra using Regina's {\tt tricensus} 
function. Then it was determined if these triangulations were solid tori. 
Of the remaining triangulations, we then (i) determined their irreducibility 
and (ii) checked for incompressible tori and Klein bottles. The first 
condition was checked using a light adaptation of Regina's 
{\tt isThreeSphere} function and the second determined using Regina's 
{\tt isSolidTorus} and a new function {\tt isT2xI} (see \cite{regina}). 
The latter function is similar to Haraway's 
algorithm \cite[Corollary 15]{haraway2014determining}, but instead of 
using an analysis of spun normal  annuli, it relies on the fact that if 
$M$ is a 2-cusped manifold with three $S^1\times D^2$ surgeries along 
slopes $\{\gamma_1,\gamma_2,\gamma_3\}$, then $M$ is homeomorphic to 
$T^2 \times I$ the Berge manifold \cite{berge1991knots,gabai1989surgery}, 
or is Seifert fibred over the annulus with one exceptional fibre. However, 
all of the $S^1\times D^2$ fillings of the Berge manifold are 
pairwise distance one from each other, and so if a manifold admits 
$S^1\times D^2$ fillings along slopes 
$\{\frac{1}{0}, -\frac{1}{1}, \frac{2}{1}\}$ 
then it must be $T^2\times I$ or Seifert fibred over the annulus. 
The latter can be ruled out since at least one of the fillings 
$\{\frac{1}{0}, -\frac{1}{1}, \frac{2}{1}\}$ will result in a Seifert 
fibred space over the disk with two (non-trivial) exceptional fibres.

This allows for a decomposition into prime and atoroidal pieces. Those 
pieces were then classified by searching the Pachner graph of the 
corresponding triangulations and simplifying the triangulation to that 
of either a Snappy {\tt OrientableCuspCensus} triangulation (using 
SnapPy's {\tt identify} function) or until it was equivalent via Pachner 
moves to a known triangulation of a lens space (in the case of some prime 
summands) or a cusped Seifert fibered space (using a dictionary of 
triangulations also included with the code). While a larger library of 
closed triangulations would be needed for more complicated prime and 
JSJ decompositions, this was sufficient for our purposes. 

After this coarse classification, the triangulations were then analysed 
for 1-efficiency. The results of this census are summarised in 
Tables \ref{table:smalltetcensus} and \ref{hypCensus} below;
the code used is available as an ancillary file to the arxiv version 
of this paper. 

\begin{table}[h]
\begin{tabular}{|c|c|c|c|c|c|c|c|c||c|c|}
\hline
n &	total& 	$S^1\times D^2$ &	$T^2 \times I$	& $P \times S^1$ &	Red.	& Tor. &	SFS & Hyp. &Taut	&1-efficient\\
\hline
2&	10&	   3&	    0&	  0&	0&	1&	 4&	2&	7&	9\\
\hline
3&	129&	   65	&    1&	  0&	0&	15&	 36&	12&		53&	102\\
\hline
4&	1852& 917& 11& 0&	107&	188& 	491&	138&		441&	1082\\
\hline
5 &26909&	14324&	197&	1&	2533& 	2164&	6344&	1346& 3310 &	12130\\
\hline
6&414946  &219080 &2981 & 32& 58508& 29451& 89933 &14961 &134538 & 28405  \\
\hline
\end{tabular}

\caption{\label{table:smalltetcensus} This table provides information on 
the various types of (orientable) ideal triangulations that can be built 
from $n=2,3,4,5,6$ tetrahedra. Total is the total number of triangulations 
for each value of $n$, $S^1\times D^2$ is the number of solid tori, 
similarly for $T^2\times I$ and $P\times S^1$. Red.~counts the number of 
reducible manifolds and Tor.~counts the number of toroidal manifolds which 
are not homeomorphic to $SFS[D(2,1)(2,1)]$. Instead $SFS[D(2,1)(2,1)]$ is 
counted in SFS which record the number of SFS over the disk or annulus 
with two or one exceptional fibres respectively. Finally, Hyp.~counts the 
number of hyperbolic manifolds observed. Note that in each case, this 
computation is rigorous, in the sense that the triangulation was connected 
to a known triangulation via Pachner moves. Finally, Taut records the number 
of triangulations admitting a taut angle structure and 1-efficient records 
the number of 1-efficient triangulations.}
\end{table}

\begin{table}[h]
\begin{tabular}{|c|c|c|c|c|c|c|c|}
\hline
n & Hyp. &Geom. &Semi-Geom. & Strict Ang. Struct. & Semi-Ang.&  Taut & 1-efficient\\
\hline
2&	2&	  2&	    2&	  2 &2 &2 &2\\
\hline
3&	12& 12 & 12 & 12& 12& 12&12\\
\hline
4&	138& 110& 121 & 119 & 134 & 134 & 134\\
\hline
5&	1346&670&798&801&1165 & 1118&1190\\
\hline
6&	14961&3857&4923&5221&10908&9927&11512\\
\hline
\end{tabular}

\caption{\label{hypCensus} This table provides data on the ideal 
triangulations of hyperbolic manifolds in the census. Hyp.~gives the 
total number of hyperbolic manifolds. Geom.~gives the number of 
triangulations which SnapPy found an approximate hyperbolic structure 
and Semi-Geom.~records the number of triangulations which SnapPy found 
an approximate hyperbolic structure possibly with flat (and non-degenerate) 
tetrahedra. Strict.~Ang.~Struct.~records the number of triangulations 
admitting a strict angle structure, while Semi-Ang.~records the number of 
triangulations admitting a semi-angle structure. Finally, Taut records the 
number of triangulations admitting a Taut structure and 1-efficient records 
the number of 1-efficient triangulations of hyperbolic manifolds.}
\end{table}

\subsection{Two tetrahedron ideal triangulations} 

Up to relabelling, there are 10 triangulations of orientable cusped 
3-manifolds that decompose into two tetrahedra, and these 10 triangulations 
correspond to seven manifolds up to homeomorphism. The relevant data are 
presented in Table \ref{tab:twoTetCensus}.

\begin{table}[ht]
\begin{tabular}{|c|c|c|}
\hline
signature & name & 1-efficient\\
\hline
\hline
cMcabbgds & {\bf $S^1\times D^2$} & Yes\\
\hline
cMcabbgij &$S^1\times D^2$ & Yes\\
\hline
cMcabbgik &$S^1\times D^2$ & Yes\\
\hline
\hline
cPcbbbalm &{\bf $SFS[D^2(3,1)(3,1)]$} & Yes\\
\hline
\hline
cPcbbbali &{\bf $SFS[D^2(3,1)(3,2)]$} & Yes\\
\hline
\hline

cPcbbbadh &{\bf $SFS[D^2(2,1)(3,1)]$} & Yes\\
\hline
cPcbbbadu &{\bf $SFS[D^2(2,1)(3,1)]$} & Yes\\
\hline
\hline
cPcbbbdxm& {\bf Figure 8 sister (m003)} & Yes\\
\hline
\hline
cPcbbbiht &{\bf Figure 8 (m004)} & Yes\\
\hline
\hline
cPcbbbdei &{\bf $SFS[D^2(2,1)(3,1)] \cup SFS[A^2(2,1)]$} & No \\
\hline

\end{tabular}
\caption{\label{tab:twoTetCensus} The complete census of ideal 
triangulations with two tetrahedra. The homeomorphism descriptions come 
from an appeal to the surgery description given by Martelli and Petronio 
\cite{MP},  (which the diligent reader can verify using a tangle computation). 
}
\end{table}

\break
\noindent
{\bf Some Pachner Paths.} 

To provide a certificate of a path in the Pachner graph we list the 
vertices and also provide the face consumed by a 2-3 move or the edge 
consumed by a 3-2 move. In the first case, we take the index of the face 
in the labelled triangulation determined by the isomorphism signature. 
In the second, we take $-index -1$ of the edge (the minus sign indicates 
that an edge is being consumed and shifting by $-1$ removes the ambiguity 
of $-0$ and $+0$ that would arise if we were required to perform a 3-2 
along edge 0). 

\begin{ex} 
{\bf Solid Torus.} 
The (open) solid torus has 3 minimal ideal triangulations with 2 tetrahedra. 
The paths shown in Table \ref{solid_torus_paths}
are 1-efficient paths in the Pachner graph
connecting these minimal triangulations. 

\begin{table}[ht]
\begin{tabular}{|c|c||c|c||c|c|}
\hline
Path 1 & & Path 2 & & Path 3& \\
\hline
isoSig & 2-3 move &isoSig & 2-3 move&isoSig & 2-3 move\\
\hline
`cMcabbgds' & 0 &`cMcabbgds'& 0&`cMcabbgij' & 0\\
 \hline
 `dLQbcccaego' & 0 & `dLQbcccaego' & 4 &`dLQbcccahgc'& 5\\
\hline
`eLPkbcdddhgcgj'& -3 & `eLPkbcdddhcgcf' & -2 &`eLAkbccddaegtr' & -3\\
\hline
`dLQbcccahgc' & -3 & `dLQacccjrgr' & 2 & `dLQbcbcaekv'& 2\\
\hline
`cMcabbgij' & `end' & `eLPkbcdddhcgbf'& -1 & `eLAkbccddaegtn' & -4\\
\hline
- & - & `dLQbcccahgo' & -3 & `dLQbcccahgo'& -3\\
\hline
- & - & `cMcabbgik' & `end' & `cMcabbgik' & `end'\\
\hline
\end{tabular}
\caption{\label{solid_torus_paths}Some 1-efficient paths for 
triangulations of the solid torus}
\end{table}
\end{ex}

\begin{rmk} 
This analysis shows that the 1-efficient Pachner graph (of ideal 
triangulations) of the solid torus is not connected. Specifically, we can 
say the following: Both `gLLAQbecefefaaopaaj' and `gLLAQbecefefaaopaan' are 
six tetrahedral triangulations of the solid torus that 
1) have no degree 3 edges and 2) each 2-3 move along a face in either 
triangulation results in a not 1-efficient triangulation. 
However, we can still relate the 3D-index to the 3D-index of the 
triangulations of the solid torus in Table \ref{tab:twoTetCensus}. 
In fact, both of these six tetrahedral triangulations have (three) degree 
two edges. After performing a 2-0 move on any degree two edge of one of 
these triangulations, the resulting triangulation is 1-efficient 
and connected to the triangulation  `cMcabbgij' via a path of 1-efficient 
triangulations. Therefore, the 3D-index is consistent with the two 
tetrahedral triangulation of an ideal solid torus `cMcabbgij' 
by \cite[Theorem 5.1]{GHRS}. 
\end{rmk} 

\begin{ex}
{\bf Trefoil Complement.} The trefoil complement ($SFS[D^2(2,1)(3,1)]$) 
has two minimal triangulations that are connected along a path of length 
6 (see Figure \ref{fig:chartPPTrefoil} and Table \ref{tab:ppTrefoil}) and 
a 1-efficient path of length 12 (see Figure \ref{fig:chartPPTrefoil} and 
Table \ref{tab:ppOETrefoil}).

\begin{figure}[htpb!]
\resizebox{16 cm}{!}{\xymatrix{
g\ar@{..}[rrrrr]    &&&& & O_4\ar@{-}@[red][dr]\ar@{..}[rrrr]  &&&& 
O_8\ar@{-}@[red][dr]\ar@{..}[rrrrr]  &&&&& 6\\
f\ar@{..}[rrrr]    &&&& O_3\ar@{-}@[red][ur]\ar@{..}[rr]  & 
&O_5\ar@{-}@[red][dr]\ar@{..}[r]& T_3\ar@{-}@[green][drr]\ar@{..}[r] 
& O_7\ar@{-}@[red][ur]\ar@{..}[rr] & &O_9\ar@{-}@[red][dr]\ar@{..}[rrrr] 
&&&&5\\
e\ar@{..}[rrr]  & & & O_2\ar@{-}@[red][ur]\ar@{..}[rr] & & 
T_2\ar@{-}@[green][urr]\ar@{..}[rr]  & &
O_6\ar@{-}@[red][ur]\ar@{..}[rr]\ar@{--}@[blue][ddrrrrrr]&&
T_4\ar@{-}@[green][drr]\ar@{..}[rr] &&O_{10}\ar@{-}@[red][dr]\ar@{..}[rrr] 
&&& 4\\
d\ar@{..}[rr] & & O_1 \ar@{-}@[red][ur]\ar@{..}[r] & 
T_1 \ar@{-}@[green][urr] \ar@{..}[rrrrrrrr] & &&&&&&&
T_5\ar@{-}@[green][drr]\ar@{..}[r] & O_{11}\ar@{-}@[red][dr]\ar@{..}[rr] 
&&  3\\
c\ar@{..}[r] &T_s  \ar@{-}@[red][ur]\ar@{.}[rrrrrrrrrrrr] 
\ar@{-}@[green][urr] \ar@{--}@[blue][uurrrrrr]  & &&&&&&&&&&& 
T_f\ar@{..}[r]  & 2
}
}
\caption{\label{fig:chartPPTrefoil} A shortest Pachner path between 
minimal triangulations of the trefoil (green) and a longer 1-efficient 
path (red). Note, the 1-efficient path is a shortest path through 
1-efficient triangulations using 2-3/3-2 moves. If 0-2/2-0, moves are 
allowed then the red and green paths shorten to $\{T_s, O_6, T_f\}$ (blue 
dashed). Also, a computation in Regina shows all triangulations on the 
green path except $T_s$ and $T_f$ are not 1-efficient.
} 
\end{figure}
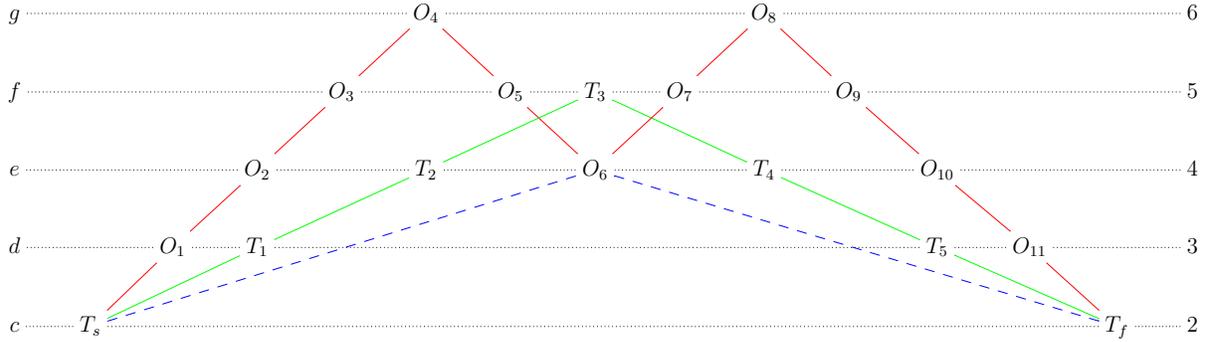

\begin{table}
\begin{tabular}{|c|c|c|}
\hline
Figure name & isoSig & move \\
\hline
$T_s$ & `cPcbbbadh' & 2\\ 
\hline
$T_1$ & `dLQacccbgjs' & 1\\
\hline
$T_2$ & `eLPkbcdddacrnn' & 4\\
\hline
$T_3$ & `fvPQccdedeeccvbfb' &  -1\\ 
\hline
$T_4$ & `eLPkbcdddackjj' &  -2\\ 
\hline
$T_5$ & `dLQacccbgbk' & -3 \\
\hline
$T_f$ & `cPcbbbadu' & `end'\\
\hline
\end{tabular}
\caption{\label{tab:ppTrefoil} A shortest Pachner path between the two 
minimal triangulations of the trefoil complement.}
\end{table}

\begin{table}
\begin{tabular}{|c|c|c|}
\hline
Figure name & isoSig & move \\
\hline
$T_s$ & `cPcbbbadh' & 1\\ 
\hline
$O_1$ & `dLQacccjgjb' & 5\\ 
\hline
$O_2$ & `eLAkbbcdddugaj' & 2\\
\hline
$O_3$ & `fLLQcccddeeabvnln'& 5\\ 
\hline
$O_4$ & `gLvQQadfeeffjatxcfj' & -5\\ 
\hline
$O_5$ & `fLAPcacceeejgjffc' & -3 \\
\hline
$O_6$ & `eLMkbbdddadiih' & 0\\
\hline
$O_7$ & `fLAPcacceeejgjcrc' & 6\\
\hline
$O_8$ & `gLvQQadfeeffjaaxcfj'& -4\\  
\hline
$O_9$ &`fLLQcccddeeabvrln'& -4\\ 
\hline
$O_{10}$ & `eLAkbbcdddurar'& -2 \\
\hline
$O_{11}$ & `dLQacccjgjs'&-3 \\
\hline
$T_f$ & `cPcbbbadu' & `end'\\
\hline
\end{tabular}
\caption{\label{tab:ppOETrefoil} A shortest one-efficient Pachner path 
between the two minimal triangulations of the trefoil complement.}
\end{table}
\end{ex}

\subsection{Angle Structures, 1-efficiency and Pachner moves}

A natural question is how are various properties of triangulations related 
by Pachner moves. Throughout this section, $\T_n$ will be a triangulation 
with $n$ tetrahedra, $f$ will be a face in this triangulation that 
identifies two distinct tetrahedra $t_{f,1}$ and $t_{f,2}$, and $\T_{n+1,f}$ 
will be the result of performing a 2-3 move along the face $f$. If we call 
the pairs of edges incident to $f$ in the bi-pyramid defined by $t_{f,1}$ 
and $t_{f,2}$, the \emph{belt of $f$}, then we can make the following 
observation that is well-known to the experts. We include it because it 
promotes a useful mentality and compares favourably to our discussion of 
1-efficiency and Pachner moves.

\begin{prop}
Let $\T_n$ be a triangulation with $n$ tetrahedra. With all notation as 
above, $\T_{n+1,f}$ admits a strict angle structure if and only if the set 
of angle structures contains a strict angle structure (resp. semi-angle 
structure) such that each of the three pairs of edges along the belt of 
$f$ have a sum in $(0,\pi)$ (resp. $[0,\pi]$).
\end{prop}

\begin{proof}
We only need to consider the tetrahedra in the bi-pyramid as the other 
angles will be unaffected by a 2-3 move. 

Considering Figure \ref{fig:twoThreeBoundary}, this problem reduces to a 
problem in the Euclidean plane: namely showing that the triangles incident 
to the vertices at the top and bottom of the bi-pyramid each split into 3 
triangles with angles in the desired range. However, the angles
around the vertex are all in $(0,\pi)$ (resp. $[0,\pi]$), so the vertex 
is embedded in the interior (resp. interior or a side). 

The other direction is obvious as an angle structure on $\T_{n+1,f}$ implies 
that all angles along the edges of the bi-pyramid are positive (resp. 
non-negative). 
\end{proof}

\subsection{Connected components of normal surfaces}

In Figure \ref{fig:twoThreeBipyramid}, there are two bipyramids related by 
a 2-3 move, one with two tetrahedra which we call $\BTwo$ and one with 
three tetrahedra which we call $\BThree$. We will also think of these as 
subsets of two triangulations related by a 2-3 move, $\T_n$ and $\T_{n+1,f}$ 
where $f$ is face identified in $\BTwo$. In this discussion, we assume that 
the exterior faces of $\BTwo$ are not identified. However, if the exterior 
faces of $\BTwo$ are identified (and therefore $\BThree$), then some of 
these pieces might not appear as subsets of embedded normal surfaces.

If $S_n$ is a normal surface of $\T_n$ then each piece will have at most 
one normal arc in face $f$. Using the notation from 
Section \ref{gluing_eqn_sect} and Figure \ref{fig:edgesAndQuad}, 
we thus obtain a census of connected components $S_n \cap \BTwo$ as follows: 

\begin{enumerate}
\item  
$t_{A0},t_{B1}$;
\item  
$(t_{A1}+t_{B1})$, $(t_{A2}+t_{B3})$, $(t_{A3}+t_{B2})$;
\item  
$(q_{A01:23}+t_{B1})$, $(q_{A02:13}+t_{B3})$, $(q_{A03:12}+t_{B2})$, 
$(q_{B01:23}+t_{A1})$, $(q_{B02:13}+t_{A3})$, and $(q_{B03:12}+t_{A2})$;
\item  
$(q_{A01:23}+q_{B01:23})$, $(q_{A02:13}+q_{B03:12})$,  $(q_{A03:12}+q_{B02:13})$.
\end{enumerate}

If $S_{n+1}$ is a normal surface of $\T_{n+1,f}$, we can obtain an analogous 
result. Namely, the observation that each connected component of 
$S_{n+1}\cap \BThree$ intersects an interior face of $\BThree$ in at most 
one normal arc, yields the following list:

\begin{enumerate}
\item  
$(t_{C0}+t_{D0}+t_{E0})$, $(t_{C1}+t_{D1}+t_{E1})$;
\item  
$(t_{C2}+t_{E3})$, $(t_{D2}+t_{C3})$,  $(t_{D3}+t_{E2})$;
\item  
$(t_{D0}+q_{C02:13}+q_{E03:12})$, $(t_{E0}+q_{D02:13}+q_{C03:12})$, 
$(t_{C0}+q_{E02:13}+q_{D03:12})$, $(q_{D01:23}+t_{C3}+t_{E2})$, 
$(t_{E1}+q_{D03:12}+q_{C02:13})$,  $(t_{C1}+q_{E03:12}+q_{D02:13})$;
\item  
$(q_{D01:23}+t_{C3}+t_{E2})$, $(q_{E01:23}+t_{D3}+t_{C2})$, 
$(q_{C01:23}+t_{E3}+t_{D2})$;
\item  
$(q_{C01:23}+q_{D01:23}+t_{E2}+t_{E3})$, $(q_{D01:23}+q_{E01:23}+t_{C2}+t_{C3})$, 
$(q_{E01:23}+q_{C01:23}+t_{D2}+t_{D3})$;  
\item 
$(q_{C01:23}+q_{D01:23}+q_{E01:23})$.
\end{enumerate}

Examples of the normal surface pieces of type 3) and 4) are given in 
Figures \ref{fig:twoThreeQuadTri} and \ref{fig:twoThreeQuadQuad}.

\begin{figure}[htpb]
\centering
\includegraphics[width=0.5\textwidth]{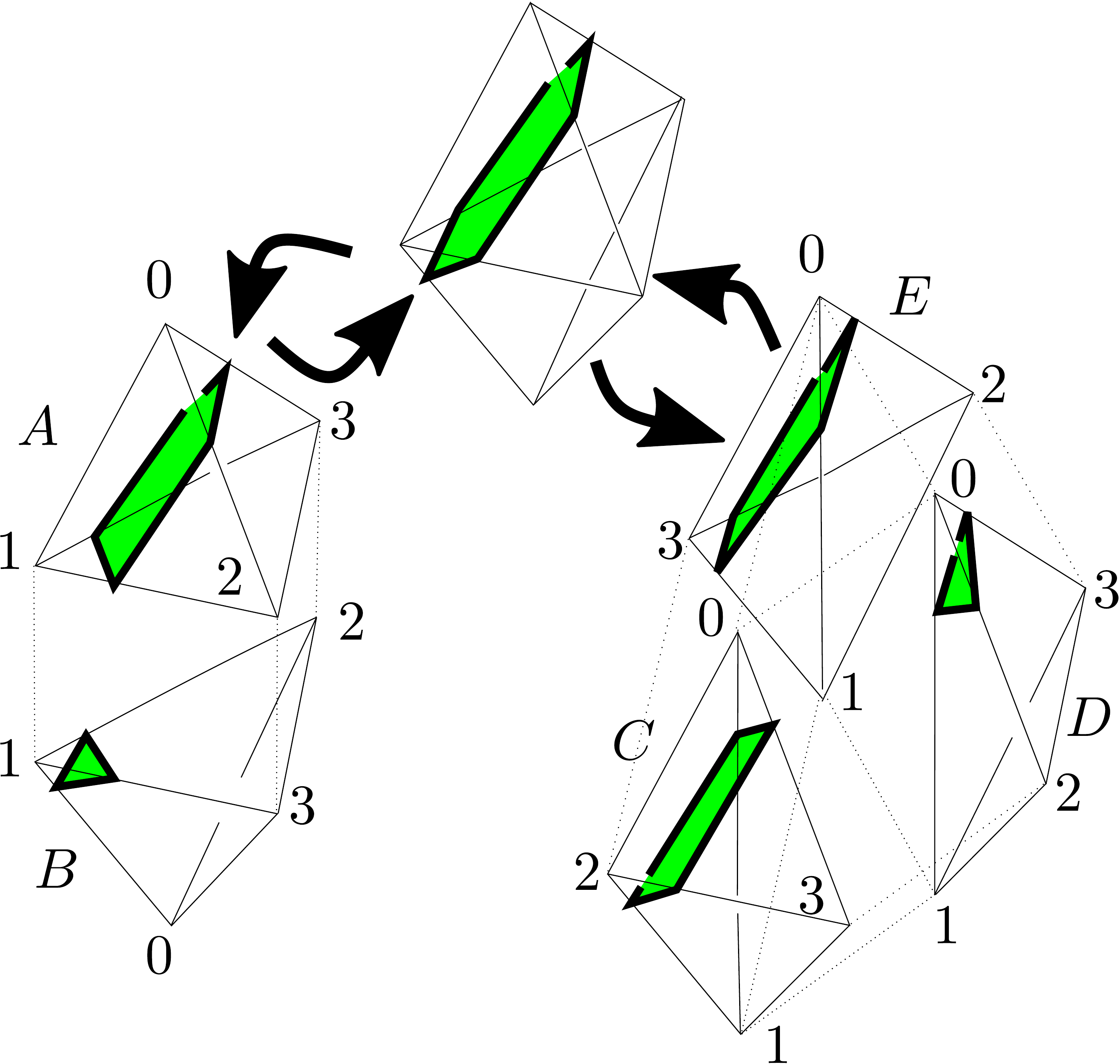}
%\resizebox{9 cm}{!}{\input{twoThreeQuadTri.eps_tex}}
\caption{\label{fig:twoThreeQuadTri} The effect of a 2-3 move on a 
normal disk comprised of a triangle and quad.}
\end{figure}

\begin{figure}[htpb]
\centering
\includegraphics[width=0.5\textwidth]{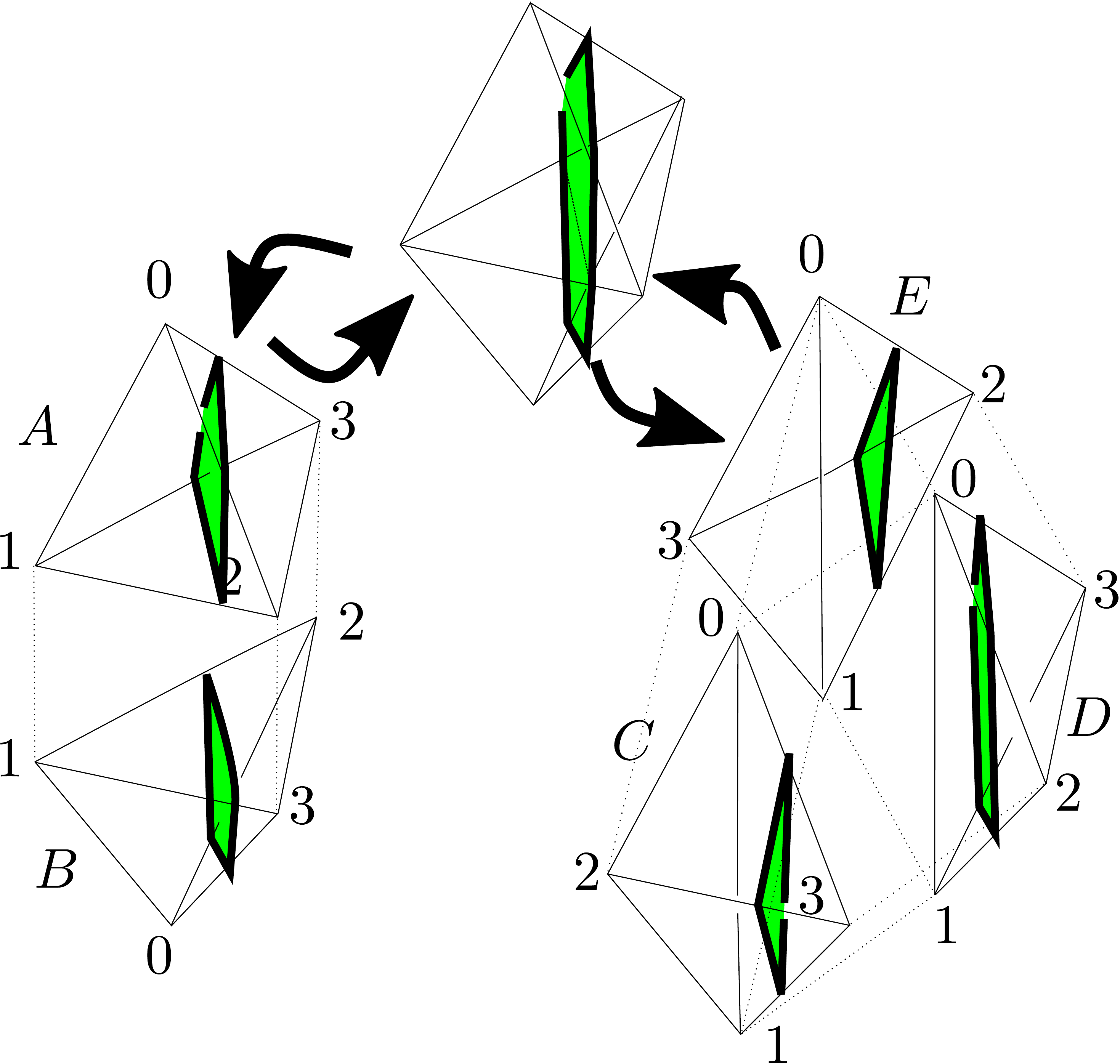}
%\resizebox{9 cm}{!}{\input{twoThreeQuadQuad.eps_tex}}
\caption{\label{fig:twoThreeQuadQuad} The effect of a 2-3 move on a 
normal disk comprised of two quads.}
\end{figure}

Examples of the normal surface pieces of type 5) and 6) are given in 
Figures \ref{fig:folder} and \ref{fig:tube}.

\begin{figure}[htpb]
\centering
\includegraphics[width=0.5\textwidth]{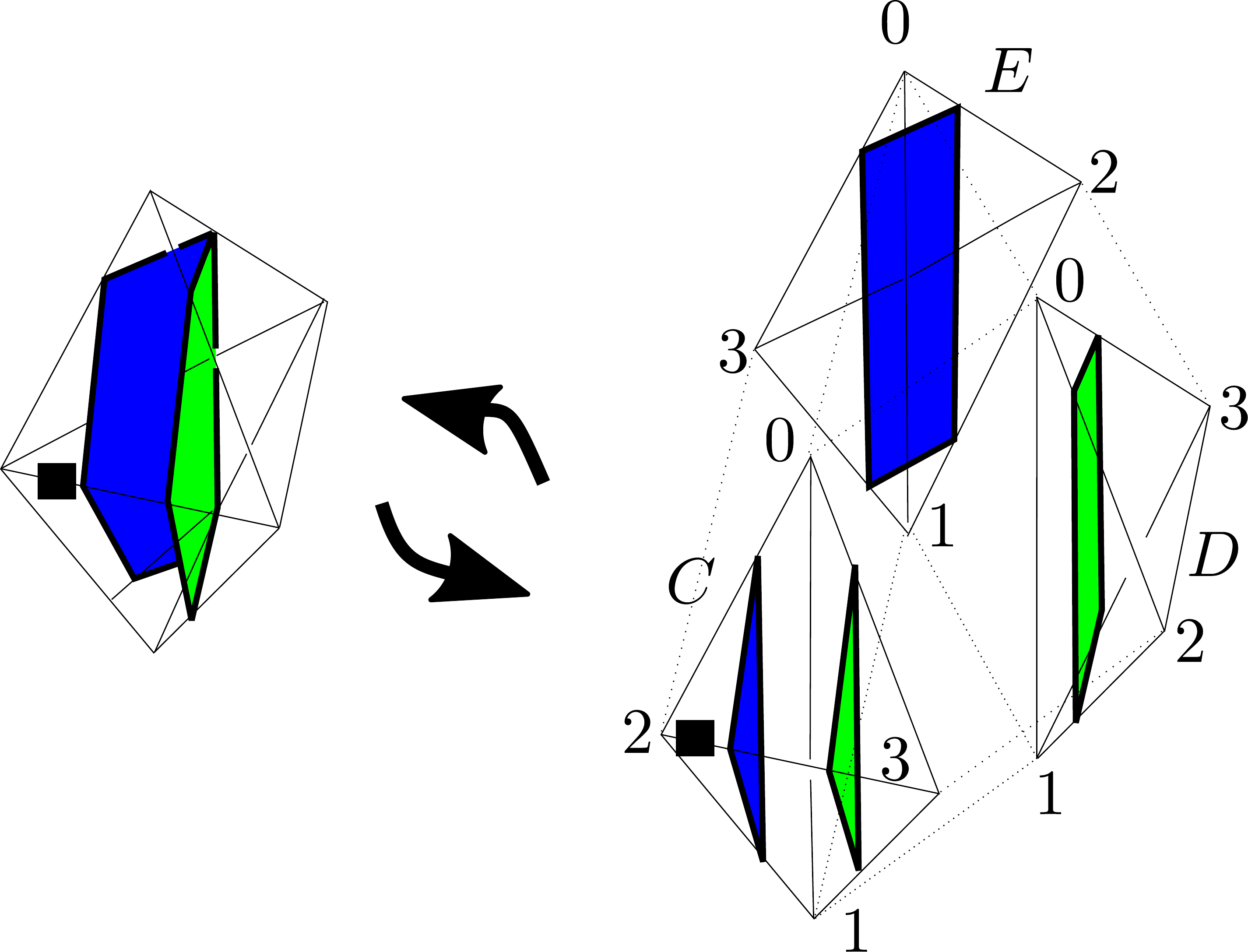}
%\resizebox{9 cm}{!}{\input{folder.eps_tex}}
\caption{\label{fig:folder} An example of a normal surface piece of $\BThree$.}
\end{figure}

\begin{figure}[htpb]
\centering
\includegraphics[width=0.5\textwidth]{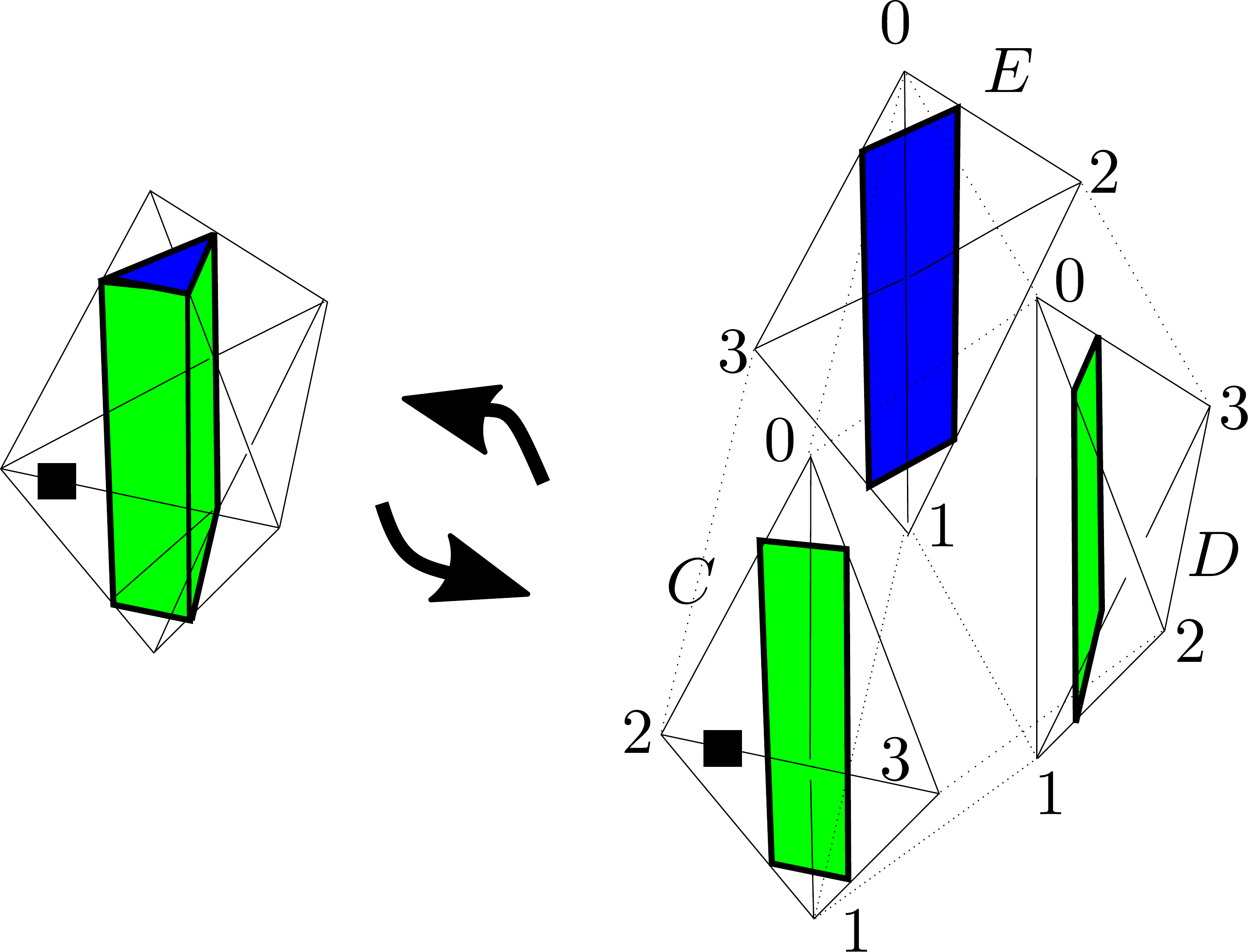}
%\resizebox{9 cm}{!}{\input{tube.eps_tex}}
\caption{\label{fig:tube} A second example of a normal surface piece of 
$\BThree$.}
\end{figure}

The following proposition allows us relate the pieces in both list above via 
the 2-3 move that transforms $\BTwo$ to $\BThree$. 

\begin{prop}
\label{cor:oneEfficientDescends}
Let $\T_n$ and $\T_{n+1,f}$ be the triangulations defined above.
For each normal surface $S_n$ of $\T_n$ there exists a normal surface 
$S_{n+1}$ of $\T_{n+1,f}$ obtained by dividing up the normal discs of $S_n$ 
which intersect $\BTwo$. In particular, if $\T_{n+1,f}$ is 1-efficient, then 
$\T_n $ is 1-efficient.
\end{prop}

\begin{proof}
Following Figure \ref{fig:twoThreeBipyramid}, we can consider the images 
of $S_n \cap \BTwo$ under the 2-3 move: 
\begin{enumerate}
\item 
$t_{A0} = t_{C0}+t_{D0}+t_{E0}$, $t_{B0} = t_{C1}+t_{D1}+t_{E1}$
\item 
$t_{A1}+t_{B1} = t_{C2}+t_{E3}$, $t_{A2}+t_{B3} = t_{D2}+t_{C3}$, 
$t_{A3}+t_{B2} = t_{E2}+t_{D3}$
\item 
$q_{A01:23} + t_{B1}= t_{D0}+q_{C02:13}+q_{E03:12}$, 
$q_{A02:13} + t_{B3}= t_{E0}+q_{D02:13}+q_{C03:12}$, 
$q_{A03:12} + t_{B2}= t_{C0}+q_{E02:13}+q_{D03:12}$, 
$q_{B01:23} + t_{A1}= t_{D1}+q_{C03:12}+q_{E02:13}$, 
$q_{B02:13} + t_{A3}= t_{E1}+q_{D03:12}+q_{C02:13}$, 
$q_{B03:12} + t_{A2}= t_{C1}+q_{E03:12}+q_{D02:13}$, 
\item 
$q_{A01:23} + q_{B01:23}= q_{D01:23}+t_{C3}+t_{E2}$, 
$q_{A02:13} + q_{B03:12}= q_{E01:23}+t_{D3}+t_{C2}$, 
$q_{A03:12} + q_{B02:13}= q_{C01:23}+t_{E3}+t_{D2}$.
\end{enumerate}

Relations (1) and (2) follow from Figure \ref{fig:twoThreeBoundary}.

Note that $q_{A01:23} + t_{B1}= t_{D0}+q_{C02:13}+q_{E03:12}$ follows from 
Figure \ref{fig:twoThreeQuadTri}, and the remaining two relations can be 
obtained by an appropriate rotation (and relabelling) of that figure.

Finally, $q_{A01:23} + q_{B01:23}= q_{D01:23}+t_{C3}+t_{E2}$ follows from 
Figure \ref{fig:twoThreeQuadQuad}, and again the remaining two relations 
can be obtained by an appropriate rotation (and relabelling) of that figure.

Thus, for any normal surface of $S_n$, we can map the normal surface 
pieces locally to build a normal surface with the same Euler 
characteristic in $\T_{n+1,f}$. Restricting this observation to the case 
of embedded normal surfaces with non-negative Euler characteristics shows 
if $\T_{n+1,f}$ is 1-efficient, then $\T_n$ must be as well.  
\end{proof}

We point out that the Luo-Tillmann's definition of Euler characteristic 
on quadrilateral discs in the presence of a generalised angle structure 
(compare  \cite[Lemma 15]{LT} in the case that the curvature is 0) provides 
a useful mnemonic for keeping track of relations of type (3) and (4). 
Namely, these relations are just the angle sums which occur after a 
2-3 move, e.g. edge $A01$ decomposes into edges $C02$ and $E03$ and 
edges $A23$ and $B23$ form $D23$.

\begin{prop}
\label{prop:oneEfficientFails}
Let $\T_n$ and $\T_{n+1,f}$ be the triangulations defined above, such that 
the exterior faces of $\BTwo$ are not identified.
If $\T_n$ is 1-efficient and $\T_{n+1,f}$ is not 1-efficient, then 
$\T_{n+1,f}$ exhibits an embedded normal surface with exactly two of the 
three quad types parallel to the degree three edge created by the 2-3 move.  
\end{prop}

\begin{proof}
If $\T_n$ is 1-efficient and $\T_{n+1,f}$ is not 1-efficient, then there 
must be an embedded normal surface $S_{n+1}$ in $\T_{n+1,f}$ which 
intersects  $\BThree$. If this intersection is only triangular discs then 
$S_{n+1}$ would descend to an embedded normal surface in $\T_n$, and so 
$S_{n+1}\cap \BThree$ must contain at least one connected component with at 
least one quad.  However, the case analysis in 
Proposition~\ref{cor:oneEfficientDescends} of the normal discs in 
$\BThree$ shows that each connected component $S_{n+1}\cap \BThree$ is in 
the image of an embedded normal disk of $\BTwo$ save four: 
$q_{C01:23}+q_{D01:23}+t_{E2}+t_{E3}$,
$q_{D01:23}+q_{E01:23}+t_{C2}+t_{C3}$, 
$q_{E01:23}+q_{C01:23}+t_{D2}+t_{D3}$,
and $(q_{C01:23}+q_{D01:23}+q_{E01:23})$.
 
However, if the last piece is part of $S_{n+1}\cap \BThree$, then there 
exists a collection of embedded normal spheres $\hat{S}_{n+1}$ of 
$\T_{n+1,f}$ formed by cutting out all copies of this annulus and capping 
off with triangles. If none of these spheres have pieces of the form:
$q_{C01:23}+q_{D01:23}+t_{E2}+t_{E3}$,
$q_{D01:23}+q_{E01:23}+t_{C2}+t_{C3}$, and 
$q_{E01:23}+q_{C01:23}+t_{D2}+t_{D3}$,
then there exists an embedded normal sphere in $\T_n$, a contradiction. 
If no annuli of the form $(q_{C01:23}+q_{D01:23}+q_{E01:23})$ exist in 
$S_{n+1}\cap \BThree$, then again there must be at least one piece of the form: 
$q_{C01:23}+q_{D01:23}+t_{E2}+t_{E3}$,
$q_{D01:23}+q_{E01:23}+t_{C2}+t_{C3}$, and 
$q_{E01:23}+q_{C01:23}+t_{D2}+t_{D3}$, which completes the proof.
\end{proof}

Although $q_{C01:23}+q_{D01:23}+t_{E2}+t_{E3}$, 
$q_{D01:23}+q_{E01:23}+t_{C2}+t_{C3}$, and $q_{E01:23}+q_{C01:23}+t_{D2}+t_{D3}$ 
are not in the image of embedded normal discs of $\T_n$, they are in the 
image of a pair of immersed normal discs. For example,
$q_{A03:12} + q_{B02:13}+q_{A01:23} + q_{B01:23}=
q_{C01:23}+q_{D01:23}+t_{E2}+t_{E3}+t_{C3}+t_{D2}$. This is the key observation 
in the proof of the following lemma.

\begin{prop}
\label{prop:strictAngToOneEff}
Let $\T_n$ and $\T_{n+1,f}$ be the triangulations defined above, such that 
the exterior faces of $\BTwo$ are not identified.
If $\T_n$ admits a strict angle structure, then $\T_{n+1,f}$ is 1-efficient. 
\end{prop}

\begin{proof}
Suppose $\T_{n+1,f}$ is not 1-efficient. Then there is a closed embedded 
normal surface $S_{n+1}$ of $\T_{n+1,f}$ with non-negative Euler characteristic. 

Just as above, $S_{n+1} \cap \BThree$ must contain pieces of the form below, 
otherwise there is a closed embedded surface $S_n$ of $\T_n$ with 
non-negative Euler characteristic:

$q_{C01:23}+q_{D01:23}+t_{E2}+t_{E3}$,
$q_{D01:23}+q_{E01:23}+t_{C2}+t_{C3}$, 
$q_{E01:23}+q_{C01:23}+t_{D2}+t_{D3}$,
and $(q_{C01:23}+q_{D01:23}+q_{E01:23})$.

By an identical argument to the one used in 
Proposition \ref{prop:oneEfficientFails}, we can rule out the annuli of 
the form  $(q_{C01:23}+q_{D01:23}+q_{E01:23})$. Thus, $S_{n+1} \cap \BThree$ 
contains exactly one of the following pairs of quads:
$q_{C01:23}+q_{D01:23}+t_{E2}+t_{E3}$,
$q_{D01:23}+q_{E01:23}+t_{C2}+t_{C3}$, or 
$q_{E01:23}+q_{C01:23}+t_{D2}+t_{D3}$.
 
Without loss of generality, assume it is $q_{C01:23}+q_{D01:23}+t_{E2}+t_{E3}$. 
Create a new surface $S'_{n+1}$, by adding a boundary torus to $S_{n+1}$ until 
$S_{n+1} \cap \BThree$ contains the same number of copies of $(t_{C3}+t_{D2})$ 
as $q_{C01:23}+q_{D01:23}+t_{E2}+t_{E3}$. Note that $\chi(S_{n+1})=\chi(S'_{n+1})$. 
The image of $S'_{n+1}$ under the 3-2 move, is an immersed normal surface 
$S_n$ with the same boundary arcs on 
$\BTwo$ as $\partial \BThree \cap S'_{n+1}$, with the property that for 
each copy of $q_{C01:23}+q_{D01:23}+t_{E2}+t_{E3} + (t_{C3}+t_{D2})$ in 
$\BThree \cap S'_{n+1}$ there is a copy of 
$q_{A03:12} + q_{B02:13}+q_{A01:23} + q_{B01:23}$. In the presence of any 
generalised angle structure on $T_n$, each quad is assigned an Euler 
characteristic, while the Euler characteristic of the triangles vanishes. 
Moreover, this Euler characteristic agrees with the Euler characteristic 
of $S'_{n+1}$. However, if $\T_n$ admits a strict angle structure, then 
each quad has a negative contribution to Euler characteristic while 
each triangle does not contribute to Euler characteristic. This contradicts 
$\chi(S_{n+1})\geq 0.$
\end{proof}

\begin{rmk}
The analogous statement to Proposition \ref{prop:strictAngToOneEff} for 
0-2 moves does not hold. 
In fact, `eLAkbbcdddhjac' has a degree 2 edge and is not 1-efficient. 
However, performing a 2-0 move along this edge results 
in the  1-efficient (in fact, geometric) triangulation of the figure 8 
sister manifold, `cPcbbbdxm'.
\end{rmk}

\begin{ex}
We next provide an example of how a geometric triangulation can `degenerate' 
to a triangulation which is not 1-efficient in two 2-3 moves:  
The triangulation `eLAkbccddhhnqw' is geometric, the 2-3 move along face 
5 gives `fLAMcbccdeemejman' which does not admit a strict angle structure, 
and the 2-3 move along face 8 gives `gLALQaccefffbgfgmqt' which is not 
1-efficient. 

However, perhaps the simplest example of a non 1-efficient triangulation 
arising from a 2-3 move on a 1-efficient triangulation has already 
been discussed in Table \ref{tab:ppTrefoil}. In fact the first Pachner move 
in that path a 2-3 move along face 2 of `cPcbbbadh' resulting 
in `dLQacccbgjs' breaks 1-efficiency. 
We again point out that there is an immersed normal surface in `cPcbbbadh' 
determined by the quads $q_{0:02:13}$, $q_{0:03:12}$, $q_{1:02:13}$, 
and $q_{1:02:13}$ (face 2 is $0(023)=1(023)$), which maps to an embedded normal 
torus in `dLQacccbgjs'. (Note: Regina's labelling routine is different 
from that of the lists in this section.)
\end{ex}

%%%%%%%%%%%%%%%%%%%%%%%%%%%%%%%%%%%%%%%%%%%%%%%%%%%%%%%%%%%%%%%%%%%%%%%%%%%
%%%%%%%%%%%%%%%%%%%%%%%%%%%%%%%%%%%%%%%%%%%%%%%%%%%%%%%%%%%%%%%%%%%%%%%%%%%

\appendix

\section{Generalised angle structures and the Euler characteristic}
\label{gen_angle_euler}

The relationship between angle structures and geometric structures is often 
described by saying that a generalised angle structure is a solution to the 
``linear part'' of Thurston's gluing equations (see \cite[\sect 4]{Th}). 
However, often this simplification involves only considering the edge 
equations, as in \cite{LT}. 
For the discussion below, we will pay special attention to the 
completeness conditions coming from the 
holonomies of the peripheral curves. 

Let $\T$ be an ideal triangulation with $n$ tetrahedra of a compact, 
orientable manifold 3-manifold $M$ with boundary consisting of $r$ tori, 
and let $\quads$ denote the set of all quad types in $\T$.

Recall, from Section \ref{angle_struct_sect}, that a 
{\em generalised angle structure} on $\T$ is a function 
$\alpha: \quads \to \R$ satisfying the equations (\ref{ang_struct1}) and 
(\ref{ang_struct2}), which represent the imaginary part of Thurston's 
logarithmic tetrahedral and edge equations (\ref{log_tet_eqns}), 
(\ref{log_edge_eqns}).

As described in \cite[\sect 4.3.2]{Th}, the logarithm of the derivative 
of the holonomy $H'(\gamma)$ for each peripheral curve $\gamma$ must also 
vanish for a {\em complete} hyperbolic structure. Writing down these 
conditions for a pair of simple closed peripheral curves generating 
the fundamental group of each boundary component gives the logarithmic 
cusp equations (\ref{log_cusp_eqns}).

The analogous condition for  angle structures is defined as follows.

\begin{defn}  
Given a generalised angle structure $\alpha$ on an ideal triangulation $\T$,
the {\em rotational holonomy} of a peripheral curve $\gamma$ is the 
imaginary part of the logarithm of the derivative of its holonomy:
$\rho_\alpha(\gamma) = \Im \log H'(\gamma)$.

If a generalised angle structure satisfies $\rho_\alpha(\gamma)=0$ for 
each peripheral curve $\gamma$, we say it has 
\emph{vanishing peripheral rotational holonomy}. 
\end{defn}

More concretely,  this means the vector 
$x =[ \alpha(q_j) ~ \alpha(q_j')~ \alpha(q_j'') ] \in \R^{3n}$ 
satisfies the linear equations given by the imaginary parts  of Thurston's 
logarithmic edge, tetrahedral and completeness equations 
(\ref{log_edge_eqns}),  (\ref{log_tet_eqns}), (\ref{log_cusp_eqns}).

In the next section, we will discuss how to assign Euler characteristics 
to quad disks in the presence of a generalised angle structure. 

\begin{prop}
\label{gen_ang_zero_periph_rot_hol}
Let $\T$ be an ideal triangulation of a compact manifold $M$ with boundary 
consisting of tori. Then there exist generalised angle structures on $\T$ with 
vanishing peripheral rotational holonomy.
\end{prop}

\begin{proof} 
We closely follow the argument for Lemma 10 in \cite{LT}. %[Luo-Tillman].
Assume that $\T$ has $n$ tetrahedra and $\bd M$ consists of $r$ tori. 
A vector $x \in \R^{3n}$ defines a generalised angle structure with vanishing 
peripheral rotational holonomy if and only if 
$$
Ax=b,
$$ 
where
$A$ is the $(2n+2r) \times 3n$ matrix with rows $E_i, T_j, M_k, L_k$
and $b \in \R^{2n+2r}$ is the vector with $n$ entries $2\pi$, followed by $n$ 
entries $\pi$, then $2r$ zero entries.

From linear algebra, the equation $Ax=b$ has a solution if and only if 
$b \in \Im(A)=\Ker(A^T)^\perp$
if and only if 
$$
z \in \R^{2n+2r} \text{ and } A^Tz=0 \Rightarrow z^T b = 0.
$$
Now
$z= \begin{bmatrix} x_i & y_j &p_k &q_k\end{bmatrix}^T \in \Ker(A^T)$
if and only if
$S= \sum_i x_i E_i + \sum_j y_j T_j + \sum_k p_k M_k + \sum_k q_k L_k =0$,
i.e. $S$ gives the trivial normal class in $Q(\T;\R)$. But this implies that 
$\bd S=0$ so $p_k=q_k=0$ for all $k$, hence 
$S= \sum_i x_i E_i + \sum_j y_j T_j=0$.
Further $-\chi(S)=\sum_i 2 x_i + \sum_j  y_j =0$, hence  $z^T b = 0$.  
This proves the result.
\end{proof}

\subsection{Euler characteristic from generalised angle structures}

The following result shows how to compute the Euler characteristic of an 
embedded spun normal surface using any generalised angle structure. 
This result is well-known for closed normal surfaces; see for 
example \cite[Lemma 15]{LT}.

\begin{prop}
\label{prop_euler_from_ang_str}
Let $\T$ be an ideal triangulation of a cusped 3-manifold $M$, and 
let $\alpha : \quads \to \R$ be a generalised angle structure on $\T$. 
Then the Euler characteristic of any embedded spun normal surface 
$S \in Q(\T;\R)$ is given by
\be 
\label{euler_from_gen_ang_struct}
\chi(S) = 
- \sum_q \frac{\alpha(q) x_q}{\pi} + \frac{ \rho_\alpha(\bd S)}{2\pi},
\ee
where the sum is over all normal classes $q$ in $\T$, $x_q$ is the normal 
coordinate of $q$, $\alpha(q) \in \R$ denotes the angle assigned to the 
two edges facing $q$, and $\rho_\alpha(\bd S)$ denotes the sum of the 
rotational holonomies of the boundary components of $S$, 
oriented as in $[\bd S]$.
\end{prop}

This result follows from a combinatorial version of the Gauss-Bonnet 
theorem, as used by Thurston in  \cite[Section 13.7]{Th}.
Let $S$ be a compact surface, possibly with boundary and let $\mathcal C$ 
be a cell decomposition of  $S$ into finitely many polygons.
Then a {\em formal angle structure} (or {\em combinatorial angle structure})
on $\mathcal C$ is an assignment of real numbers (``angles'')  to the 
corners of the polygons such that the sum of angles is $2\pi$ at each 
internal vertex.  

We compute the Euler characteristic of $S$ as 
$\chi(S) = v-e+f$, where $v,e,f$ are the numbers of vertices, edges and 
faces in the cell decomposition.
Now each $n$-gon $P$ contributes $\frac{1}{2\pi} 
\text{(sum of corner angles)}$ to $v$, 
$n \times \frac{1}{2}=n/2$ to $e$ and $1$ to $f$,
so contributes a {\em local Euler characteristic} 
\be
\label{chi_polygon}
\chi_P = \frac{1}{2\pi} \text{(sum of corner angles in $P$)} - \frac{n}{2}+1
\ee
to $\chi(S) = v-e+f$.

Adding up the terms $\chi_P$ over all polygons $P$ gives a contribution to 
$\chi(S)$ of $+1$ for each internal vertex, $-1$ for each internal edge
and $+1$ for each face.   Assume that there are also $k$ boundary vertices 
and, hence, $k$ boundary edges. Then to obtain $\chi(S)$ we need to add 
additional contributions of $\frac{1}{2\pi}(2\pi-\theta_i)$ where 
$\theta_i$ is the sum of internal angles at the $i$th boundary vertex 
and $-1/2$ for each boundary edge. This  gives
$$
\chi(S)= \sum_P \chi_P +  \frac{1}{2\pi} 
\sum_{i=1}^k (2\pi-\theta_i) - \frac{k}{2} =  \sum_P \chi_P + \frac{1}{2\pi} 
\sum_{i=1}^k (\pi-\theta_i),
$$
and last term 
\be 
\label{geod_curb_bd} 
k_g(\bd S) = \sum_{i=1}^k (\pi-\theta_i)
\ee
represents the {\em total geodesic curvature} of the boundary of $S$. Hence

\begin{prop}[Combinatorial Gauss-Bonnet]
\label{comb_Gauss-Bonnet}
Let $S$ be a compact surface with a finite cell decomposition 
$\mathcal C$ equipped with a formal angle structure. 
Then the Euler characteristic of $S$ is given by
\be
\chi(S) = \sum_P \chi_P + \frac{k_g(\bd S)}{2\pi},
\ee
summed over all polygons $P$, where $\chi_P$ is defined by 
(\ref{chi_polygon}) and where $k_g(\bd S)$ denotes the total geodesic 
curvature of $\bd S$. 
\end{prop}

\begin{proof}[Proof of Proposition \ref{prop_euler_from_ang_str}]
Let $S$ be a (non-compact) embedded spun normal surface in $\T$ and choose 
a compact subsurface $S_0$ consisting of all quads and a finite number 
of normal triangles such that $S \setminus S_0$ is a collection of annuli 
spiralling out to the cusps of $M$. By adding an additional layer of 
triangles, if necessary,  we can assume that there is a neighbourhood of 
$\bd S_0$ consisting entirely of triangles. Projecting these normal 
triangles to a horosphere gives a layer of Euclidean triangles to the 
left of the boundary, and the rotational holonomy of $\bd S_0$ is the sum 
$$
\rho_\alpha(\bd S) = \rho_\alpha(\bd S_0) = \sum_{i=1}^k (\pi-\theta_i),
$$
where $\theta_i$ is the sum of triangle angles at the $i$th vertex on the 
boundary. 

Now $S_0$ has a cell decomposition into quads and triangles, and a 
generalised angle structure $\alpha$ on $\T$ assigns an ``angle'' (in $\R$) 
to each corner of each 2-cell in this decomposition, giving a formal angle 
structure on this cell decomposition. Now each 2-cell gives a contribution 
to the Euler characteristic as follows:
\begin{itemize}
\item
Each quad $q$ contributes 
$\chi(q) = \frac{1}{2\pi}\text{(angle sum)} - 1$. 
But the sum of angles in the quad is $2\pi$ minus twice the angle 
$\alpha(q)$ on the two edges facing $q$,
so we have $\chi(q) = - \alpha(q)/\pi$.
\item 
Each triangle $t$ 
contributes $\chi(t) = \frac{1}{2\pi}\text{(angle sum)} 2-3/2+1= 0$.  
\end{itemize}

Hence the combinatorial Gauss-Bonnet theorem gives
$$
\chi(S) =\chi(S_0) 
 = 
-\sum_q \frac{\alpha(q)x_q}{\pi}  + \frac{1}{2\pi} \sum_{i=1}^k (\pi-\theta_i) \\
= -\sum_q \frac{\alpha(q)x_q}{\pi} + \frac{\rho_\alpha(\bd S)}{2\pi}.
$$
\end{proof}

Since the boundary term vanishes for any angle structure with vanishing 
peripheral rotational holonomy, this gives the following
convenient way of computing the Euler characteristic:

\begin{cor} 
Let $\T$ be an ideal triangulation of a cusped 3-manifold $M$, and 
let $\alpha : \quads \to \R$ be a generalised angle structure on $\T$ with 
vanishing peripheral rotational holonomy. Then the Euler characteristic 
of any embedded spun normal surface $S \in Q(\T;\R)$ is given by
\be 
\chi(S) = - \sum_q \frac{\alpha(q)x_q}{\pi} .
\ee
\end{cor}

\begin{defn}
\label{formal_euler_from_ang_str}
For a general Q-normal class $S \in Q(\T;\R)$ we take the formula 
(\ref{euler_from_gen_ang_struct}) as the definition of the 
{\em formal Euler characteristic} of $S$. This is a linear function 
$\chi: Q(\T;\R) \to \R$.
\end{defn}

\begin{rmk}
The independence of choice of angle structure follows directly from 
\cite[Lemma 15]{LT}. For the angle structures in this paper, 
$(A,\kappa) =(0,0)$ and our $\alpha(q)$ is equal to one half of $ A(q)$ 
as in \cite{LT}. 
\end{rmk}

Another consequence of these results is 

\begin{cor}
Each peripheral curve solution has vanishing formal Euler characteristic.
Hence 
\be
\label{euler_from_basic_solutions}
\chi\left(\sum_i x_i E_i + \sum_i y_j T_j 
+ \sum_k (p_k M_k  + q_k L_k) \right) = - \sum_i 2 x_i - \sum_j y_j,
\ee
since $\chi(E_i)=-2$ and $\chi(T_j)=-1$.
\end{cor}

\begin{proof} 
Let $\alpha$ be a generalised angle structure on $\T$, and
let $S$ be the  peripheral curve solution given by a curve with homology 
class $c \in H_1(\bd M;\Z)$.
Then we have $\bd S = 2 c$ (see Section \ref{subsect:gen_q_normal}).
and, by definition, $\rho_\alpha(c)= \sum_{q} \alpha(q) x_q$ where $x_q$ are 
the quad coordinates of $S$.
Hence
$$
\chi(S) = - \sum_{q} \frac{\alpha(q) x_q}{\pi} + \frac{ \rho_\alpha(\bd S)}{2\pi} 
= -  \frac{\rho_\alpha(c)}{\pi} + \frac{ \rho_\alpha(\bd S)}{2\pi} =0.
$$
\end{proof}

\begin{rmk}
We could also use  the formula (\ref{euler_from_basic_solutions}) to define 
the formal Euler characteristic of a general $Q$-normal class $S \in Q(\T;\R)$.
\end{rmk}

\begin{ex} 
For the figure eight knot complement with its triangulation given in 
Example \ref{ex:fig8}, the angle structure equations have general 
solution of the form 
$$
\alpha= (\alpha_0+t_1 \alpha_1 + t_2 \alpha_2 + t_3 \alpha_3)\pi, 
\text{ with } t_i \in \R
$$
where $\alpha_0=(-1, 2, 0, 1, 0, 0), \alpha_1=(2, -2, 0, -1, 0, 1), 
\alpha_2=(1, -1, 0, -1, 1, 0),\alpha_3=(1, -2, 1, 0, 0, 0)$.
This gives rotational holonomies for the meridian and longitude
$$
\rho_\alpha(M)= (-1 + t_1 + t_2 + t_3)\pi  \text{ and }  \rho_\alpha(L)
= (2 - 4 t_1 - 2 t_2)\pi.
$$

Now there is an embedded spun normal once punctured Klein bottle with quad 
coordinates 
$K=(0,0,2,0,1,0)$ with boundary $\bd K = 4 \mu + 1 \lambda$.  
The general formula (\ref{euler_from_gen_ang_struct}) for the Euler 
characteristic gives
$$
\chi(K) = - \sum_q \frac{\alpha(q)}{\pi} + \frac{ \rho_\alpha(\bd S)}{2\pi}
=-\frac{1}{\pi} \alpha\cdot K + \frac{1}{2\pi} \alpha \cdot(4M+L) 
= (-t_2-2t_3)+(-1+t_2+2t_3) = -1
$$
as expected.

Alternatively we can write $K= 2M + \frac{1}{2} L  + 1 T_2$, hence 
$\chi(K)=-1$ by formula (\ref{euler_from_basic_solutions}).
\end{ex}

In fact, finding angle structures can be seen as a dual problem to the 
existence of surfaces of non-negative Euler characteristic. To make this 
precise, we first remind the reader of a key tool, Farkas' Lemma. 

\begin{thm}[Farkas' Lemma]
If $A$ is a real $m\times n$ matrix and $b \in \R^m$, and $\cdot$ the 
standard Euclidean inner product on $\R^m$, then the following holds:

$\{x \in \R^n \mid Ax=b, x>0 \}$ is non-empty if and only if for all 
$z\in \R^m$ such that $A^Tz \neq 0$ and $A^Tz\geq0$, $z\cdot b>0$.

Here $x \geq 0$ (respectively $x > 0$) indicates that all coordinates 
of $x$ are non-negative (respectively positive). 
\end{thm}

By using Farkas' Lemma, we can mimic the above argument for 
Theorem \ref{gen_ang_zero_periph_rot_hol} to get necessary and 
sufficient conditions for the existence of strict angle structures with 
vanishing peripheral rotational holonomy. 

\begin{prop}
Let $\T$ be an ideal triangulation of a compact manifold $M$ with 
boundary consisting of tori. Then there exist a strict angle structure 
on $\T$ with vanishing peripheral rotational holonomy if and only if 
there exists no Q-normal class $0 \ne Q \in Q(\T;\R_+)$ with $\chi(S)\ge 0$ 
if and only if $0 \ne Q \in Q(\T;\R_+)$ implies $\chi(S)< 0$.
\end{prop}

\begin{proof}
We closely follow the argument of Theorem \ref{gen_ang_zero_periph_rot_hol} 
(and Lemma 10 in \cite{LT}). Using the notation from the proof of 
Theorem \ref{gen_ang_zero_periph_rot_hol},  
a vector $x \in \R^{3n}$ defines a strict angle structure with vanishing 
peripheral rotational holonomy if and only if 
$$Ax=b \text{ and } x>0.$$ 
By Farkas' Lemma, this is equivalent to the condition that
\begin{center}
for all $z\in \R^{2n+2r}$ such that $A^Tz \neq 0$ and $A^Tz\geq0$, $z\cdot b>0$. 
\end{center}
Now if
$z= \begin{bmatrix} x_i & y_j &p_k &q_k\end{bmatrix}$ 
then $S=A^Tz = \sum_i x_i E_i + \sum_j y_j T_j + \sum_k p_k M_k 
+ \sum_k q_k L_k$ represents a Q-normal class in  $Q(\T;\R)$. 
But $A^Tz \neq 0$ and $A^Tz\geq0$ iff $S \ne 0$ and $S$ has all quad 
coordinates $\ge 0$, i.e. $S \in Q(\T;\R_+)$.
Further, $z\cdot b = \sum_i 2 x_i + \sum_j  y_j = -\chi(S)$. So the 
result follows.
\end{proof}

%%%%%%%%%%%%%%%%%%%%%%%%%%%%%%%%%%%%%%%%%%%%%%%%%%%%%%%%%%%%%%%%%%%%%%%%%%%%
%%%%%%%%%%%%%%%%%%%%%%%%%%%%%%%%%%%%%%%%%%%%%%%%%%%%%%%%%%%%%%%%%%%%%%%%%%%%

\bibliographystyle{hamsalpha}
\bibliography{biblio}
\end{document}